\documentclass[12pt]{amsart}
\usepackage{amsmath,amssymb,amsfonts,latexsym,amscd,psfrag,mathabx,graphicx,mathrsfs,stmaryrd}
\usepackage[square,sort,comma,numbers]{natbib}
\usepackage{tikz}
\usepackage{verbatim}
\usepackage{tikz}
\usepackage[bottom]{footmisc}
\usetikzlibrary{shapes,shadows,calc}
\usepackage{footmisc}
\usepackage{epsfig}
\usepackage{float}
\usepackage{listings}
\usepackage{algorithm}
\usepackage{algorithmic}

\usepgflibrary{arrows}
\usetikzlibrary{arrows, decorations.markings, calc, fadings, decorations.pathreplacing, , decorations.shapes, patterns, decorations.pathmorphing, positioning}
\tikzset{nodde/.style={circle,draw=blue!50,fill=pink!80,inner sep=4.2pt}}
\tikzset{noddee/.style={circle,draw=black,fill=black,inner sep=2pt}}
\newcommand*{\equal}{=}
\newcommand{\lebn}

\theoremstyle{plain}
\newtheorem{proposition}[equation]{Proposition}
\newtheorem{theorem}[equation]{Theorem}

\newtheorem{corollary}[equation]{Corollary}
\newtheorem{lemma}[equation]{Lemma}

\theoremstyle{definition}
\newtheorem{problem}[equation]{Problem}
\newtheorem{definition}[equation]{Definition}
\newtheorem{remark}[equation]{Remark}
\newtheorem{example}[equation]{Example}
 
\numberwithin{equation}{section}

\newcommand{\Z}{\mathbb{Z}}

\DeclareMathOperator{\NP}{NP}
\newcommand{\PD}{\mathcal{P}\mathcal{D}}

\newcommand{\PF}{\mathcal{P}\mathcal{F}}

\newcommand{\BW}{\mathscr{BW}}

\newcommand{\ID}{\mathcal{I}\mathcal{D}}

\newcommand{\ME}{\mathcal{M}}

\newcommand{\DE}{\mathcal{D}}
\newcommand{\FE}{\mathcal{F}}
\newcommand{\RE}{\mathcal{R}}

\newcommand{\GE}{\mathcal{G}}
\newcommand{\LE}{\mathcal{L}}

\newcommand{\IE}{\operatorname{Ind}}
\newcommand{\HE}{\mathcal{H}}

\newcommand{\TE}{\mathcal{T}}

\newcommand{\kk}{\Bbbk}
\newcommand{\gp}{\mathfrak{g}}
\newcommand{\fp}{\mathfrak{f}}
\newcommand{\tp}{\mathfrak{t}}

\newcommand{\af}{\mathfrak{a}}

\newcommand{\Ind}{\operatorname{Ind}}
\newcommand{\Del}{\operatorname{Del}}
\newcommand{\Add}{\operatorname{Add}}
\newcommand{\lk}{\operatorname{lk}}
\newcommand{\del}{\operatorname{del}}

\newcommand{\cd}{\operatorname{cochord}}
\newcommand{\reg}{\operatorname{reg}}

\newcommand{\im}{\operatorname{im}}
\newcommand{\vim}{\operatorname{vim}}

\newcommand{\m}{\operatorname{m}}

\newcommand{\ol}{\overline}
\newcommand{\D}{\Delta}

\begin{document}

\bibliographystyle{plain}

\title[Regularity of graphs]{Castelnuovo-Mumford regularity of graphs}
\author{T\" urker B\i y\i ko$\breve{g}$lu and Yusuf Civan}

\address{}

\address{Department of Mathematics, Suleyman Demirel University,
Isparta, 32260, Turkey.}

\email{tbiyikoglu@gmail.com\\
yusufcivan@sdu.edu.tr}

\keywords{ Castelnuovo-Mumford regularity, prime graph, Cohen-Macaulay graph, matching number, induced matching number.}

\date{\today}

\thanks{Both authors are supported by T\" UB\. ITAK (grant no: 111T704), and the first author is
also supported by ESF EUROCORES T\" UB\. ITAK (grant no: 210T173)}

\subjclass[2010]{13F55, 05E40, 05C70, 05C75, 05C76.}

\begin{abstract}
We present new combinatorial insights into the calculation of (Castelnuovo-Mumford) regularity of graphs. 
We first show that the regularity of any graph can be reformulated as a generalized induced matching problem. On that direction, we introduce the notion of a \emph{prime graph} by calling a connected graph $G$ as a prime graph over a field $\kk$, if $\reg_{\kk}(G-x)<\reg_{\kk}(G)$ for any vertex $x\in V(G)$.  We exhibit some structural properties of prime graphs that enables us to compute the regularity in specific hereditary graph classes. In particular, we prove that $\reg(G)\leq \Delta(G)\im(G)$ holds for any graph $G$, and in the case of claw-free graphs, we verify that this bound can be strengthened by showing that $\reg(G)\leq 2\im(G)$, where $\im(G)$ is the induced matching number of $G$. By analysing the effect of Lozin transformations on graphs, we narrow the search of prime graphs into bipartite graphs having sufficiently large girth with maximum degree at most three, and show that the regularity
of bipartite graphs $G$ with such constraints is bounded above by $2\im(G)+1$. Moreover, we prove that any non-trivial Lozin operation preserves the primeness of a graph that enables us to generate new prime graphs from the existing ones.

We introduce a new graph invariant, the \emph{virtual induced matching number} $\vim(G)$ satisfying $\im(G)\leq \vim(G)\leq \reg(G)$ for any graph $G$ that results from the effect of edge-contractions and vertex-expansions both on graphs and the independence complexes of graphs to the regularity. In particular, we verify the equality $\reg(G)=\vim(G)$ for a graph class containing all Cohen-Macaulay graphs of girth at least five.

Finally, we prove that there exist graphs satisfying $\reg(G)=n$ and $\im(G)=k$ for any two integers $n\geq k\geq 1$. The proof is based on a result of Januszkiewicz and Swiatkowski~\cite{JS} accompanied with Lozin operations. We provide an upper bound on the regularity of any $2K_2$-free graph $G$ in terms of the maximum privacy degree of $G$. In addition, if $G$ is a prime $2K_2$-free graph, we show that $\reg(G)\leq \frac{\delta(G)+3}{2}$.

\end{abstract} 

\maketitle

\section{Introduction}\label{sect:intro}
Computing or finding applicable bounds on the graded Betti numbers, the regularity and the projective dimension of a monomial ideal $I$ in the polynomial ring $R=\kk[V]$ with a finite set $V$ of indeterminates over a field $\kk$
has been a central problem in combinatorial commutative algebra~\cite{Ha}. In general, very little is known even in the case where the
ideal $I$ is squarefree. For squarefree monomial ideals, the main approach to tackle such a problem is combinatorial, that is,
the instruments of (hyper)graph theory are the main tools when finding the exact value or establishing
tight bounds on these invariants. A common way to associate a combinatorial object to a squarefree monomial ideal $I$ goes by the language of Stanley-Reisner theory.  
In detail, let $I$ be minimally generated by squarefree monomials $m_1,\ldots,m_r$, where $V=\{x_1,\ldots, x_n\}$. If we define $F_i:=\{x_j\colon x_j|m_i\}$
for any $1\leq i\leq r$, then the family $\D:=\{K\subseteq V\colon F_i\nsubseteq K\;\text{for any}\;i\in [r]\}$ is a simplicial complex on $V$ such that its minimal non-faces exactly correspond to the generators of $I$. Under such an association, the ideal $I$ is said to be the \emph{Stanley-Reisner ideal} of the simplicial complex
$\D$, denoted by $I=I_{\D}$, and the simplicial complex $\D=\D_I$ is called the \emph{Stanley-Reisner complex} of the ideal $I$. In particular, when each monomial $m_i$ is quadratic, then the pair $G_I:=(V,\{F_1,\ldots,F_r\})$ is a (simple) graph on the set $V$, and under such a correspondence, the ideal $I$ is called the \emph{edge ideal} of the graph $G_I$.

These interrelations can be reversible in the following way.
Let $G=(V,E)$ be a graph on $V$. Then the family of subsets of $V$ containing no edges of $G$ forms a simplicial complex $\IE(G)$ on $V$, the \emph{independence complex} of $G$, and the corresponding Stanley-Reisner ideal $I_{G}:=I_{\IE(G)}$ is exactly the edge ideal of $G$. 

The resulting one-to-one correspondence among squarefree monomial ideals in $R$ and simplicial complexes on the set $V$ enable us to state the definitions of the invariants in question by the well-known Hochster's formula. For instance, if $\D$ is a simplicial complex on $V$, then 
\begin{equation*}
\reg(\D):=\reg(R/I_{\D})=\max \{j\colon \widetilde{H}_{j-1}(\D[S];\kk)\neq 0\;\textrm{for\;some}\;S\subseteq V\},
\end{equation*}
where $\D[S]:=\{F\in \D\colon F\subseteq S\}$ is the induced subcomplex of $\D$ by $S$, and $\widetilde{H}_{*}(-;\kk)$ denotes
the (reduced) singular homology. Since $\reg(I)=\reg(R/I)+1$ for any ideal $I$ in $R$, we have $\reg(I_{\D})=\reg(\D)+1$. In a similar vein, we conclude that $\reg(G):=\reg(\IE(G))=\reg(I_{G})-1$ for any graph $G$.

Most of the recent work on the graph's regularity has been devoted to the existence of tight bounds on the regularity via other graph parameters, and most likely candidate is the induced matching number. By a theorem of Katzman~\cite{MK},
it is already known that the induced matching number $\im(G)$ provides a lower bound for the regularity of a graph $G$, and the characterization of graphs in which
the regularity equals to the induced matching number has been the subject of many recent papers~\cite{BC,HVT,MMCRTY,MV,VT,RW2}. 
Observe that for a graph $G$, if we have 
$\reg(G)=\im(G)=n$ for some $n\geq 1$, then $G$ contains an induced copy of the graph $nK_2$, and $n$ is the greatest integer with this property. On the other hand, the graph $nK_2$ has the minimal order among all graphs satisfying the equality $\reg(G)=\im(G)=n$. Such an observation brings the idea of decomposing a graph into its induced subgraphs for which each subgraph in the decomposition has the minimal order with respect to its regularity. 
In other words, we call a connected graph $H$ as a \emph{prime graph} over a field $\kk$ if $\reg_{\kk}(H-x)<\reg_{\kk}(H)$ for any vertex $x\in V(H)$. Now, we say that a family $\{H_1,\ldots,H_r\}$ of vertex disjoint induced subgraphs of $G$ is a \emph{prime decomposition} of $G$ over $\kk$, if each graph $H_i$ is a prime graph over $\kk$ and the induced subgraph of $G$ on $\cup_{i=1}^r V(H_i)$ contains no edge of $G$ that is not contained in any of $E(H_i)$, and denote by $\PD_{\kk}(G)$, the set of prime decompositions of $G$ over $\kk$.   

\begin{theorem}\label{thm:reg-dec}
For any graph $G$ and any field $\kk$, we have $$\reg_{\kk}(G)=\max \{\sum_{i=1}^{r}\reg_{\kk}(H_i)\colon \{H_1,\ldots,H_r\}\in \PD_{\kk}(G)\}.$$
\end{theorem}

In its greatest generality, the notion of primeness can also be expressed in the language of simplicial complexes, namely that a connected simplicial complex $\D$ is called a \emph{prime complex} (over $\kk$) whenever $\reg_{\kk}(\del_{\D}(x))<\reg_{\kk}(\D)$ for any vertex $x$ of $\D$, where $\del_{\D}(x)$ is the subcomplex  obtained from $\D$ by removing $x$ from every face containing it. Observe that these two definitions coincide when $\D=\Ind(G)$ for some graph $G$.

As the regularity is dependent on the characteristic of the coefficient field, so is the notion of primeness. We verify by a computer calculation that the graph
provided by Morey and Villareal (Example 3.6 in~\cite{MV}) is a prime graph over $\Z_2$, while it is not prime with respect to $\Z_3$ (see Figure~\ref{fig:MV}). As a result we call a graph $G$ as a \emph{perfect prime graph} if it is a prime graph over any field. The graphs $K_2$, the cycles $C_{3k+2}$ and the complement of cycles $\overline{C}_m$ for any $k \geq 1$ and $m\geq 4$ are examples of perfect prime graphs, and except these graphs, we show that the M\"{o}bius-Kantor graph, which is the generalized Petersen graph $G(8,3)$ is also a ($3$-regular bipartite) perfect prime graph. In the case of simplicial complexes, any minimal triangulation of an oriented pseudomanifold provides an example of a perfect prime simplicial complex.

Obviously, Theorem~\ref{thm:reg-dec} reduces the graph's regularity computation into finding prime graphs as well as decompositions of a graph into primes, which is still a difficult task. On such a ground, we choose to follow two different roads. In one way, we look for combinatorial conditions on graphs that may effect their primeness. For instance, we prove that any graph $G$ with a (closed or open) dominated vertex can not be a prime graph. In other words, if $G$ contains two vertices $x$ and $y$ satisfying  $N_G[x]\subseteq N_G[y]$ or $N_G(x)\subseteq N_G(y)$, then $G$ can not be a prime graph, where $N_G(x)$ is the set of neighbours of $x$ in $G$ and $N_G[x]=N_G(x)\cup \{x\}$. Such simple combinatorial observations allow us to determine the set of primes that a graph may contain under some restrictions and reformulate the regularity of such graphs as a generalized induced matching problem (see Section~\ref{sect:primes} for details).

On the other hand, the theory of prime graphs also allows us to prove the upper bound $\reg(G)\leq \Delta(G)\im(G)$ for any graph $G$ (see Theorem~\ref{thm:red-alg}). Moreover, in the case of claw-free graphs, we verify that this bound can be strengthened by showing that $\reg(G)\leq 2\im(G)$, which generalizes an earlier result of Nevo~\cite{EN} on $(\text{claw}, 2K_2)$-free graphs. Additionally, we prove that  $C_5$ is the unique connected graph satisfying $\im(G)<\reg(G)=\m(G)$, where $\m(G)$
is the matching number of $G$, that answers a recent question of Hibi et al.~\cite{HHKT}. 

The other direction is shaped by the result of an operation on graphs due to Lozin~\cite{VVL}, which says that it suffices to restrict ourselves to the search of prime graphs in a narrow subclass of bipartite graphs. During his search on the complexity of the induced matching number, Lozin describes an operation (he calls it as the \emph{stretching operation}) on graphs, and  proves that when it is applied to a graph, the induced matching number increases exactly by one. His operation works simply by considering a vertex $x$ of a graph $G$ whose (open) neighbourhood is split into two disjoint parts $N_G(x)=Y_1\cup Y_2$, and replacing the vertex $x$ with a four-path on $\{y_1,a,b,y_2\}$ together with edges $uy_i$ for any $u\in Y_i$ and $i=1,2$ (see Section~\ref{sect:lozin}). One of the interesting results of his work is that the induced matching problem remains $\NP$-hard in a narrow subclass of bipartite graphs. We here prove that his operation has a similar effect on the regularity:

\begin{theorem}\label{thm:lozin+reg}
Let $G=(V,E)$ be a graph and let $x\in V$ be given. Then $\reg(\LE_x(G))=\reg(G)+1$, where $\LE_x(G)$ is the Lozin transform of $G$ with respect to the vertex $x$.
\end{theorem}
 
When it is combined with Theorem~\ref{thm:reg-dec}, one of the immediate result of Theorem~\ref{thm:lozin+reg} is that the determination of bipartite prime graphs having sufficiently large girth with maximum
degree at most three is in fact our main task. On that matter, we first prove that any such graph must be $2$-connected, and then show that the bound
$\reg(G)\leq 2\im(G)+1$ holds for any such graph $G$. In particular, we verify that any non-trivial Lozin operation preserves the primeness of a graph that enables us to generate new prime graphs from the existing ones.

We recall that the only known graph parameter so far that produces a lower bound to the regularity is the induced matching number. However, we here introduce a new graph invariant, the \emph{virtual induced matching number} $\vim(G)$, satisfying $\im(G)\leq \vim(G)\leq \reg(G)$ for any graph $G$. The logic behind defining such an invariant is to introduce a combinatorial operation that in a sense it closes the gap between the induced matching number and the regularity. The definition of the virtual induced matching number is based on a detail analyse on the effects of certain combinatorial operations to the regularity. To be more specific, we determine the behaviour of the regularity under edge contractions on the independence complexes of graphs and vertex expansions on graphs. By way of an application, we prove the equality $\reg(G)=\vim(G)$ for a graph class containing all Cohen-Macaulay graphs of girth at least five. 

Note also that one of the special case of Lozin operations corresponds to a triple edge subdivision on the graph, and Theorem~\ref{thm:lozin+reg} already determines its effect to the regularity. Along the same lines, we determine the effect of edge subdivisions of arbitrary lengths to the regularity, and verify that the impacts of a double edge subdivision on an edge and the contraction of that edge to the regularity are closely related. 
In particular, we prove the inequality $\reg(G/e)\leq \reg(G)\leq \reg(G/e)+1$ for any edge $e$ in a graph $G$ that in turn implies $\reg(H)\leq \reg(G)$ whenever $H$ is an edge contraction minor of $G$.

There are various graph parameters that can be used to bound the regularity, and most of them are far from being tight in general. For instance,
the inequality $\im(G)\leq \reg(G)\leq \cd(G)$ holds for any graph $G$, where $\cd(G)$ is the cochordal cover number of $G$~\cite{RW2}. Woodroofe has constructed in~\cite{RW2} graphs for which the gap between the regularity and the cochordal cover number could be arbitrarily large. For the lower bound, let $R_n$ be the graph obtained from $(n+1)$ disjoint five cycles by adding an extra vertex and connecting it to exactly one vertex of each five cycle (see Figure~\ref{fig-reg-im-1}). Then the graph $R_n$ is vertex-decomposable~\cite{DK,BC} and the equality $\reg(R_n)=\im(R_n)+n$ holds for any $n\geq 1$. 
\begin{figure}[ht]
\begin{center}
\begin{tikzpicture}[scale=0.7]

\node [noddee] at (0,0) (v1){};
\node [noddee] at (2,0) (v2){}
	edge [] (v1);
\node [noddee] at (0,2) (v3){}
	edge [] (v1);
\node [noddee] at (2,2) (v4){}
	edge [] (v2);
\node [noddee] at (1,3) (v5){}
	edge [] (v3)
	edge [] (v4);
\node [noddee] at (3,0) (v6){};
\node [noddee] at (5,0) (v7){}
     edge [] (v6);
\node [noddee] at (3,2) (v8){}	
	 edge [] (v6);
\node [noddee] at (5,2) (v9){}	
	 edge [] (v7);	 
\node [noddee] at (4,3) (v10){}	
	 edge [] (v8)
	 edge [] (v9);	 
\node [noddee] at (9,0) (v11){};
\node [noddee] at (11,0) (v12){}
     edge [] (v11);
\node [noddee] at (9,2) (v13){}	
	 edge [] (v11);
\node [noddee] at (11,2) (v14){}	
	 edge [] (v12);	 
\node [noddee] at (10,3) (v15){}	
	 edge [] (v13)
	 edge [] (v14);	
\draw [dashed, thick] (6,1) -- (8,1);
\node [noddee] at (6,4) (v16){}	
	 edge [] (v5)
	 edge [] (v10)
	 edge [] (v15);		
\end{tikzpicture}

\end{center}
\caption{The graph $R_n$.}
\label{fig-reg-im-1}
\end{figure}
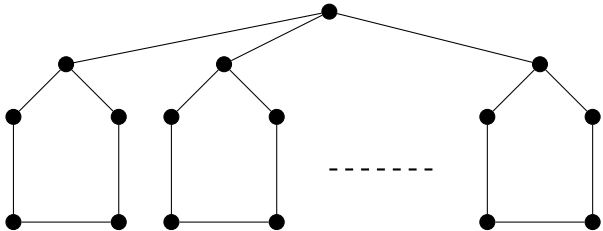
Even if the gap between the regularity and the induced matching number could be arbitrarily large, this does not guarantee that any pair $(n,k)$ of positive integers can be realized as
$(\reg(G),\im(G))$ for some graph $G$. So, we may ask whether there exists
a graph $G(n,k)$ such that $\reg(G(n,k))=n$ and $\im(G(n,k))=k$ for every pair $(n,k)$ of integers with $n\geq k\geq 1$ (compare to Question $7.1. (5)$ in~\cite{GW}).  
The case $k=1$ is of particular importance, since no example is known when $n\geq 5$~\cite{EP}.
We remark that for $n=4$, there is the Coxeter $600$-cell $X_{600}$~\cite{PS}, which is the independence complex of a $2K_2$-free graph $G_{600}$ and the geometric realization of $X_{600}$ is homeomorphic to the $3$-dimensional sphere so that $\reg(G_{600})=4$. Furthermore, Przytycki and Swiatkowski~\cite{PS} show that no generalized homology sphere of dimension $n\geq 4$ can be triangulated as the independence complex of a $2K_2$-free graph.  The existence of $2K_2$-free graphs of arbitrary large regularity can be deduced from the result of Januszkiewicz and Swiatkowski~\cite{JS}, stating\footnotemark that there exists a $2K_2$-free graph such that its independence complex is an oriented pseudomanifold of dimension $(n-1)$ for any $n\geq 1$. \footnotetext{In fact, the statement of their result is somewhat different than what we express here (see Section~\ref{sect:$2K_2$-free} for details).}

Now, combining Theorem~\ref{thm:lozin+reg} and the Januszkiewicz and Swiatkowski's result, we have a complete answer:

\begin{theorem}\label{thm:reg-im-pair}
For any two integers $n\geq k\geq 1$, there exists a graph $G(n,k)$ satisfying $\reg(G(n,k))=n$ and $\im(G(n,k))=k$.
\end{theorem}

{\it The organization of the paper:} The conventional background notations and informations that may be needed in the sequel are introduced in the Preliminaries section.
The prime complexes and prime factorizations are introduced in Section~\ref{sect:primes}, where we also investigate the structural properties of prime graphs and perform the regularity calculations in some hereditary graph classes. We there provide the proofs of various upper bounds  to the regularity involving the induced matching number.
The following three sections are devoted to the detailed analyse on the behaviour of the regularity under specific combinatorial operations, including the Lozin operations in Section~\ref{sect:lozin}, the edge-contractions and vertex-expansions on the independence complexes in Section~\ref{sect:contr-expan} and the edge subdivisions and contractions on graphs in Section~\ref{sect:edge-subdiv}. The proof of Theorem~\ref{thm:reg-im-pair} is given in Section~\ref{sect:$2K_2$-free}, where we also provide an upper bound on the regularity of $2K_2$-free graphs in terms of their local constraints.  

\section{Preliminaries}\label{sect:prel}
We first recall some general notions and notations needed throughout the paper, and repeat some of the
definitions mentioned in the introduction more formally. 
We choose to follow the standard terminology from combinatorial commutative algebra~\cite{Ha,MV}, topological combinatorics~\cite{DK} 
and graph theory~\cite{BM}. So, for any undefined terms, we refer to these references. We remark that most of our results is independent from the characteristic of the coefficient field,
so whenever it is appropriate we drop $\kk$ from our notation. 

{\bf Graphs:}
By a (simple) graph $G$, we will mean an undirected graph without loops or multiple edges. If $G$ is a graph, $V(G)$ and $E(G)$ (or simply $V$ and $E$) denote its vertex and edge sets. An edge between $u$ and $v$ is denoted by $e=uv$ or
$e=(u,v)$ interchangeably. If $U\subset V$, the graph induced on $U$ is written $G[U]$, and in particular, we abbreviate $G[V\backslash U]$ to $G-U$, and write $G-x$ whenever $U=\{x\}$. For a given subset $U\subseteq V$, the (open) neighbourhood of $U$ is
defined by $N_G(U):=\cup_{u\in U}N_G(u)$, where $N_G(u):=\{v\in V\colon uv\in E\}$, and similarly, $N_G[U]:=N_G(U)\cup U$ is the closed neighbourhood of $U$. Furthermore, if $F=\{e_1,\ldots,e_k\}$ is a subset of edges of $G$, we write $N_G[F]$ for the set $N_G[V(F)]$, where $V(F)$ is the set of vertices incident to edges in $F$. The cardinality of the set $N_G(x)$ is the degree $\deg_G(x)$ of the vertex $x$, and the maximum and the minimum degrees of a graph $G$ are denoted by $\Delta(G)$  and $\delta(G)$ respectively. The \emph{distance} of any two vertices $x,y\in V(G)$ will be denoted by $d_G(x,y)$.

Throughout $K_n$, $K_{n,m}$, $C_n$ and $P_n$ will denote the complete, complete bipartite, cycle and path graphs for $n, m\geq 1$  respectively. 
In particular, the graph $K_{1,3}$ is known as the \emph{claw} graph.

We say that $G$ is $H$-free if no induced subgraph of $G$ is isomorphic to $H$. A graph $G$ is called \emph{chordal} if it is $C_r$-free for any $r>3$. Moreover, a graph $G$
is said to be \emph{cochordal} if its complement $\overline{G}$ is a chordal graph. A subgraph $H$ of a graph $G$ is said to be \emph{dominating} if $N_G[V(H)]=V(G)$. A graph $G$ is said to be \emph{well-covered} if each maximal independent set in $G$ has the same cardinality.

Recall that a subset $M\subseteq E$ is called a {\it matching} of $G$ if no two edges in $M$ share a common vertex. Moreover, a matching $M$ of $G$ is an {\it induced matching} if the edges in $M$ are exactly the edges of the induced graph of $G$ over the vertices contained in $V(M)$, and the cardinality of a maximum induced matching is called the {\it induced matching number} of $G$ and denoted by $\im(G)$. 

{\bf Simplicial complexes:} 
An \emph{(abstract) simplicial complex} $\D$ on a finite set $V$ is a family of subsets of $V$ satisfying the following properties.
\begin{itemize}
\item[$(i)$] $\{v\}\in \D$ for all $v\in V$,
\item[$(ii)$] If $F\in \D$ and $H\subset F$, then $H\in \D$.
\end{itemize} 

The elements of $\D$ are called \emph{faces} of it; the \emph{dimension} of a face $F$ is $\textrm{dim}(F):=|F|-1$, and the \emph{dimension} of $\D$ is defined to be  $\textrm{dim}(\D):=\textrm{max}\{\textrm{dim}(F)\colon F\in \D\}$. 
The $0$ and $1$-dimensional faces of $\D$ are called \emph{vertices} and \emph{edges} while maximal faces are called \emph{facets}. In particular, we denote by $\FE_{\D}$, the set of facets of $\D$.

The simplicial \emph{join} of two complexes $\D_1$ and $\D_2$ on disjoint sets of vertices, denoted by $\D_1 * \D_2$, is the simplicial complex defined by 
$$\D_1\ast \D_2=\{\sigma \cup \tau\colon \sigma \in \D_1, \tau \in \D_2 \}.$$

Even if the regularity is not a topological invariant, the use of topological methods plays certain roles. In many cases, we will appeal to an induction on the cardinality of the vertex set by a particular choice of a vertex accompanied by two subcomplexes. To be more explicit, if $x$ is a vertex of $\D$, then the subcomplexes $\del_{\D}(x):=\{F\in \D\colon x\notin F\}$ and $\lk_{\D}(x):=\{R\in \D\colon x\notin R\;\textrm{and}\;R\cup \{x\}\in \D\}$ are called the \emph{deletion} and \emph{link} of $x$ in $\D$ respectively. Such an association brings the use of a Mayer-Vietoris sequence of the pair $(\D,x)$:

\begin{equation*}
\cdots \to \widetilde{H}_{j}(\lk_{\D}(x)) \to \widetilde{H}_{j}(\del_{\D}(x))\to \widetilde{H}_{j}(\D)\to \widetilde{H}_{j-1}(\lk_{\D}(x)) \to
\cdots \to \widetilde{H}_0(\D)\to 0.
\end{equation*}
The following provides an inductive bound on the regularity whose proof can be easily obtained from the associated Mayer-Vietoris exact sequence of the pair $(\D,x)$. 

\begin{proposition}\label{prop:induction-sc}
Let $\D$ be a simplicial complex and let $x\in V$ be given. Then
\begin{equation*}
\reg (\D)\leq \max\{\reg (\del_{\D}(x)), \reg(\lk_{\D}(x))+1\}.
\end{equation*}
\end{proposition}
We note that Dao, Huneke and Schweig~\cite{DHS} show that $\reg(\D)$ always equals to one of $\reg (\del_{\D}(x))$ or
$\reg(\lk_{\D}(x))+1$ for any simplicial complex $\D$ and any vertex $x$.

When considering the complex $\IE(G)$ of a graph, the deletion and link of a given vertex $x$ correspond to the independence complexes of induced subgraphs, namely that $\del_{\IE(G)}(x)=\IE(G-x)$ and $\lk_{\IE(G)}(x)=\IE(G-N_G[x])$. Therefore, the following is an immediate consequence of Proposition~\ref{prop:induction-sc}, which is also proven in algebraic setting in~\cite{MV}.

\begin{corollary}\label{cor:induction-sc}
Let $G$ be a graph and let $v\in V$ be given. Then
\begin{equation*}
\reg (G)\leq \max\{\reg (G-v), \reg(G-N_G[v])+1\}.
\end{equation*}
\end{corollary}

The existence of vertices satisfying some extra properties is useful when dealing with the homotopy type of simplicial complexes (see~\cite{MT} for details).

\begin{theorem}\label{thm:hom-simp-induction}
If $\lk_{\D}(x)$ is contractible in $\del_{\D}(x)$, then $\D\simeq \del_{\D}(x)\vee \Sigma(\lk_{\D}(x))$, where $\Sigma X$ denotes the (unreduced) suspension of $X$.
\end{theorem}

In the case of the independence complexes of graphs, the existence of a (closed or open) dominated vertex may guarantee that the required condition of Theorem~\ref{thm:hom-simp-induction} holds. 

\begin{corollary}\cite{AE, MT}\label{thm:hom-induction}
If $N_G(u)\subseteq N_G(v)$, then there is a homotopy equivalence $\IE(G)\simeq \IE(G-v)$. On the other hand,
if $N_G[u]\subseteq N_G[v]$, then the homotopy equivalence $\IE(G)\simeq \IE(G-v)\vee \Sigma \IE(G-N_G[v])$
holds.
\end{corollary}

\begin{definition}
An edge $e=uv$ is called an {\it isolating edge} of a graph $G$ with respect to a vertex $w$, if $w$ is an isolated vertex of $G-N_G[e]$.
\end{definition}

\begin{theorem}\cite{MA}\label{thm:isolating}
If $\IE(G-N_G[e])$ is contractible, then the natural inclusion $\IE(G)\hookrightarrow \IE(G-e)$ is a homotopy equivalence.
\end{theorem}

In particular, Theorem~\ref{thm:isolating} implies that $\IE(G)\simeq \IE(G-e)$ whenever the edge $e$ is isolating. This brings the use of an operation, adding or removing an edge, on a graph without altering its homotopy type. We will follow \cite{MA} to write $\Add(x,y;w)$ (respectively $\Del(x,y;w)$) to indicate that we add the edge $e=xy$ to (resp. remove the edge $e=xy$ from) the graph $G$, where $w$ is the corresponding isolated vertex. 

\begin{remark}
In order to simplify the notation, 
we note that when we mention the homology, homotopy or a suspension of a graph, we mean that of its independence complex, so 
whenever it is appropriate, we drop $\IE(-)$ from our notation.
\end{remark}



\section{Prime graphs and Prime Factorizations}\label{sect:primes}
As we have already mentioned in Section~\ref{sect:intro}, the notion of primeness brings a new strategy for the calculation of the regularity.
Even if we express its definition in Section~\ref{sect:intro}, there seems no harm for restating it in its greatest generality.
\begin{definition}
A connected simplicial complex $\D$ is called a \emph{prime complex} over a field $\kk$, if $\reg_{\kk}(\del_{\D}(x))<\reg_{\kk}(\D)$ for any vertex $x\in V(\D)$. 
Furthermore, we call a connected complex $\D$ as a \emph{perfect prime complex} if it is a prime complex over any field. 
\end{definition}

There is the degenerate case where $\D=\{\emptyset\}$ in which we count it as the (trivial) perfect prime. This is consistent with the usual conventions that $\widetilde{H}_{-1}(\{\emptyset\};\kk)\cong \kk$ and $\widetilde{H}_{p}(\{\emptyset\};\kk)\cong 0$ for any $p\neq -1$ in that case.

If $\D$ is a prime complex, then $\reg(\D)=\reg(\lk_{\D}(x))+1$
for any vertex $x\in V$. Moreover, observe that a graph $G$ is a prime graph if and only if $\IE(G)$ is a prime complex. 

Recall that the regularity is in general not a topological invariant for simplicial complexes. However, if the simplicial complex is prime,
then the determination of its homology suffices for the calculation of its regularity. In other words, the equality
$\reg_{\kk}(\D)=\min \{i\colon \widetilde{H}_j(\D;\kk)=0\;\text{for any}\;j>i\}+1$ holds for any prime complex over $\kk$.
Obviously, this has some interesting consequences. For example, any minimal triangulation of an orientable pseudomanifold~\cite{JM}
is a perfect prime complex.

\begin{definition}
Let $\D$ be a simplicial complex and let $\RE=\{R_1,\ldots, R_r\}$ be a set of pairwise vertex disjoint subsets of $V$ such that $|R_i|\geq 2$ for each
$1\leq i\leq r$. Then $\RE$ is said to be an \emph{induced decomposition} of $\D$ if $\D[\bigcup_{i=1}^r R_i]\cong \D[R_1]\ast\ldots \ast \D[R_r]$, and  $\RE$ is maximal with this property. The set of induced decompositions of a complex $\D$ is denoted by $\ID(\D)$.
Furthermore, if $\D[R_i]\cong \D[R_j]$ for some $i\neq j$, we may identify the sets $R_i$ and $R_j$ under such an isomorphism, and consider $\RE$ as a (multi)set in which the set $R_i$ has multiplicity $n_i$. In such a case, we  abbreviate it to $\RE=\{n_{i_1}R_{i_1},\ldots,n_{i_r}R_{i_r}\}$. 
\end{definition}

\begin{definition}
Let $\RE=\{R_1,\ldots, R_r\}$ be an induced decomposition of a complex $\D$. If each $\D[R_i]$ is a prime complex, 
then we call $\RE$ as a \emph{prime decomposition} of $\D$, and the set of prime decompositions of a complex $\D$ is denoted by
$\PD(\D)$. 
\end{definition}

\begin{corollary}
$\PD(\D)\neq \emptyset$ for any complex $\D$.
\end{corollary}
\begin{proof}
We proceed by an induction on the cardinality of $V$. If $\D$ is itself a prime graph, then $\{\D\}\in \PD(\D)$. Otherwise,
there exists a vertex $x\in V$ such that $\reg(\del_{\D}(x))=\reg(\D)$. Then, the complex $\del_{\D}(x)$ admits a prime decomposition by 
the induction so that $\PD(\D)\neq \emptyset$, since $\PD(\del_{\D}(x))\subseteq \PD(\D)$.
\end{proof}

\begin{theorem}\label{thm:complex-reg-dec}[Compare to Theorem~\ref{thm:reg-dec}]
For any simplicial complex $\D$ and any field $\kk$, we have $$\reg_{\kk}(\D)=\max \{\sum_{i=1}^{r}\reg_{\kk}(\D[R_i])\colon \{R_1,\ldots,R_r\}\in \PD_{\kk}(\D)\}.$$
\end{theorem}
\begin{proof}
If $\D$ is itself a prime complex, there is nothing to prove. Otherwise there exists a vertex $x\in V$ such that $\reg(\D)=\reg(\del_{\D}(x))$. If $\del_{\D}(x)$ is a prime complex, then $\{\del_{\D}(x)\}\in \PD(\D)$ so that the result follows. Otherwise, we have 
$\reg(\del_{\D}(x))=\max \{\sum_{i=1}^{t}\reg(\del_{\D}(x)[S_i])\colon \{S_1,\ldots,S_t\}\in \PD(\del_{\D}(x))\}$ by the induction. However,
since $\PD(\del_{\D}(x))\subseteq \PD(\D)$ for such a vertex, the claim follows.
\end{proof}

\begin{definition}
A prime decomposition $\RE$ of a simplicial complex $\D$ for which the equality of Theorem~\ref{thm:complex-reg-dec} holds is called a 
\emph{prime factorization} of $\D$, and the set of prime factorizations of $\D$ is denoted by $\PF(\D)$.
\end{definition}
\begin{corollary}
A simplicial complex $\D$ is prime if and only if $\PF(\D)=\{\D\}$.
\end{corollary}

We now turn our attention to prime graphs and prime decompositions of graphs.
\begin{remark}
If $\RE=\{R_1,\ldots,R_k\}$ is an induced decomposition of a graph $G$ (i.e., that of $\IE(G)$), 
we will not distinguish the sets $R_i$'s and the subgraphs that they induce in $G$.
Furthermore, we note that for such a decomposition, the inequality $k\leq \im(G)$ always holds. Moreover, it is also possible that $\sum_{i=1}^k \im(R_i)<\im(G)$, 
even if $\RE\in \PF(G)$.
\end{remark}

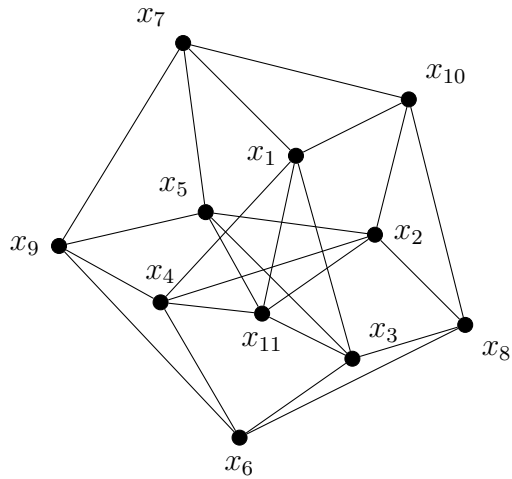
\begin{figure}[ht]
\begin{center}
\begin{tikzpicture}[scale=1.5]

\node [noddee] at (0,0) (v6) [label=below:$x_6$] {};
\node [noddee] at (2,1) (v8) [label=below right:$x_8$] {}
	edge [] (v6);
\node [noddee] at (-1.6,1.7) (v9) [label=left:$x_9$] {}
	edge [] (v6);
\node [noddee] at (1.5,3) (v10)[label=above right:$x_{10}$] {}
	edge [] (v8);
\node [noddee] at (-0.5,3.5) (v7) [label=above left:$x_7$] {}
	edge [] (v9)
	edge [] (v10);
\node [noddee] at (1,0.7) (v3) [label=above right:$x_3$] {}
	edge [] (v6)
	edge [] (v8);
\node [noddee] at (-0.7,1.2) (v4) [label=above:$x_4$] {}
	edge [] (v9)
	edge [] (v6);
\node [noddee] at (1.2,1.8) (v2) [label=right:$x_2$] {}
	edge [] (v4)
	edge [] (v8)
	edge [] (v10);
\node [noddee] at (0.5,2.5) (v1) [label=left:$x_1$] {}
	edge [] (v3)
	edge [] (v4)
	edge [] (v7)
	edge [] (v10);
\node [noddee] at (-0.3,2) (v5) [label=above left:$x_5$] {}
	edge [] (v2)	
	edge [] (v3)
	edge [] (v7)
	edge [] (v9);
\node [noddee] at (0.2,1.1) (v11) [label=below:$x_{11}$] {}
	edge [] (v1)
	edge [] (v2)
	edge [] (v3)
	edge [] (v4)
	edge [] (v5);

\end{tikzpicture}

\end{center}
\caption{The Morey-Villarreal graph $G_{MV}$.}
\label{fig:MV}
\end{figure}
We verify by a computer calculation~\cite{SAGE} that the graph
provided by Morey and Villareal (Example 3.6 in~\cite{MV}) is a prime graph over $\Z_2$, while it is not prime with respect to $\Z_3$ (see Figure~\ref{fig:MV}).

Apart from the already existing examples of perfect prime graphs, our next example shows that there exists a $3$-regular bipartite perfect prime graph.

\begin{example}\label{exmp:MK}
The bipartite $3$-regular M\"{o}bius-Kantor graph $G_{MK}$~(see~Figure~\ref{fig:tcbp}), also known as
the generalized Petersen graph $G(8,3)$~\cite{MP}  is a perfect prime graph. In order to verify that we first determine the homotopy type of the independence complex of $G_{MK}$ by repeated use of Theorem~\ref{thm:hom-simp-induction} and Corollary~\ref{thm:hom-induction} together with the fact that $G_{MK}$ is a vertex-transitive graph. It turns out that 
$G_{MK}-N_{G_{MK}}[x]\simeq S^2\vee S^3$ and $G_{MK}-x\simeq S^3\vee S^3\vee S^3$ for any vertex $x\in V(G_{MK})$ so that $\reg(G_{MK}-N_{G_{MK}}[x])$ and $\reg(G_{MK}-x)$ are at least $4$. It then follows that
\begin{equation*}
G_{MK}\simeq (\bigvee_{i=1}^4 S^3)\vee S^4,
\end{equation*}
which in particular implies that $\reg(G_{MK})\geq 5$. On the other hand, we have computed the induced matching and cochordal cover numbers of the corresponding graphs as follows;

\begin{center}
  \begin{tabular}{ |c || c | c |c|}
    \hline
    $H$ & $G_{MK}$ & $G_{MK}-x$ & $G_{MK}-N_{G_{MK}}[x]$\\ \hline
    $\cd(H)$ & 6 & 4 & 4 \\ \hline
    $\im(H)$ & 4 & 4 & 3 \\
    \hline
  \end{tabular}
\end{center}
where the computation of the cochordal cover number in each cases follows easily from Theorem~$27$ of~\cite{BC}, since the girth of $G_{MK}$ is $6$. As a result, observe that $\reg(G_{MK}-N_{G_{MK}}[x])=\reg(G_{MK}-x)=4$;
hence, we have $\reg(G_{MK})=5$ by Corollary~\ref{cor:induction-sc}, and that the graph $G_{MK}$ is a perfect prime. 
\begin{figure}[ht]
\begin{center}
\begin{tikzpicture}[scale=1.2]

\node [noddee] at (0,0) (v1) [label=left:$1$] {};
\node [noddee] at (2,0) (v2) [label=below:$2$] {}
	edge [] (v1);
\node [noddee] at (4,0) (v3) [label=right:$3$] {}
	edge [] (v2);
\node [noddee] at (0,2) (v4)[label=left:$8$] {}
	edge [] (v1);
\node [noddee] at (1,1) (v5) [label=left:$1'$] {}
	edge [] (v1);
\node [noddee] at (2,1) (v6)[label=left:$4'$] {}
    edge [] (v2);
\node [noddee] at (3,1) (v7) [label=right:$7'$] {}
     edge [] (v3);
\node [noddee] at (1,2) (v8)[label=above:$6'$] {}	
	 edge [] (v4)
	 edge [] (v7);	 

\node [noddee] at (4,2) (v10)[label=right:$4$] {}	
	 edge [] (v3);	
\node [noddee] at (3,2) (v9) [label=above:$2'$] {}	
	 edge [] (v10)
	 edge [] (v5);		   
\node [noddee] at (0,4) (v11)[label=left:$7$] {}	
	 edge [] (v4);	 
\node [noddee] at (2,4) (v12) [label=above:$6$] {}	
	 edge [] (v11);	 
\node [noddee] at (4,4) (v13) [label=right:$5$] {}	
	 edge [] (v12)
	 edge [] (v10);	
\node [noddee] at (1,3) (v14) [label=above:$3'$] {}	
	 edge [] (v11)
	 edge [] (v6)
	 edge [] (v9);
\node [noddee] at (2,3) (v15) [label=right:$8'$] {}	
	 edge [] (v12)
	 edge [] (v5)
	 edge [] (v7);
\node [noddee] at (3,3) (v16) [label=right:$5'$] {}	
	 edge [] (v13)
	 edge [] (v6)
	 edge [] (v8);	 	 	  
	 
\end{tikzpicture}
\end{center}
\caption{A $3$-regular bipartite perfect prime graph; the M\"{o}bius-Kantor graph $G_{MK}$.}
\label{fig:tcbp}
\end{figure}
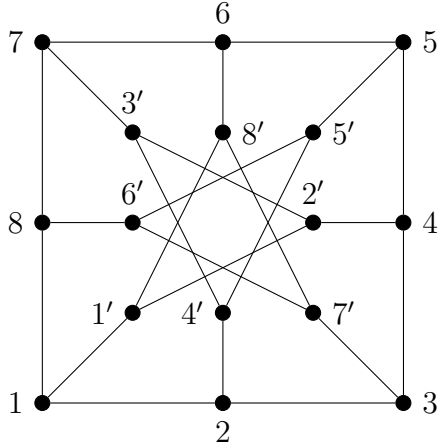
We further note that the graph $G_{MK}$ lacks certain properties. Firstly,
it is not well-covered, since for the independent set $S=\{2,3',6'\}$,
the graph $G_{MK}-N_{G_{MK}}[S]$ is not well-covered as it is isomorphic to  $P_5$. Moreover, it is not vertex-decomposable as it contains no shedding vertex; hence, it is not even sequentially Cohen-Macaulay~\cite{VT}. 
\end{example}

We may restate Theorem~\ref{thm:complex-reg-dec} in a more graph theoretical terms under which the regularity corresponds to a generalized induced matching problem.
\begin{definition}
Let $G$ be a graph, $\TE=\{T_1,\ldots,T_k\}$ be a set of connected graphs and $\af=(a_1,\ldots,a_k)$ be a sequence of non-negative integers. We then call the integer
\begin{equation*}
\im(G;\TE;\af):=\max\{a_1n_1+\ldots+a_kn_k\colon \{n_1T_1,\ldots,n_kT_k\}\in \ID(G)\}
\end{equation*}
as the \emph{induced matching number} of $G$ with respect to the pair $(\TE,\af)$. We make the convention that $\im(G;\TE;\af):=0$ if there exists no sequence of
non-negative integers $(n_1,\ldots,n_k)$ such that $\{n_1T_1,\ldots,n_kT_k\}\in \ID(G)$. 
\end{definition} 

\begin{corollary}\label{cor:reg-gim}
For any graph $G$, we have $\reg(G)\geq \im(G;\RE;\af_{\RE})$ for each $\RE\in \PD(G)$, where $\af_{\RE}=(\reg(H)\colon H\in \RE)$. In particular,
the equality $\reg(G)=\im(G;\RE;\af_{\RE})$ holds if $\RE\in \PF(G)$.
\end{corollary}

\begin{example}
We have $\reg(R_n)=\im(R_n;\{C_5\};\textbf{2})=2(n+1)$ for the graph $R_n$ depicted in Figure~\ref{fig-reg-im-1} for any $n\geq 1$. Note that any induced prime
subgraph of $R_n$ is isomorphic to either $K_2$ or $C_5$.
\end{example}

The above example suggests that Corollary~\ref{cor:reg-gim} is more useful when we know the set of induced prime subgraphs of a given graph. So, our next target is to look for combinatorial conditions on a graph that may effect its primeness or its induced primes.

\begin{proposition}\label{prop:dominated}
If $N_G(y)\subseteq N_G(x)$ for vertices $x$ and $y$, then $G$ can not be a prime graph. Similarly, if $N_G[u]\subseteq N_G[v]$ holds in $G$ such that $\deg_G(v)\geq 2$, then $G$ can not be a prime graph.
\end{proposition}
\begin{proof}
Let $G$ be a prime graph with $\reg(G)=k$.
Suppose first that $N_G(y)\subseteq N_G(x)$. Since $G$ is prime,
we have $\widetilde{H}_{k-1}(G)\neq 0$. However, the inclusion
$N_G(y)\subseteq N_G(x)$ implies that $G\simeq G-x$ by Corollary~\ref{thm:hom-induction} so that $\widetilde{H}_{k-1}(G-x)\neq 0$,
that is, $\reg(G-x)=k$, a contradiction.

Suppose next that $N_G[u]\subseteq N_G[v]$ holds in $G$. Assume that $G$ is a prime graph. This in turn implies that 
$\reg(G)=\reg(G-N_G[v])+1$. On the other hand, $\{K_2, G-N_G[v]\}$ is an induced decomposition of $G$, where the graph $K_2$ is induced by the pair $\{u,v\}$. In such a decomposition, we may replace the graph $G-N_G[v]$, if necessary, with a prime factorization of itself that creates a prime factorization for $G$, and such a factorization of $G$ never contains any neighbour of $v$ other than $u$, a contradiction. 
\end{proof}

One of the immediate consequences of Proposition~\ref{prop:dominated} is that the equality $\reg(G)=\im(G)$ holds for any graph in a hereditary graph class defined in terms of having a non-isolated (closed or open) dominated vertex, 
since only prime graph that such a graph in that class can contain is isomorphic to a $K_2$. For instance, the family of such graphs contains all chordal graphs~\cite{BC,HVT,Ha} or distance-hereditary graphs~\cite{BLS}. For the latter, we recall that any distance-hereditary graph $G$ can be inductively constructed by adding a new vertex $y$ for each existing vertex $x$ such that either $N_G(y)=\{x\}$ or the (open or closed) neighbourhood set of $y$ equals to that of $x$ in $G$.
Therefore, only prime distance-hereditary graph must be isomorphic to a $K_2$. We also note that distance-hereditary graphs are contained in the class of weakly chordal graphs; hence, the equality $\reg(G)=\im(G)$ for such graphs is already known~\cite{RW2}. However, we do not know whether a weakly chordal graph can contain primes other than a $K_2$.

\begin{theorem}\label{thm:3-5}
If $G$ is a $(C_3,P_5)$-free graph, then $\reg(G)=\im(G;\{K_2,C_5\};{\bf(1,2)})$.
\end{theorem}
\begin{proof} We choose to state two distinct proofs of the claim. One of them is purely combinatorial and the other uses topological methods.

{\it The first proof:} Suppose that $G$ is a prime and $(C_3,P_5)$-free graph. If $G$ is $C_5$-free, then it is a weakly chordal graph; hence, we have $\reg(G)=\im(G)$. So, we may assume that $G$ contains at least one induced five cycle $C$, say on the vertices $x_1,\ldots,x_5$ such that $x_ix_{i+1}\in E(G)$ in cyclic fashion. Suppose first that any vertex of $G$ not contained in $V(C)$ has exactly two neighbours in $C$, and let $y$ be such a vertex. We may assume without loss of generality that neighbours of $y$ in $C$ are $x_1$ and $x_3$. Then by Proposition~\ref{prop:dominated}, there exist vertices $u\in N_G(y)$ and $v\in N_G(x_2)$ such that $ux_2, yv\notin E(G)$. Note that we must have $uv\in E(G)$, since otherwise the set $\{u,x_2,x_1,y,v\}$ induces a $P_5$. To prevent the existence of
induced $5$-paths in $G$, the vertices $u$ and $v$ must have at least one neighbour in $C$. However, since $G$ is triangle-free, the only possible neighbours could be
$x_4$ and $x_5$. If $ux_4\in E(G)$, then the set $\{u,x_4,x_5,x_1,x_2\}$ induces a $P_5$, while if $ux_5\in E(G)$, then the set $\{u,x_5,x_4,x_3,x_2\}$ induces a $P_5$
in $G$, any of which is impossible. 

Assume now that any vertex in $V(G)\setminus V(C)$ has exactly one neighbour in $C$. If $y$ is such vertex and its neighbour in $C$ is $x_1$, then it follows from Proposition~\ref{prop:dominated} that $y$ has a neighbour $z$ outside of $C$. However, in such a case either the set $\{z,y,x_1,x_5,x_4\}$ or the set $\{z,y,x_1,x_2,x_3\}$ induces a $P_5$ in $G$, since $z$ can have at most one neighbour in $C$ by our assumption.

Therefore any such graph must be isomorphic to a $C_5$.

{\it The second proof:}
To prove the claim, it suffices to verify that any prime $(C_3,P_5)$-free graph $G$ with minimum degree at least two is isomorphic to a $C_5$, since only prime graph having a degree one vertex is $K_2$ by Proposition~\ref{prop:dominated}. Suppose that $G$ is such a graph, and consider a vertex $x\in V(G)$ with $\deg_G(x)\geq 3$. We let $N_G(x)=\{x_1,\ldots,x_k\}$. 

We first note that the vertices $x_1,\ldots,x_k$ can not have a common neighbour other than $x$ by Proposition~\ref{prop:dominated}. Furthermore, if there exist $i\in [k]$ and a minimal non-empty subset $J\subseteq [k]\setminus \{i\}$ of cardinality at least two such that $N_G(x_i)\subseteq \bigcup_{j\in J} N_G(x_j)$, then we claim that $\IE(G-N_G[x_i])$ is contractible; hence, $\IE(G)\simeq \IE(G-x_i)$, a contradiction, since $G$ is prime. Observe that $\deg_{G-x}(x_i)\geq 2$, and by the minimality of $J$, any vertex $x_j$ for $j\in J$ has at least one private neighbour in 
$N_G(x_i)$, that is, there exists $y_j\in N_G(x_i)\cap N_G(x_j)$ such that $y_j\notin N_G(x_t)$ for any $t\in J\setminus \{j\}$. Furthermore, any two vertices $x_{j_1}$ and $x_{j_2}$ for $j_1,j_2\in J$ can not have a common neighbour in $H:=G-N_G[x_i]$. Assume otherwise that $z\in V(H)\cap N_G(x_{j_1})\cap N_G(x_{j_2})$, then the set $\{z,x_{j_1},x,x_i,y_{j_2}\}$ induces a path in $G$ of length $4$, since $G$ is $C_3$-free. However, this contradicts to the fact that $G$ is $P_5$-free.
Once again, $(C_3,P_5)$-freeness of $G$ forces that either $|J|=2$ or at least one vertex $x_j$ for some $j\in J$ has degree zero in $H$. In the later case, the complex
$\IE(H)$ becomes a cone, so it is contractible. We may therefore assume that $J=\{j_1,j_2\}$ and both vertices $x_{j_1}$ and $x_{j_2}$ have degree at least one in $H$.
At this point, the union $N_H(x_{j_1})\cup N_H(x_{j_2})$ induces a complete bipartite graph in $H$, since $G$ is $P_5$-free. So, we need to have
$N_H(x_{j_1})\subseteq N_H(u)$ for any $u\in N_H(x_{j_2})$. This means that we can remove each such vertex $u\in N_H(x_{j_2})$ one-by-one without altering the homotopy type of $\IE(H)$ by Corollary~\ref{thm:hom-induction}. However, the vertex $x_{j_2}$ becomes isolated eventually, that is, $\IE(H)$ is contractible as claimed.

We may therefore assume that each vertex $x_i$ has a private neighbour $v_i$ in $G$ for each $i\in [k]$. However, since $k\geq 3$ and $G$ is $P_5$-free,
the set $\{v_1,\ldots,v_k\}$ must be a clique of $G$ that contradicts to $G$ being triangle-free.

Therefore, the maximum degree of $G$ must be two, that is, $G$ is isomorphic to a $C_5$. 
\end{proof}
Observe that for any $P_5$-free bipartite graph $B$, we have $\reg(B)=\im(B)$ by Theorem~\ref{thm:3-5}. Note that since bipartite $P_5$-free graphs are weakly chordal, the equality $\reg(B)=\im(B)$ also follows from a result of~\cite{RW2} for such graphs. However, we can extend it further:

\begin{theorem}\label{thm:bip-6}
If $G$ is a bipartite $P_6$-free graph, then $\reg(G)=\im(G)$.
\end{theorem}  
\begin{proof}
Once again, it suffices to prove that only induced prime in such a graph is isomorphic to a $K_2$. So, let $H$ be a prime and $P_6$-free bipartite graph. Following the characterization of $P_6$-free graphs due to van't Hof and Paulusma~\cite{HP}, such a graph must contain either a dominating complete bipartite subgraph or else an induced dominating $C_6$. 
We accordingly divide the proof into two cases, while noting that the methods of the proof in both cases are almost identical.
 
Suppose that $K:=K_{m,n}$ be a dominating complete bipartite subgraph of $H$,
and assume for contradiction that $H\ncong K_2$. So, let $V(H)=U\cup V$ and $V(K)=U'\cup V'$ such that $U'\subseteq U$ and $V'\subseteq V$. By Proposition~\ref{prop:dominated}, the equality $m=n=1$ is not possible so that we may suppose that $n\geq 2$. Furthermore, the graph $H$ can not contain any dominated vertex, which is again due to Proposition~\ref{prop:dominated}.

{\it Claim $1$.} $H$ contains an induced $C_6$.

{\it Proof of Claim $1$.} Let $x_1,x_2\in U'$ be given. Since they can not dominate each other, there exist $a_{12}, a_{21}\in V\setminus V'$ such that
$a_{12}\in N_H(x_1)\setminus N_H(x_2)$ and $a_{21}\in N_H(x_2)\setminus N_H(x_1)$. Pick a vertex $y_1\in V'$, and since $a_{12}$ is not dominated,
it has a neighbour, say $b_{12}\in U\setminus U'$ such that $b_{12}y_1\notin E(H)$. However, since $H$ is $P_6$-free, the edge $b_{12}a_{21}$ must be present in $H$; hence, the set $\{x_1,y_1,x_2,a_{21},b_{12},a_{12}\}$ induces the desired $C_6$.

{\it Claim $2$.} For any vertex $x\in U'$, the graph $T=H-N_H[x]$ is $2K_2$-free.

{\it Proof of Claim $2$.} Assume otherwise that $M$ is an induced matching in $T$ having order at least two. Observe that if $b\in V\setminus V'$ is an end vertex of an edge in $M$, since it can not be dominated by any vertex $y\in V'$, it has a neighbour $c_{(b,y)}\in U\setminus U'$ such that
$yc_{(b,y)}\notin E(H)$. 

{\it Case $2.1$.} Suppose that $V(M)\cap U'=\emptyset$ so that $M$ contains edges $a_1b_1,a_2b_2$ with $a_1,a_2\in U\setminus U'$ and $b_1,b_2\in V\setminus V'$. Since $K$ is dominating, the vertices $a_1$ and 
$a_2$ (resp. $b_1$ and $b_2$) have at least one neighbour in $V'$ (resp. in $U'$). 

{\it Subcase $2.1.(i)$.} Assume that $a_1,a_2\in N_H(y_1)$ and $b_1,b_2\in N_H(x_1)$ for some $y_1\in V'$ and $x_1\in U'\setminus \{x\}$. 
In such a case, we must have $c_{(b_1,y_1)}b_2\notin E(H)$, since otherwise the set
$\{x,y_1,a_1,b_1,c_{(b_1,y_1)},b_2\}$ induces a $P_6$. But then the set $\{c_{(b_1,y_1)},b_1,a_1,y_1,a_2,b_2\}$ induces a $P_6$ in $H$, a contradiction.

{\it Subcase $2.1.(ii)$.} Assume that $a_1,a_2\in N_H(y_1)$, while there exist distinct vertices $x_1,x_2\in U'$ such that $b_1\in N_H(x_1)\setminus N_H(x_2)$ and $b_2\in N_H(x_2)\setminus N_H(x_1)$. 
If $c_{(b_1,y_1)}b_2\notin E(H)$, then the set $\{c_{(b_1,y_1)},b_1,a_1,y_1,x_2,b_2\}$, and if
$c_{(b_1,y_1)}b_2\in E(H)$, then the set $\{x,y_1,x_2,b_1,c_{(b_1,y_1)},b_2\}$ induces a $P_6$ in $H$, both of which is impossible.

{\it Subcase $2.1.(iii)$.} Assume that $b_1,b_2\in N_H(x_1)$, while there exist distinct vertices $y_1,y_2\in V'$ such that $a_1\in N_H(y_1)\setminus N_H(y_2)$ and $a_2\in N_H(y_2)\setminus N_H(y_1)$. 
If $c_{(b_1,y_1)}y_2\notin E(H)$, the set $\{c_{(b_1,y_1)},b_1,a_1,y_1,x,y_2\}$ induces a $P_6$;
hence, $c_{(b_1,y_1)}y_2\in E(H)$. On the other hand, if $c_{(b_1,y_1)}b_2\notin E(H)$, then the set
$\{b_2,a_2,y_2,c_{(b_1,y_1)},b_1,a_1\}$, and if
$c_{(b_1,y_1)}b_2\in E(H)$, then the set $\{x,y_1,a_1,b_1,c_{(b_1,y_1)},b_2\}$ induces a $P_6$ in $H$.

{\it Subcase $2.1.(iv)$.} Assume that there exist distinct vertices $x_1,x_2\in U'$ and $y_1,y_2\in V'$ such that $b_1\in N_H(x_1)\setminus N_H(x_2)$, $b_2\in N_H(x_2)\setminus N_H(x_1)$ and $a_1\in N_H(y_1)\setminus N_H(y_2)$ and $a_2\in N_H(y_2)\setminus N_H(y_1)$. 
If $c_{(b_1,y_1)}b_2\notin E(H)$, then the set $\{c_{(b_1,y_1)},b_1,a_1,y_1,x_2,b_2\}$, and if
$c_{(b_1,y_1)}b_2\in E(H)$, then the set $\{x,y_1,x_2,b_1,c_{(b_1,y_1)},b_2\}$ induces a $P_6$ in $H$, both of which is impossible.

{\it Case $2.2$.} Suppose that $|V(M)\cap U'|=1$. We may therefore assume that $M$ contains edges of the form $x_1b_1$ and $a_2b_2$, where
$x_1\in U'\setminus \{x\}$, $a_2\in U\setminus U'$ and $b_1,b_2\in V\setminus V'$. Pick a vertex $y_1\in V'$.  If
$c_{(b_1,y_1)}b_2\in E(H)$, then the set $\{x,y_1,x_1,b_1,c_{(b_1,y_1)},b_2\}$ induces a $P_6$ in $H$ so that we must have $c_{(b_1,y_1)}b_2\notin E(H)$. However, it then follows that the edges $b_1c_{(b_1,y_1)}$ and $a_2b_2$ form an induced matching that shares no vertex with $U'$, which is not possible by Case $2.1$.

{\it Case $2.3$.} Suppose that $|V(M)\cap U'|=2$, and let
$M$ contains the edges $x_1b_1$ and $x_2b_2$ such that $x_1,x_2\in U'\setminus \{x\}$ and $b_1, b_2\in V\setminus V'$. Once again, pick a vertex $y_1\in V'$. If $c_{(b_1,y_1)}b_2\notin E(H)$, then the set $\{c_{(b_1,y_1)},b_1,x_1,y_1,x_2,b_2\}$, while
if $c_{(b_1,y_1)}b_2\in E(H)$, then the set $\{x,y_1,x_2,b_2,c_{(b_1,y_1)},b_1\}$ induces a $P_6$ in $H$, both of which is not possible.

This completes the proof of the Claim $2$. Now, since the graph $T$ is $2K_2$-free and bipartite, it follows that $T$ is a cochordal graph so that  $\reg(T)=1$, which in turn implies that $\reg(H)=2$, since $H$ is prime. 
However, such a graph can not be prime, since it contains an induced $C_6$ by Claim $1$.
 
Assume now that $H$ has a dominating induced $6$-cycle $C$, say on the vertices $x_1,\ldots,x_6$ such that $x_ix_{i+1}\in E(H)$ in the cyclic fashion. Since a $6$-cycle itself is not a prime graph, the set $V(H)\setminus V(C)$ is not empty.

{\it Claim $3$.} Any vertex $x\in V(H)\setminus V(C)$ has at least two neighbours in $C$, that is, $|N_C(x)|\geq 2$. 

{\it Proof of Claim $3$.} Since $V(C)$ is dominating in $H$, any such vertex has at least one neighbour in $V(C)$, and if it has a unique neighbour, then $H$ contains an induced $P_6$ which is not possible.

{\it Claim $4$.} For any vertex $x_i\in V(C)$, the graph $H-N_H[x_i]$ is $2K_2$-free.

{\it Proof of Claim $4$.} Consider the vertex $x_5$, and suppose that the graph $L=H-N_H[x_5]$ has an induced matching $M$ of cardinality $2$.

{\it Subclaim $4.1$.} $M\cap \{x_1x_2,x_2x_3\}=\emptyset$.

{\it Proof of Subclaim $4.1$.} Assume without loss of generality that $M=\{x_1x_2, xy\}$ for some $x,y\in V(H)\setminus V(C)$. Note that $x,y\notin N_H[x_5]$.
Now, if $xx_3\in E(H)$, then one of the vertices $x_4$ or $x_6$ must be adjacent to $x$ by Claim $3$. However, we then necessarily have $|N_C(y)\cap V(C)|\leq 1$, that contradicts to Claim $3$. Furthermore, if none of the vertices $x$ and $y$ is not adjacent to $x_3$, then one of these vertices has no neighbours in $V(C)$ which is not possible again by Claim $3$.

{\it Subclaim $4.2$.} $V(M)\cap \{x_1,x_2,x_3\}=\emptyset$.

{\it Proof of Subclaim $4.2$.} Assume without loss of generality that $M=\{x_ju, xy\}$ for some $u,x,y\in V(H)\setminus V(C)$ and $j\in [3]$. By the symmetry, it suffices to consider the cases only when $j\in \{1,3\}$ or $j=2$. 

{\it Case $4.2.(i)$.} $j=1$. In this case, we note that one of $x$ and $y$ is adjacent to either $x_2$ or $x_3$. So, we let $xx_2\in E(H)$. It follows that
$x_4,x_6\in N_C(y)$. However, this forces $|N_C(x)\cap V(C)|\leq 1$, a contradiction.

{\it Case $4.2.(ii)$.} $j=2$. It is sufficient to consider the case, where $x_1,x_3\in N_C(y)$ and $x_4,x_6\in N_C(x)$. On the other hand, the vertex $u$ must be adjacent to at least one of the vertices $x_4$ and $x_6$. If $ux_4\in E(H)$, then the set $\{y,x_1,x_2,u,x_4,x_5\}$ induces a $P_6$ in $H$, while if $ux_6\in E(H)$, then the set $\{x_5,x_6,u,x_2,x_3,y\}$ induces a $P_6$ in $H$, both of which is impossible. 

We may therefore assume that $M=\{xy,ab\}$ for some $x,y,a,b\in V(H)\setminus V(C)$. Again by Claim $3$, we note that at least one of the end vertex of the edges $xy$ and $ab$ has exactly two neighbours in $\{x_1,x_2,x_3\}$. So, assume that $x_1,x_3\in N_C(x)\cap N_C(a)$. It then follows that each of the vertices $y$ and $b$ has at least two neighbours in $\{x_2,x_4,x_6\}$. If $x_6\in N_C(y)\cap N_C(b)$ while $x_4\notin N_C(y)\cup N_C(b)$, then the set $\{b,x_6,y,x,x_3,x_4\}$ induces a $P_6$ in $H$. Thus,
we must have that at least one of $yx_4$ and $bx_4$ as an edge in $H$. However, if $yx_4\in E(H)$, then the set $\{a,x_1,x,y,x_4,x_5\}$, and if $bx_4\in E(H)$, then the set
$\{x_5,x_4,b,a,x_1,x\}$ induces a $P_6$ in $H$, any of which is impossible. By the symmetry, the case $x_4\in N_C(y)\cap N_C(b)$ while $x_6\notin N_C(y)\cup N_C(b)$ can be treated similarly. This completes the proof of Claim $4$. 

Now, as in the first case, since the graph $L$ is $2K_2$-free and bipartite, it follows that $L$ is a cochordal graph so that  $\reg(L)=1$, which in turn implies that $\reg(H)=2$, since $H$ is prime. 
However, such a graph can not be prime, since it contains an induced $C_6$.
\end{proof}

Our next aim is to provide an answer to a question 
of Hibi et al.~\cite{HHKT} (see Question $1.11$) regarding the existence of some infinite family of connected graphs satisfying $\im(G)<\reg(G)=\m(G)$,
where $\m(G)$ is the matching number of $G$. We recall that the inequality 
$\reg(G)\leq \m(G)$ naturally holds for any graph $G$~\cite{HVT}. They already observe that the only graph up to seven vertices with these constraints is $C_5$, and remark that such graphs might be rare. Indeed, we confirm their observation by showing that the graph $C_5$ is unique with these properties.

\begin{lemma}\label{lem:reg-match}
Let $H$ be a graph with $\reg(H)=\m(H)$. If $\{H_1,\ldots,H_k\}$ is a prime factorization of $H$, then $\reg(H_i)=\m(H_i)$ for each $i\in [k]$.
\end{lemma}
\begin{proof}
If there exists $j\in [k]$ such that $\reg(H_j)<\m(H_j)$, then 
\begin{equation*}
\m(H)=\reg(H)=\reg(H_1)+\ldots+\reg(H_k)<\m(H_1)+\ldots+\m(H_k)\leq m(H),
\end{equation*}
a contradiction.
\end{proof}

\begin{theorem}\label{thm:hibi}
If $G$ is a connected graph satisfying $\im(G)<\reg(G)=\m(G)$, then $G\cong C_5$.
\end{theorem}
\begin{proof}
We first prove the claim for prime graphs. So, let $G$ be a prime graph satisfying $\im(G)<\reg(G)=\m(G)$.
We proceed by an induction on the order of $G$. If $v\in V(G)$, we then have $\reg(G)=1+\reg(G-N_G[v])$. Observe that $\reg(G-N_G[v])=\m(G-N_G[v])$, since otherwise
$\reg(G)=1+\reg(G-N_G[v])<1+\m(G-N_G[v])\leq \m(G)$, which is not possible by the assumption. Now, let $\HE=\{H_1,\ldots,H_k\}$ be a prime factorization of 
$H:=G-N_G[v]$. It then follows from Lemma~\ref{lem:reg-match} that $\reg(H_i)=\m(H_i)$ for each $i\in [k]$. 
On the other hand, if $\im(H_i)<\reg(H_i)$ for some $i\in [k]$, we have 
$H_i\cong C_5$ by the induction hypothesis. Furthermore, if $\im(H_j)=\reg(H_j)$
for some $j\in [k]$, we then have $H_j\cong K_2$, since, in such a case, the equality $\im(H_j)=\m(H_j)$ implies that the graph $H_j$ must contain a (closed) dominated vertex by Theorem~$30$ of~\cite{BC}, which is only possible when $H_j\cong K_2$ by Proposition~\ref{prop:dominated}. Therefore, the prime factorization $\HE$ can be divided into two pieces $\HE_1:=\{nK_2\}$ and 
$\HE_2:=\{mC_5\}$ such that $n+m=k$.

{\it Case $1$.} $m\neq 0$. If we define $T:=\cup_{i=1}^k V(mC_5)$, then the set $U:=V(H)\setminus T$ is an independent set, which is due to the fact that $\reg(H)=\m(H)$. Moreover, by the same reason, none of the vertex $u\in U$ can have a neighbour in $T$. However, this implies that
each $C_5$ in $\HE_2$ is a connected component of $H$. On the other hand, the vertex $v$ has at least two neighbours, say $w$ and $z$ in $G$ by Proposition~\ref{prop:dominated}, and since $G$ is connected, at least one of these two vertices is adjacent to a vertex in $T$. Now, if $wp\in E(G)$ for some $p\in T$, then $\reg(G)<\m(G)$, since the addition of the edges $wp$ and
$vz$ increases the matching number of $G$ by two, a contradiction. 

{\it Case $2$.} $m=0$. In such a case, we necessarily have $\im(H)=\reg(H)=\m(H)$. If we define $L:=\cup_{i=1}^k V(kK_2)$, the set
$W:=V(H)\setminus L$ is an independent set as in the previous case.
Once again, by Theorem~$30$ of~\cite{BC}, the graph $H$ contains two vertices, say $p$ and $q$ such that $N_H[p]\subseteq N_H[q]$. However, since $G$ is prime, such a vertex $p$ can not be dominated
in $G$; hence, there exists $x\in N_G(v)$ such that $xp\in E(G)$, while $xq\notin E(G)$. Since $W$ is an independent set, one end of the edge $pq$ must be contained in $L$. We also note that the vertex $v$ has at least one neighbour $y$ other than $x$ by again Proposition~\ref{prop:dominated}.

{\it Subcase $2.1$.} If $q\in L$, while $p\notin L$, then the set $\HE\cup \{vy,xp\}$ is a matching in $G$ of size one more than $\reg(G)=\m(G)$, which is not possible.

{\it Subcase $2.2$.} Let $p\in L$ and $q\notin L$. Then there exists $t\in L$,
such that the edge $pt$ is in the prime factorization $\HE$. It then follows from Lemma~$1$ of~\cite{KR} that $tq\in E(H)$ and $\deg_H(p)=\deg_H(t)=2$.
Now, the set $\HE'\cup \{xp, vy\}$, where $\HE':=(\HE\setminus \{pq\})\cup \{tq\}$, is then a matching in $G$ of size one more than $\reg(G)=\m(G)$, which is not possible.

{\it Subcase $2.3$.} Assume that $p,q\in L$. Observe that the case where the vertices $p$ and $q$ have a common neighbour in $H$ is not possible by Subcase $2.2$. However, it then follows from Lemma~$1$ of~\cite{KR} that
we must have $\deg_H(p)=1$. Now, if there exists $h\in W$ such that $qh\in E(H)$, then the set $\HE'':=(\HE\setminus \{pq\})\cup \{xp,qh, vy\}$ is a matching in $G$ of size one more than $\reg(G)=\m(G)$, a contradiction.
So, we may further assume that $\deg_H(q)=1$. However, this forces that
the set $W$ is empty, that is, $H\cong kK_2$. On the other hand, since $G$ is prime, the vertex $q$ must have a neighbour, say $y$, in $N_G(v)$.
If the vertex $v$ has a neighbour $s$ other than $x$ and $y$, then the set $\HE''':=(\HE\setminus \{pq\})\cup \{xp,qy, vs\}$ is a matching in $G$ of size one more than $\reg(G)=\m(G)$; hence, $N_G(v)=\{x,y\}$.
Therefore, we must have either $k=1$ so that $G\cong C_5$ or else
the graph $G-N_G[y]$ is a star, that is, $G-N_G[y]\cong K_{1,l}$ for some $l\geq 1$. But, the latter case contradicts to the fact that $G$ is a prime graph.

Finally, assume that $G$ is not a prime graph, and let $\GE=\{G_1,\ldots,G_r\}$ be a prime factorization of $G$.
It then follows from Lemma~\ref{lem:reg-match} that $\reg(G_i)=\m(G_i)$
for any $i\in [r]$. Now, if $\im(G_i)=\reg(G_i)$ for each $i\in [r]$, then $G_i\cong K_2$ for by Proposition~\ref{prop:dominated} and Theorem~$30$ of~\cite{BC} so that $r=\im(G)=\reg(G)=\m(G)$, a contradiction. Therefore,
there exists $j\in [r]$ such that $\im(G_j)<\reg(G_j)$. Since $G_j$ is prime, then $G_j\cong C_5$  by the above argument. However, since $G$ is not prime, the set $V(G)\setminus \cup_{i=1}^rV(G_i)$ can not be empty. On the other hand, since $G$ is connected, we can use any vertex in this set together with those in $V(G_j)$ to create a matching in $G$ of size larger than $m(G)$, a contradiction.
\end{proof}

As we have already mentioned in Section~\ref{sect:intro}, the determination of bipartite prime graphs can be restricted to those having maximum degree at most three by the effect of Lozin operations (see Section~\ref{sect:lozin}). On this direction, we have the following partial result on the structure of such graphs.  
\begin{theorem}\label{thm:2-connected}
If $G$ is a prime bipartite graph with $\D(G)\leq 3$ and $|G|\geq 3$, then $G$ is $2$-connected.
\end{theorem}
\begin{proof}
Suppose that $G$ is a prime bipartite graph that is not $2$-connected. Then there exists a vertex $v$ such that $G-v$ is disconnected. We only handle the case where the vertex $v$ has degree three, since a similar argument applies to the case when $\deg_G(v)=2$. We first assume that the graph $G-v$ has three connected components, say $G_1$, $G_2$ and
$G_3$. Let $w_1$, $w_2$ and $w_3$ be the neighbours  of $v$ in these components respectively. Since $G$ is a prime graph, we have
\begin{align*}
\reg(G)-1=&\reg(G-v)=\reg(G_1)+\reg(G_2)+\reg(G_3)\\
=&\reg(G-N_G[v])=\reg(G_1-w_1)+\reg(G_2-w_2)+\reg(G_3-w_3).
\end{align*}

This in particular forces that $\reg(G_i)=\reg(G_i-N_{G_i}[w_i])$ for each $i\in [3]$. In fact, we have $\reg(G)-1=\reg(G-N_G[w_1])=\reg(G_1-N_{G_1}[w_1])+\reg(G_2)+\reg(G_3)$ so that $\reg(G_1)=reg(G_1-N_{G_1}[w_1])$. It then follows that
$$\reg(G)-1=\reg(G_1-N_{G_1}[w_1])+\reg(G_2-N_{G_2}[w_2])+\reg(G_3-N_{G_3}[w_3]).$$

Now, we note that $\{K_2, G_1-N_{G_1}[w_1],G_2-N_{G_2}[w_2],G_3-N_{G_3}[w_3]\}$ is an induced decomposition of $G$, where the graph $K_2$ is induced by the pair $\{v,w_1\}$. We may replace each subgraph $G_i-N_{G_i}[w_i]$, if necessary, with a prime factorization of itself that creates a prime factorization for $G$ other than itself, a contradiction.

Suppose now that $G-v$ has two connected components, say $G_1$ and $G_2$. We may assume that $v$ has one neighbour $w_1$ in $G_1$ and two neighbours $w_2$ and $w_3$ in $G_2$. Since $\reg(G)-1=\reg(G-v)=\reg(G-N_G[v])=\reg(G_1)+\reg(G_2)$, we can conclude as above that $\reg(G)-1=\reg(G_1-N_{G_1}[w_1])+\reg(G_2-\{w_2,w_3\})$. Once again, the family $\{K_2, G_1-N_{G_1}[w_1], G_2-\{w_2,w_3\}\}$ is an induced decomposition of $G$, where the isolated edge $K_2$ is induced by the pair $\{v,w_1\}$. However, this creates a prime factorization for $G$ other than itself, a contradiction.
\end{proof}

We next present a reduction algorithm from which we state some interesting new upper bounds on the regularity. We call a vertex $v$ of a graph $G$ a \emph{prime vertex} if $\reg(G-v)<\reg(G)$. Observe that the upper bound $\reg(G)\leq counter_F+\reg(G_F)$ always holds as a result of the reduction algorithm for any choice of the set $F\subseteq V(G)$.

\begin{algorithm}\label{algo:algo}
\caption{A reduction algorithm.}
\begin{algorithmic}[1] 
\STATE Input: A graph $G=(V,E)$ and a subset $F \subseteq V$ 
\STATE $counter=0$
\FOR{$v \in F$}
\IF{$v$ is a prime vertex}
\STATE $G=G-N_G[v]$ and $F=F-N_F[v]$
\STATE $counter=counter+1$
\ELSE 
\STATE $G=G-v$ and $F=F-v$
\ENDIF
\IF{$F=\emptyset$}
\STATE $BREAK$
\ENDIF
\ENDFOR
\STATE $counter_F=counter$
\STATE $G_F=G$
\end{algorithmic}
\end{algorithm}

\begin{proposition}\label{prop:reg-alg1}
The inequality $\reg(G)\leq 2(\Delta(G)-1)\im(G)$ holds for any connected graph $G$.
\end{proposition}
\begin{proof}
Let $V(M)$ be the vertex set of a maximum induced matching $M$ in $G$.
We set $F(M):=N_G(V(M))-V(M)$ and define $S(M):=N_G(F(M))-(F(M)\cup V(M))$. It then follows from the definition that $S(M)$ is an independent set in $G$ and $|F(M)|\leq 2(\Delta(G)-1)$.

Now, we choose $F:=F(M)$ and run the reduction algorithm. After $|F|$ steps,
the graph $G_F$ can only have edges from $G[V(M)]$. If $G_F$ is an edgeless graph, there is nothing to do. On the other hand, if $e=uv$ is an edge in $G_F$, then any vertex $w\in N_{G}[e]\setminus \{u,v\}$ has no effect to the counter, where the set $N_{G}[e]\setminus \{u,v\}$ is not empty, since $G$ is connected. This means that $counter_F\leq 2(\Delta(G)-1)\im(G)-\sum_{e=uv\in E(G_F)} |N_{G}[e]\setminus \{u,v\}|$ holds. Therefore, by adding one to the counter for each edge in $G_F$, the regularity can be at most $2(\Delta(G)-1)\im(G)$.
\end{proof}

Proposition~\ref{prop:reg-alg1} has an immediate consequence:

\begin{corollary}
If $M$ is a maximal induced matching in a connected graph $G$, then
$\reg(G)\leq \max\{\im(G), |F(M)|\}$, where $F(M):=N_G(V(M))-V(M)$.
\end{corollary}

For claw-free graphs, we can further strength the upper bound of Proposition~\ref{prop:reg-alg1}. 

\begin{proposition}\label{prop:alg-claw}
If $G$ is a connected claw-free graph, then $\reg(G)\leq 3\im(G)$.
\end{proposition}
\begin{proof}
As in the proof of Theorem~\ref{thm:red-alg}, choose $F:=F(M)$ for a maximum induced matching $M$ in $G$, and run the reduction algorithm. Let $e=uv$ be an edge in $M$. We note that since $G$ is claw-free, 
the set $N_G[e]\setminus \{u,v\}$ can contain at most three pairwise non adjacent prime vertices; hence, $\reg(G)\leq 3\im(G)$. 
\end{proof}

Even if we have the bound of Proposition~\ref{prop:alg-claw} for claw-free graphs, the expected one~\cite{RW2} is much smaller.

\begin{theorem}\label{thm:reg-claw-alg}
If $G$ is a claw-free graph, then $\reg(G)\leq 2\im(G)$.
\end{theorem}
\begin{proof}
We proceed by an induction on the order of $G$. Let $M=\{u_1v_1,\ldots, u_kv_k\}$ be a maximum induced matching in $G$. Observe that if none of the vertices $u_i$ and $v_i$ is prime in $G$ for $i\in [k]$, we conclude that
$\reg(G)=\reg(G-u_1)\leq 2\im(G-u_1)\leq 2\im(G)$, where the middle inequality is due to the induction. We may therefore assume that each vertex in $U:=\{u_1,\ldots,u_k\}$ is a prime vertex of $G$. We define $H:=G-N_G[U]$,
and note that the set $K_i:=N_H(v_i)$ is a clique of $H$ for each  $i\in [k]$, since $G$ is claw-free. However, it then follows that $H$ has a split edge covering in terms of the graphs $T_i:=(K_i,F_i)$, where $F_i:=N_H(K_i)\setminus K_i$, that is, $\reg(H)\leq k$, which in turn implies that $\reg(G)\leq \reg(H)+\im(G)\leq 2\im(G)$ as claimed.
\end{proof}

We remark that Theorem~\ref{thm:reg-claw-alg} generalizes the result of Nevo~\cite{EN} on $(\text{claw},2K_2)$-free graphs. Moreover, 
the bound of Theorem~\ref{thm:reg-claw-alg} is tight as shown
by the graphs $C_5$ and $\overline{C}_n$ for any $n\geq 6$.

The idea of the proof of Theorem~\ref{thm:reg-claw-alg} can be helpful to improve the bound of Proposition~\ref{prop:reg-alg1} further.

\begin{theorem}\label{thm:red-alg}
The inequality $\reg(G)\leq \D(G)\im(G)$ holds for any graph $G$.
\end{theorem}
\begin{proof}
Let $M=\{u_1v_1,\ldots, u_kv_k\}$ be a maximum induced matching in $G$.
Once again we proceed by an induction on the order of $G$.
Note that if none of the vertices $u_i$ and $v_i$ is prime in $G$ for $i\in [k]$, we have that $\reg(G)=\reg(G-u_1)\leq \D(G-u_1)\im(G-u_1)\leq \D(G)\im(G)$ by the induction. Suppose that each vertex in $U:=\{u_1,\ldots,u_k\}$ is a prime vertex of $G$. We define $F(M):=N_G(V(M))-V(M)$ and $H:=G-N_G[U]$, and note that the set $S(M):=N_G(F(M))-(F(M)\cup V(M))$ is an independent set in $G$.
If we set $F:=F(M)\cap V(H)$, observe that $F$ can contain at most $(\D(G)-1)\im(G)$ vertices. Now, we run the reduction algorithm for the graph $H$ with respect to $F$ so that $\reg(G)\leq \reg(H)+\im(G)\leq (\D(G)-1)\im(G)+\im(G)=\D(G)\im(G)$ as claimed.
\end{proof}

We note that Theorem~\ref{thm:red-alg} already provides the upper bound
$\reg(G)\leq 3\im(G)$ for bipartite graphs with $\D(G)\leq 3$, while we may strength it further by a detail look at their local structure. However,
we first need two technical results which are also needed throughout the sequel.

\begin{lemma}\label{lem:remove-deg}
Let $x,y,z$ be three vertices of a graph $G$ with $xy,yz\in E$.
If $\deg_G(x)=1$ and $\deg_G(y)=2$, then $\reg(G)=\reg(G-z)$.
\end{lemma}
\begin{proof}
Assume otherwise that $\reg(G-z)<\reg(G)=k$. It then follows that
$\reg(G)=\reg(G-N_G[z])+1$. Therefore, there exists a minimal set $S\subseteq V(G-N_G[z])$ such that $\widetilde{H}_{k-2}(G-N_G[z])\neq 0$.
However, if we set $S^*:=S\cup \{x,y\}$, then $G[S^*]\simeq \Sigma(G[S])$; hence $\widetilde{H}_{k-1}(G[S^*])\neq 0$. But then we must have 
$\reg(G-z)=k$, since $z\notin S^*$, a contradiction.
\end{proof}

\begin{lemma}\label{lem:deg-one}
If $N_G(x)=\{y\}$ in $G$, then either $\reg(G)=\reg(G-x)$ or else
$\reg(G)=\reg(G-N_G[y])+1$.
\end{lemma}
\begin{proof}
We let $\reg(G)=m$, and assume that $S\subseteq V$ is a minimal subset satisfying $\widetilde{H}_{m-1}(G[S])\neq 0$. If $x$ is not a prime vertex of $G$, then $\reg(G)=\reg(G-x)$. So, we may further assume that $x$ is a prime vertex in $G$, which in turn implies that $x\in S$. However, since $\deg_G(x)=1$, we must have $y\in S$ by the minimality of $S$. Now, if
$u\in N_{G[S]}(y)\setminus \{x\}$, then $N_{G[S]}(x)\subseteq N_{G[S]}(u)$ so that
$G[S]\simeq G[S]-u$ by Corollary~\ref{thm:hom-induction}, which is not possible by the minimality of $S$. But then $S\cap N_G(y)=\{x\}$.
We therefore have $G[S]\cong (G-N_G[y])[S]\cup K_2$, where the component $K_2$ is induced by the edge $xy$; hence, $\reg(G)=\reg(G-N_G[y])+1$ as claimed.
\end{proof}
\begin{theorem}\label{thm:sharp-bip}
If $G$ is a connected bipartite graph such that $\delta(G)<\Delta(G)\leq 3$,
then $\reg(G)\leq 2\im(G)$. Furthermore, if $G$ is a $3$-regular connected bipartite graph, then $\reg(G)\leq 2\im(G)+1$.
\end{theorem}
\begin{proof}
Suppose that $G$ is a connected bipartite graph with $\delta(G)<\Delta(G)\leq 3$.
We proceed by the induction on the order of $G$. We first prove the claim when $G$ is a prime graph. Now, note that $G$ must have a vertex, say $v\in V$, of degree two. Let $N_G(v)=\{a,b\}$. We choose a neighbour $u$ of $a$ that is not adjacent to $b$. Observe that the existence of such a vertex $u$ is guaranteed by Proposition~\ref{prop:dominated}. However, we then have $\reg(H)=\reg(G)-1$, where $H:=G-N_G[u]$. Furthermore, the vertex $v$ is of degree one in $H$, and $b$ is its only neighbour. It then follows from Lemma~\ref{lem:deg-one} that
we have either $\reg(H)=\reg(H-v)$ or $\reg(H)=\reg(H-N_H[b])+1$, where we analyse each case separately.

{\it Case $1$.} $\reg(H)=\reg(H-N_H[b])+1$. Let $\HE=\{H_1,\ldots,H_s\}$ be a prime factorization of $H-N_H[b]$. Observe that none of the graphs $H_i$ is  $3$-regular, since $G$ is connected and $\Delta(G)\leq 3$. 
However, we may then apply the induction to conclude that 
\begin{align*}
\reg(G)=\reg(H-N_H[b])+2&=\sum_{i=1}^s \reg(H_i)+2\leq 2(\sum_{i=1}^s \im(H_i))+2\\
&\leq 2\im(H-N_H[b])+2\leq 2\im(G),
\end{align*}
since, if $M$ is an induced matching of $H-N_H[b]$, then $M\cup \{vb\}$ is an induced matching in $G$.

{\it Case $2$.} $\reg(H)=\reg(H-v)$. If $\deg_G(a)=2$, then $\im(H-v)+1\leq \im(G)$ so that we can consider a prime factorization of $H-v$ as in Case $1$ to conclude that $\reg(G)=\reg(H-v)+1\leq 2\im(H-v)+1\leq 2\im(G)$.
We may therefore assume that $\deg_G(a)=3$. Let $w$ be the neighbour of $a$ other than $u$ and $v$. If $w$ is not a prime vertex of $H-v$, then
$\reg(H-v)=\reg(H-\{v,w\})$. Since the vertex $a$ has degree two in $H-w$,
we may proceed as above, that is, we can further suppose that $w$ is a prime vertex of $H-v$. It then follows that $\reg(H-v)=\reg((H-v)-N_{H-v}[w])+1$.
Once again, we can consider a prime factorization of $(H-v)-N_{H-v}[w]$
so that $\reg(G)=\reg((H-v)-N_{H-v}[w])+2\leq 2\im((H-v)-N_{(H-v)}[w])+2\leq 2\im(G)$.

Assume that $G$ is not a prime graph, and let $\{G_1,\ldots,G_k\}$ be a prime factorization of $G$. Since $G$ is connected, we have $\delta(G_i)<\Delta(G_i)\leq 3$ for each $i\in [k]$ so that $\reg(G_i)\leq 2\im(G_i)$ by the induction together with the fact that $G_i$ is prime. But then $\reg(G)=\sum_{i=}^k \reg(G_i)\leq
2(\sum_{i=1}^k\im(G_i))\leq 2\im(G)$. 

Finally, let $G$ be a $3$-regular connected bipartite graph. If $G$ is not prime, then there exists a vertex $x\in V$ such that $\reg(G)=\reg(G-x)$.
However, the connected components of the graph $G-x$ can not be $3$-regular, since $G$ is connected; thus, $\reg(G)=\reg(G-x)\leq 2\im(G-x)\leq 2\im(G)$ by the previous case.
On the other hand, if $G$ is a prime graph, then $\reg(G)=\reg(G-N_G[x])+1$
for any vertex $x\in V$, while each connected component of $G-N_G[x]$ is not a $3$-regular. Therefore, we conclude that $\reg(G)=\reg(G-N_G[x])+1\leq 2\im(G-N_G[x])+1\leq 2\im(G)+1$ as claimed.
\end{proof}

\section{Lozin operations}\label{sect:lozin}

The general knowledge on the behaviour of the regularity under combinatorial operation is limited to only vertex or edge removals on (hyper)graphs. So,  we offer a detailed combinatorial approach in order to describe the effect of various operations on the regularity through the coming three sections.

In his search of algorithmic aspect of the induced matching problem, Lozin describes an operation that increases the induced matching number exactly by one, and proves that a successive applications of his operations turns the given graph into a graph having maximum degree at most three and arbitrarily large girth. We prove in this section\footnote{Some of the results of this section is appeared in an unpublished manuscript~\cite{BCU}.} that his operation has a similar effect on the regularity. In particular, we show that any non-trivial Lozin operation preserves the primeness of a graph that in turn allows us to generate new prime graphs from the existing ones.

We begin with recalling the definition of the Lozin transformation on graphs~\cite{VVL}.

Let $G=(V,E)$ be a graph and let $x\in V$ be given. The \emph{Lozin transform} $\LE_x(G)$ of $G$ with respect to
the vertex $x$ is defined as follows:
\begin{itemize}
\item[$(i)$] partition the neighbourhood $N_G(x)$ of the vertex $x$ into two subsets $Y$ and $Z$ in arbitrary way;\\
\item[$(ii)$] delete vertex $x$ from the graph together with incident edges;\\
\item[$(iii)$] add a $P_4=(\{y,a,b,z\},\{ya,ab,bz\})$ to the rest of the graph;\\
\item[$(iv)$] connect vertex $y$ of the $P_4$ to each vertex in $Y$, and connect $z$ to each vertex in $Z$.
\end{itemize}

\begin{figure}[ht]
\begin{center}
\begin{tikzpicture}[scale=1]

\node [noddee] at (1.5,2.5) (v1) [label=above:$x$] {};
\node [noddee] at (5,2.5) (v2) [label=above:$y$] {};
\node [noddee] at (6,2.5) (v3) [label=above:$a$] {}
	edge [] (v2);
\node [noddee] at (7,2.5) (v4)[label=above:$b$] {}
	edge [] (v3);
\node [noddee] at (8,2.5) (v5)[label=above:$z$] {}
	edge [] (v4);
	
\draw [rounded corners] (0,0.2) rectangle (1,1);		
\draw [rounded corners] (2,0.2) rectangle (3,1);
\draw [rounded corners] (4.5,0.2) rectangle (5.5,1);
\draw [rounded corners] (7.5,0.2) rectangle (8.5,1);

\node  at (0.5,0.6) (v6)  {$Y$};
\node  at (2.5,0.6) (v7)  {$Z$};
\node  at (5,0.6) (v8)  {$Y$};
\node  at (8,0.6) (v9)  {$Z$};

\draw (1.5,2.5) -- (0.05,0.97);		
\draw (1.5,2.5) -- (0.95,0.97);
	
\draw (1.5,2.5) -- (2.05,0.97);	
\draw (1.5,2.5) -- (2.95,0.97);	

\draw (5,2.5) -- (4.55,0.97);	
\draw (5,2.5) -- (5.45,0.97);	
	
\draw (8,2.5) -- (7.55,0.97);	
\draw (8,2.5) -- (8.45,0.97);		
	
\end{tikzpicture}

\end{center}
\caption{The Lozin transformation}
\end{figure}
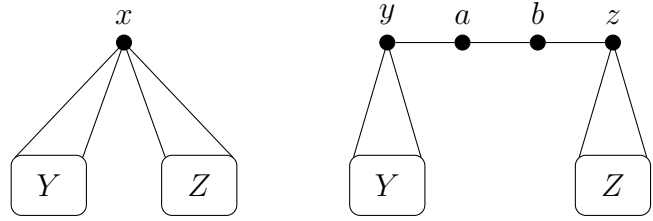

When the decomposition $N_G(x)=Y\cup Z$ is of importance, we will write $\LE_x(G;Y,Z)$ instead of $\LE_x(G)$.
It should be noted that we allow one of the sets $Y$ and $Z$ to be an empty set, in which case, we call the resulting operation as a \emph{trivial Lozin operation}. Furthermore, if $x$ is an isolated vertex,
the corresponding Lozin transform of $G$ with respect to the vertex $x$ is the graph $(G-x)\cup P_4$.

\begin{lemma}\label{lem:lozin-inmatch}~\cite{VVL}
For any graph $G$ and any vertex $x\in V(G)$, the equality $im(\LE_x(G))=im(G)+1$ holds. Furthermore, any graph can be transformed by a sequence of Lozin transformations into a bipartite graph of maximum degree three having arbitrary large girth.
\end{lemma}

We next prove that the Lozin transformation has a similar effect on the regularity, whose proof is divided into several steps. 
\begin{proposition}\label{prop:reg+isolating}
Let $e=uv$ be an isolating edge of $G$ with respect to the vertex $w$. If $w$ has a neighbour $x$ of degree two,
then $\reg(G)=\reg(G-e)$.
\end{proposition}
\begin{proof}
Without loss of generality, we may assume that $N_G(x)=\{u,w\}$. We let $R:=G-e$, and assume that $\reg(G)=m$.
Let $S\subseteq V=V(G)=V(R)$ be a subset satisfying $\widetilde{H}_{m-1}(G[S])\neq 0$. We suppose that
$S$ is maximal in the sense that $\widetilde{H}_{m-1}(G[T])=0$ for any $S\subsetneq T$.

If $u\notin S$ or $v\notin S$, we clearly have $G[S]\cong R[S]$. Thus, we may assume that $u,v\in S$. In this case, if $w\in S$, then $G[S]\simeq G[S]-e\cong R[S]$, since $e$ is then an isolating edge of $G[S]$. So, suppose that $w\notin S$. If $x\in S$, we conclude that $G[S]\simeq G[S]-v\cong R[S\backslash \{v\}]$, where the homotopy equivalence is due to the inclusion $N_{G[S]}(x)\subseteq N_{G[S]}(v)$ (see Corollary~\ref{thm:hom-induction}). It follows that we may further assume $x\notin S$. We then define $S^*:=S\cup \{x\}$, and consider the Mayer-Vietoris sequence of the pair $(G[S^*],x)$: 
\begin{equation*}
\cdots \to \widetilde{H}_{m-1}(G[S^*]-N_{G[S^*]}[x])\to \widetilde{H}_{m-1}(G[S^*]-x)\to \widetilde{H}_{m-1}(G[S^*]) \to \cdots. 
\end{equation*} 
Note that $\widetilde{H}_{m-1}(G[S^*])=0$ by the maximality of $S$. Moreover, since $G[S^*]-x\cong G[S]$,
we need to have $\widetilde{H}_{m-1}(G[S^*]-x)\neq 0$. Therefore, we conclude that $\widetilde{H}_{m-1}(G[S^*]-N_{G[S^*]}[x])\neq 0$. But, $G[S^*]-N_{G[S^*]}[x]\cong G[S\backslash \{u\}]$ and $G[S\backslash \{u\}]\cong R[S\backslash \{u\}]$ so that we have $\widetilde{H}_{m-1}(R[S\backslash \{u\}])\neq 0$. As a result, we deduce that $\reg(R)\geq \reg(G)$.  

Assume that $\reg(R)=n$, and let $W\subseteq V$ be a maximal subset satisfying $\widetilde{H}_{n-1}(R[W])\neq 0$.
Similar to the previous cases, we may assume that $u,v\in W$ and $w\notin W$. Suppose first that
$x\in W$. Then, $R[W]\simeq \Sigma(R[W]-N_{R[W]}[u])$, since $x$ has degree one in $R[W]$ and $u$ is
its only neighbour. It follows that $\widetilde{H}_{n-2}(R[W]-N_{R[W]}[u])\neq 0$. We define $K:=(W\backslash N_G(u))\cup \{v,x,w\}$. Since $e$ is an isolating edge in $G[K]$, we have $G[K]\simeq G[K]-e$. However, the vertex $u$ is of degree one in $G[K]-e$ and $x$ is its only neighbour; hence, $G[K]-e\simeq \Sigma((G[K]-e)-N_{G[K]-e}[x]))$ by Corollary~\ref{thm:hom-induction}. On the other hand, we have  $(G[K]-e)-N_{G[K]-e}[x]\cong R[W]-N_{R[W]}[u]$
so that $\widetilde{H}_{n-1}(G[K])\neq 0$. 

Finally, suppose now that $x\notin W$. As in the previous case, we let $W^*:=W\cup \{x\}$, and consider the Mayer-Vietoris sequence of the pair $(R[W^*],x)$:
\begin{equation*}
\cdots \to \widetilde{H}_{n-1}(R[W^*]-N_{R[W^*]}[x])\to \widetilde{H}_{n-1}(R[W^*]-x)\to \widetilde{H}_{n-1}(R[W^*]) \to \cdots. 
\end{equation*} 
Again, we have $\widetilde{H}_{n-1}(R[W^*])=0$ by the maximality of $W$. Since $R[W^*]-x\cong R[W]$,
the group $\widetilde{H}_{n-1}(R[W^*]-x)$ is non-trivial that implies $\widetilde{H}_{n-1}(R[W^*]-N_{R[W^*]}[x])\neq 0$.
However, $R[W^*]-N_{R[W^*]}[x]\cong R[W\backslash \{u\}]$ and $R[W\backslash \{u\}]\cong G[W\backslash \{u\}]$ so that we have $\widetilde{H}_{n-1}(G[W\backslash \{u\}])\neq 0$. As a result, we deduce that $\reg(G)\geq \reg(R)$. This completes
the proof.     
\end{proof}
We next verify that the homotopy type of Lozin transform can be deduced from the source graph.

\begin{lemma}\label{lem:lozin+stable}
Let $G=(V,E)$ be a graph and let $x\in V$ be given. Then $\LE_x(G;Y,Z)\simeq \LE_x(G;Y',Z')$
for any two distinct decompositions $\{Y,Z\}$ and $\{Y',Z'\}$ of $N_G(x)$. 
\end{lemma}
\begin{proof}
To prove the claim, it is enough to verify that for a given decomposition $N_G(x)=Y\cup Z$ and a vertex $u\in Z$,
the graphs $\LE_x(G;Y,Z)$ and $\LE_x(G;Y\cup \{u\},Z\backslash \{u\})$ have the same homotopy type.
However, moving the vertex $u$ from $Z$ to $Y$ corresponds to the sequence of isolating operations $\Add(u,y;b)$ and $\Del(u,z;a)$ in $\LE_x(G;Y,Z)$; therefore, the claim follows from Theorem~\ref{thm:isolating}.
\end{proof}

\begin{proposition}\label{prop:lozin+stable}
Let $G=(V,E)$ be a graph and let $x\in V$ be given. Then $\LE_x(G)\simeq \Sigma(G)$.
\end{proposition}
\begin{proof}
In view of Lemma~\ref{lem:lozin+stable}, it is sufficient to show that $\LE_x(G;N_G(x),\emptyset)\simeq \Sigma(G)$.
In such a case, we set $\RE_x(G)=\LE_x(G;N_G(x),\emptyset)$ and note that $N_{\RE_x(G)}(z)\subseteq N_{\RE_x(G)}(a)$ so that
$\RE_x(G)$ is homotopy equivalent to $\RE_x(G)-a$, while the latter graph is clearly isomorphic to 
$G\cup K_2$, where $K_2$ is induced by the edge $bz$. It then follows that $\RE_x(G)\simeq \Sigma(G)$
as required. 
\end{proof}

\begin{lemma}\label{lem:lozin+reg}
Let $G=(V,E)$ be a graph and let $x\in V$ be given. Then $\reg(\LE_x(G;Y,Z))=\reg(\LE_x(G;Y',Z'))$
for any two distinct decompositions $\{Y,Z\}$ and $\{Y',Z'\}$ of $N_G(x)$. 
\end{lemma}
\begin{proof}
We proceed as in the proof of Lemma~\ref{lem:lozin+stable}. So, let a decomposition $N_G(x)=Y\cup Z$ and a vertex $u\in Z$ be given. We claim that $\reg(\LE_x(G;Y,Z))=\reg(\LE_x(G;Y\cup \{u\},Z\backslash \{u\})$. If we set $H:=\LE_x(G;Y,Z)\cup uy$, then the edge $uy$ is an isolating edge of $H$ with respect to the vertex $b$, while $b$ has a neighbour $a$ of degree two. Therefore, Proposition~\ref{prop:reg+isolating} applies, that is, $\reg(H)=\reg(\LE_x(G;Y,Z))$.
On the other hand, $uz$ is also an isolating edge in $H$ with respect to $a$, and the vertex $a$ has a degree two neighbour, namely $b$. Thus we have $\reg(H)=\reg(H-uz)=\reg(\LE_x(G;Y\cup \{u\},Z\backslash \{u\})$. So the claim follows.
\end{proof}

\begin{proof}[{\bf Proof of Theorem~\ref{thm:lozin+reg}}]
Given a subset $S\subseteq V$. Assume first that $x\notin S$. If we define $S':=S\cup \{a,b\}$,
it follows that $\LE_x(G)[S']\cong G[S]\cup K_2$, where $K_2$ is induced by the edge $ab$.
Therefore, the graph $\LE_x(G)[S']$ is homotopy equivalent to the suspension of $G[S]$.

Suppose now that $x\in S$. In such a case, we consider $S':=(S\backslash \{x\})\cup \{y,a,b,z\}$, and write 
$H:=G[S]$ and $H':=\LE_x(G)[S']$. Note that the Lozin transform $\LE_x(H)$ of $H$ with respect to the partition $N_H(x)=(S\cap Y)\cup (S\cap Z)$ is exactly isomorphic to $H'$. Therefore, we have $H'\simeq \Sigma(H)$ by Proposition~\ref{prop:lozin+stable}, which implies that $\reg(\LE_x(G))\geq \reg(G)+1$.

For the converse, by Lemma~\ref{lem:lozin+reg}, it is sufficient to show that $\reg(\LE_x(G;N_G(x),\emptyset))\leq \reg(G)+1$. So,  set $\LE_x(G)=\LE_x(G;N_G(x),\emptyset)$. Then we have
\begin{equation*}
\reg (\LE_x(G))\leq \max\{\reg (\LE_x(G)-a), \reg(\LE_x(G)-N_{\LE_x(G)}[a])+1\}.
\end{equation*}
by Corollary~\ref{cor:induction-sc}. However, the graph $\LE_x(G)-a$ is isomorphic to $G\cup bz$ so that $\reg(\LE_x(G)-a)=\reg(G)+1$. Moreover, the graph $\LE_x(G)-N_{\LE_x(G)}[a]$ is isomorphic to $(G-x)\cup \{z\}$ in which $z$ is an isolated vertex. It means that $\reg(\LE_x(G)-N_{\LE_x(G)}[a])=\reg(G-x)\leq \reg(G)$. Therefore, we conclude that
$\reg(\LE_x(G))\leq \reg(G)+1$ as claimed.
\end{proof}

\begin{remark}\label{rem:lozin}
Note that Lozin transformation does not need to preserve the vertex-decomposability of a graph in general. Furthermore, other graph invariants such as the matching number or cochordal cover number may not need to increase exactly by one under a Lozin transformation (see Lemma~\ref{lem:triple-cd}). 
 \end{remark}
We close this section with the proof of the fact that any non-trivial Lozin operation preserves the primeness of a graph. We remark that in the trivial case, the primeness of a graph is not preserved under a Lozin transformation. For instance, the graph $\LE_x(C_8;N_{C_8}(x),\emptyset)$ is not prime for any vertex $x\in V(C_8)$.

\begin{lemma}\label{lem:edge-remove-1}
If $e$ is an edge of a prime graph $G$, then $\reg(G-e)\leq \reg(G)$.
\end{lemma}
\begin{proof}
Let $\reg(G)=k$. We write $e=xy$ and $H:=G-e$, and consider two cases separately.

{\it Case $1$.} $x$ is a prime vertex of $H$. It then follows that
$\reg(H)=\reg(H-N_H[x])+1$, while the graph $H-N_H[x]$ is isomorphic to
$G-(N_G[x]\setminus \{y\})$; hence, we have $\reg(H)=\reg(G-(N_G[x]\setminus \{y\}))+1\leq (k-1)+1=k$, since $G$ is prime.

{\it Case $2$.} $x$ is not a prime vertex of $H$, that is, $\reg(H)=\reg(H-x)$. It then follows $\reg(H)=\reg(H-x)=\reg(G-x)< \reg(G)$, since $H-x\cong G-x$.
\end{proof}

\begin{lemma}\label{lem:edge-remove-2}
If $G$ is a prime graph with $\delta(G)\geq 2$, then $\reg(G-N_G[e])\leq \reg(G)-2$.
\end{lemma}
\begin{proof}
Let $e=xy$ be an edge of $G$, and assume that $\reg(G)=k$. Observe that if we define $D(x,y):=(N_G[x]\cup N_G[y])\setminus \{x,y\}$, then $\reg(G-D(x,y))\leq k-1$, since $G$ is prime. However, we have $G-D(x,y)\cong (G-N_G[e])\cup K_2$, where the $K_2$ component is induced by the edge $e$; hence, $\reg(G-N_G[e])=\reg(G-D(x,y))-1\leq k-2$.
\end{proof}

\begin{proposition}\label{prop:tsub-prime}
Let $G$ be a prime graph with $\delta(G)\geq 2$. If $e=xy$ is an edge of $G$, then the graph $\LE_x(G;(N_G(x)\setminus \{y\}), \{y\})$ 
is prime.
\end{proposition}
\begin{proof}
If we write $H:=\LE_x(G;(N_G(x)\setminus \{y\}), \{y\})$, the graph $H$ is obtained from $G$ by replacing the edge $e=xy$ with the path $x-x_0-x_1-x_2-y$. Suppose that $\reg(G)=k$ so that $\reg(H)=k+1$.

{\it Case $1$.} Consider the graph $H-z$ for any vertex $z\in V(H)\setminus \{x,x_0,x_1,x_2,y\}$. Note that we have $\LE_x(G-z;(N_G(x)\setminus \{y\}), \{y\})\cong H-z$; hence,
$\reg(H-z)\leq k$, since $\reg(G-z)<k$ by the primeness of $G$.

{\it Case $2$.} The cases obtained by the removal of vertices $x$ and $y$ from $H$ respectively are of typical nature, so we may only consider the graph $H-x$. However, the graph $H-x$ is isomorphic to the graph
$\LE_y(G-x; N_{G-x}(y),\emptyset)$. Therefore we conclude that $\reg(H-z)\leq k$.

{\it Case $3$.} Once again, note that the cases $H-x_0$ and $H-x_2$ are similar, so we only consider one of them. For the graph $H-x_0$, observe that $N_{H-x_0}(x_1)\subseteq N_{H-x_0}(y)$ such that $\deg_{H-x_0}(x_1)=1$ and
$\deg_{H-x_0}(x_2)=2$; hence, we can remove the vertex $y$ form $H-x_0$ without altering its regularity by Lemma~\ref{lem:remove-deg}, that is,  $\reg(H-x_0)=\reg(G-y)+1\leq (k-1)+1=k$.

{\it Case $4$.} Finally, we consider the graph $H_x:=H-x_1$.
We verify this case by analysing two subcases separately. Observe that since
$x_0$ has degree one in $H_x$, we have either $\reg(H_x)=\reg(H_x-x_0)$
or else $\reg(H_x)=\reg(H_x-N_{H_x}[x])+1$ by Lemma~\ref{lem:deg-one}.

{\it Subcase $4.1$.} We consider the equality $\reg(H_x)=\reg(H_x-N_{H_x}[x])+1$. For the graph
$T_1:=H_x-N_{H_x}[x]$, note that we have either $\reg(T_1)=\reg(T_1-x_2)$
or $\reg(T_1)=\reg(T_1-N_{T_1}[y])+1$ by again Lemma~\ref{lem:deg-one}.
If $\reg(T_1)=\reg(T_1-x_2)$, then $T_1-x_2\cong G-(N_G[x]\setminus \{y\})$ so that $\reg(T_1)=\reg(G-(N_G[x]\setminus \{y\}))\leq (k-1)$, that is, $\reg(H_x)\leq k$. We should note at this point that the sets $N_G[x]$
and $N_G[y]$ are not comparable with respect to the inclusion by Proposition~\ref{prop:dominated}. 
So, suppose that $\reg(T_1)=\reg(T_1-N_{T_1}[y])+1$, then $T_1-N_{T_1}[y]\cong G-N_G[e]$ so that $\reg(H_x)=\reg(G-N_G[e])+2\leq (k-2)+2=k$ by Lemma~\ref{lem:edge-remove-2}.

{\it Subcase $4.2$.} Assume next that $\reg(H_x)=\reg(H_x-x_0)$. Now, for the graph $T_2:=H_x-x_0$, we have either $\reg(T_2)=\reg(T_2-x_2)$ or $\reg(T_2)=\reg(T_2-N_{T_2}[y])+1$ depending whether $x_2$ is a prime vertex or not in $T_2$.
If $\reg(T_2)=\reg(T_2-x_2)$, then $\reg(T_2)=\reg(G-e)\leq k$ by Lemma~\ref{lem:edge-remove-1}, since $T_2-x_2\cong G-e$. On the other hand,
if $\reg(T_2)=\reg(T_2-N_{T_2}[y])+1$, then $T_2-N_{T_2}[y]\cong G-(N_G[y]\setminus \{x\})$ so that $\reg(T_2)=\reg(G-(N_G[y]\setminus \{x\})+1\leq (k-1)+1=k$.

\end{proof}

\begin{theorem}\label{thm:lozin-prime}
If $G$ is a prime graph with $\delta(G)\geq 2$, then so is $\LE_x(G;Y,Z)$ for any vertex $x$ and any partition $N_G(x)=Y\cup Z$ such that $Y,Z\neq \emptyset$.
\end{theorem}
\begin{proof}
By Proposition~\ref{prop:tsub-prime}, we may assume that both sets $Y$ and
$Z$ are of size at least two.

Now assume that $\reg(G)=k$ so that $\reg(\LE_x(G;Y,Z))=k+1$. To ease the notation, we write $H:=\LE_x(G;Y,Z)$.

{\it Case $1$.} For the graph $H-y$, we note that $N_{H-y}(a)\subseteq N_{H-y}(z)$ such that $\deg_{H-y}(a)=1$ and $\deg_{H-y}(b)=2$ so that we can remove the vertex $z$ without altering the regularity by Lemma~\ref{lem:remove-deg}. Since the resulting graph is isomorphic to $(G-x)\cup K_2$, we therefore have $\reg(H-y)=\reg(G-x)+1\leq (k-1)+1=k$. The case of the graph $H-z$ is identical to this case.

{\it Case $2$.} For some vertex $w\in Y$, consider the graph $H-w$. We should note at this point that $Y\setminus \{w\}\neq \emptyset$, since $|Y|\geq 2$. However, this graph is isomorphic to $\LE_x(G-w;Y\setminus \{w\},Z)$ so that
$\reg(H-w)=\reg(G-w)+1\leq k$. Once again, the case $H-u$ for any vertex $u\in Z$ can be similarly treated. 

{\it Case $3$.} Consider the graph $H-a$. Suppose that $\reg(H-a)=m$, and let $S\subseteq V(H-a)$ be a minimal subset satisfying $\widetilde{H}_{m-1}((H-a)[S])\neq 0$. We analyse two cases separately. 

{\it Subcase $3.1$.} The vertex $b$ is a prime vertex of $H-a$, that is, $b\in S$. It then follows that $z\in S$, since $\deg_{H-a}(b)=1$. Furthermore, if
$u\in S\cap Z$, then $N_{(H-a)[S]}(b)\subseteq N_{(H-a)[S]}(u)$ so that
$(H-a)[S]\simeq (H-a)[S]-u$ by Corollary~\ref{thm:hom-induction}, which is not possible by the minimality of $S$. We may therefore assume that $S\cap Z=\emptyset$. However, this forces that $\deg_{(H-a)[S]}(z)=1$; hence,
$(H-a)[S]\cong G[S]\cup K_2$, where we identify the vertex $y\in V(H-a)$ with $x$ and the component $K_2$ is induced by the edge $bz$. But then the graph $G[S]$ is an induced subgraph of $G-Z$ so that
$m\leq \reg(G-Z)+1\leq (k-1)+1=k$, since $G$ is prime.

{\it Subcase $3.2$.} The vertex $b$ is not a prime vertex of $H-a$, that is, $\reg(H-a)=\reg(H-\{a,b\})$. We write $T:=H-\{a,b\}$, and consider the possibility of primeness of the vertices $y$ and $z$ in $T$. If $y$ is a prime vertex of $T$, then $\reg(T)=\reg(T-N_T[y])+1$, while the graph
$T-N_T[y]$ is isomorphic to $G-Y$ so that $\reg(T-N_T[y])=\reg(G-Y)\leq k-1$, that is, $\reg(T)\leq k$. On the other hand, if $y$ is not a prime vertex in $T$, then $\reg(T)=\reg(T-y)$. Now, if $z$ is prime in $T-y$, then we conclude that $\reg(T-y)=\reg((T-y)-N_{(T-y)}[z])+1$, where the graph
$(T-y)-N_{(T-y)}[z]$ is isomorphic to $G-(Z\cup \{x\})$ so that
$\reg(T)=\reg(G-(Z\cup \{x\})+1\leq (k-1)+1=k$. Finally, if $z$ is not prime in $T-y$, then we have $\reg(T)=\reg(T-y)=\reg(T-\{y,z\})=\reg(G-x)\leq k-1$.
 
\end{proof}

\section{Contractions, expansions and the virtual induced matching number}\label{sect:contr-expan}
In this section, we introduce a new lower bound to the regularity of graphs,
the virtual induced matching number $\vim(G)$, which is larger than or equals to the ordinary induced matching number $\im(G)$ in general, while we provide an example showing that the gap between $\vim(G)$ and $\im(G)$ could be arbitrarily large. For that purpose, we first begin to describe the behaviour of the regularity under vertex expansions and edge contractions on the independence complexes of graphs.

The notion of an edge-contraction on a simplicial complex seems firstly studied by Hoppe~\cite{HH}, and it was later investigated in the language of independence complexes by Ehrenborg and Hetyei~\cite{EH} (see also~\cite{ALS, BC1}). 

Let $\D$ be a simplicial complex on the vertex set $V=V(\D)$, and consider $x,y\in V$ such that $\{x,y\}\in \D$ 
and $w_{xy}\notin V$. The \emph{edge contraction} $xy\mapsto w_{xy}$ can be defined to be a map $f\colon V\to (V\backslash \{x,y\})\cup \{w_{xy}\}$ by
\begin{equation*}
f(v):=\begin{cases}
v, & \textrm{if}\;v\notin \{x,y\},\\
w_{xy}, & \textrm{if}\;v\in \{x,y\}.
\end{cases}
\end{equation*} 
We then extend $f$ to all simplices $F=\{v_0,v_1,\ldots,v_k\}$ of $\D$ by setting $$f(F):=\{f(v_0),f(v_1),\ldots,f(v_k)\}.$$
The simplicial complex $\D_{xy}:=\{f(F)\colon F\in \D\}$ is called the contraction of $\D$ with respect to the edge $\{x,y\}$.

Observe that the contraction of an edge may not need to preserve the homotopy type of the complex in general. However, under a suitable restriction on the contracted edge, we can guarantee this to happen. We say that an edge
$\{x,y\}$ is a \emph{contractible-edge} in $\D$ if no minimal non-face of $\D$ contains it.

\begin{theorem}\citep[Theorem 2.4]{EH}, \citep[Theorem 1]{ALS}\label{thm:hmtpy-edge}
Let $\D$ be a simplicial complex and let $\{x,y\}\in \D$ be an edge. Then, the simplicial complexes $\D$ and $\D_{xy}$ are homotopy equivalent provided that the edge $\{x,y\}$ is a contractible-edge. 
\end{theorem}

We note that the required condition in Theorem~\ref{thm:hmtpy-edge} is equivalent to $\lk_{\D}(x)\cap \lk_{\D}(y)=\lk_{\D}(\{x,y\})$. Furthermore, if $\D=\Ind(G)$ for some graph $G$, then the contraction of any edge in $\Ind(G)$ preserves the homotopy type.

In graph's side, the edge contraction is one of the fundamental operations in graph minor theory.

\begin{definition}
Let $e=xy$ be an edge of a graph $G$. Then the \emph{contraction} of $e$ on $G$ is the graph $G/e$ defined by
$V(G/e)=(V(G)\setminus\{x,y\})\cup \{w\}$ and $E(G/e)=E(G-\{x,y\})\cup \{wz\colon z\in N_G(x)\cup N_G(y)\}$. 
\end{definition}

In this section, we mainly deal with edge contractions on the independence complexes in which we choose to describe the operation on graphs, while their equivalence is ensured by Lemma~\ref{lem:ind-contract-expand}. Furthermore, the relation between $\reg(G/e)$ and $\reg(G)$ for a graph $G$ and any edge $e\in E(G)$ will be investigated in Section~\ref{sect:edge-subdiv} (see Theorem~\ref{thm:contract-reg}).

\begin{definition}
Let $G$ be a graph and let $u$ and $v$ be two non-adjacent vertices of $G$. We define the \emph{fake-contraction}
(or \emph{$\fp$-contraction}) $\fp(G;uv)$ of $G$ with respect to $u$ and $v$ to be the graph $(G\cup uv)/uv$, where the graph $G\cup uv$ is obtained from $G$ by the addition of the edge $uv$ to $G$. 
\end{definition}

We need to be sure that after an edge contraction on $\IE(G)$, the resulting complex is still the independence complex of a graph.

\begin{definition}
Let $G=(V,E)$ be a graph. Two non-adjacent vertices $\{x,y\}$ in $G$ is called a \emph{genuine-pair} (or simply a \emph{$\gp$-pair}) if there exist no vertices $u,v\in V\backslash \{x,y\}$ such that $G[\{x,y,u,v\}]\cong 2K_2$. When $\{x,y\}$ is a $\gp$-pair in $G$, the graph $\gp(G;xy)$ constructed by $V(\gp(G;xy)):=(V\backslash \{x,y\})\cup \{w\}$ and $E(\gp(G;xy)):=E(G-\{x,y\})\cup \{uw\colon u\in N_G(x)\cap N_G(y)\}$ is called the \emph{$\gp$-contraction} of $G$ with respect to the pair $\{x,y\}$.
\end{definition}

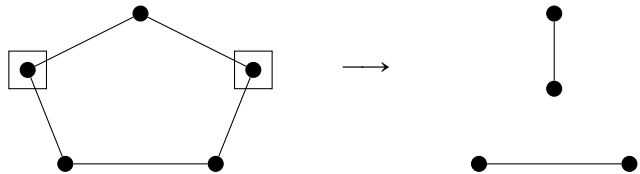
\begin{figure}[ht]
\begin{center}
\begin{tikzpicture}[scale=1]

\node [noddee] at (0,0) (v1)  {};
\node [noddee] at (2,0) (v2)  {}
	edge [] (v1);
\node [noddee] at (2.5,1.25) (v3)  {}
	edge [] (v2);
\node [noddee] at (1,2) (v4) {}
	edge [] (v3);
\node [noddee] at (-0.5,1.25) (v5) {}
	edge [] (v1)
	edge [] (v4);
			
\draw (-0.25,1) rectangle (-0.75,1.5);

\draw (2.25,1) rectangle (2.75,1.5);

\node at (4,1.25) (v5) {$\longrightarrow$};

\node [noddee] at (5.5,0) (t1)  {};
\node [noddee] at (7.5,0) (t2)  {}
	edge [] (t1);
\node [noddee] at (6.5,1) (t3)  {};
\node [noddee] at (6.5,2) (t4) {}
	edge [] (t3);
	
\end{tikzpicture}

\end{center}
\caption{A $\gp$-contraction.}
\label{contract1}
\end{figure}

\begin{lemma}\label{lem:ind-contract-expand}
If $\{x,y\}$ is a $\gp$-pair in $G$, then the simplicial complexes $\Ind(\gp(G;xy))$ and  $\Ind(G)_{xy}$ are isomorphic.
\end{lemma}
\begin{proof}
We may identify the vertex, say $w_{xy}$ to which the pair $\{x,y\}$ is contracted in both cases so that complexes 
$\Ind(\gp(G;xy))$ and  $\Ind(G)_{xy}$ have identical vertex sets. We show that both complexes have the same set of minimal non-faces containing the distinguished vertex $w_{xy}$, so the required isomorphism follows. Let $R$ be a minimal non-face of $\Ind(\gp(G;xy))$ containing $w_{xy}$. Since $\Ind(\gp(G;xy))$ is the independence complex of a graph,
this means that there exists a vertex $z\in V(\gp(G;xy))$ such that $R=\{z,w_{xy}\}$; hence, we have $zw_{xy}\in E(\gp(G;xy))$, which in turn implies that $z\in N_G(x)\cap N_G(y)$. However, this forces that $\{z,x\}$ and $\{z,y\}$ are not faces of $\Ind(G)$ so that $\{z,w_{xy}\}$ can not be a face in $\Ind(G)_{xy}$.

Suppose now that $L$ is a minimal non-face of $\Ind(G)_{xy}$ containing $w_{xy}$. We first verify that the size of $L$ must be two. Assume to the contrary that $L=\{w_{xy},u_1,\ldots,u_k\}$ such that $u_i\neq u_j$ for any $i,j\in [k]$ and $k>1$. It then follows that neither of the sets $\{x,u_1,\ldots,u_k\}$ and $\{y,u_1,\ldots,u_k\}$ forms a face in $\Ind(G)$, since $L$ is not a face in $\Ind(G)_{xy}$. Thus, there exist $i,j\in [k]$ such that $xu_i, yu_j\in E(G)$ while $xu_j,yu_i\notin  E(G)$, where the latter is due to the fact that $\{u_i,w_{xy}\}$ and $\{u_j,w_{xy}\}$ are faces in $\Ind(G)_{xy}$. On the other hand, since $L$ is minimal, we must have $\{u_i,u_j\}\in \Ind(G)_{xy}$, that is, $u_iu_j\notin E(G)$. However, we then have $G[x,y,u_i,u_j]\cong 2K_2$ which contradicts to $\{x,y\}$ being a $\gp$-pair. Therefore, the set $L$ must be of size two, say $L=\{w_{xy},u\}$. This implies that $u\notin N_G(x)\cap N_G(y)$ so that $uw_{xy}\notin E(\gp(G;xy))$; hence, $L$ is a minimal non-face in $\Ind(\gp(G;xy))$.
\end{proof}

We next introduce an operation that could be considered as the reversal of a $\gp$-contraction.

Let $G=(V,E)$ be a graph. We say that two disjoint (possibly empty) subsets $A$ and $B$ of $V$ constitute a 
\emph{complete-pairing} in $G$, denoted by $[A,B]$, if $ab\in E$ for any $a\in A$ and $b\in B$. 
We note that $[A,\emptyset]$ or $[\emptyset,B]$ is always a complete-pairing in $G$. We say that
a complete-pairing $[A,B]$ is \emph{non-trivial} provided that the sets $A$ and $B$ are both non-empty. We omit the proof of the following easy fact.

\begin{lemma}\label{lemma:genuine-pairing}
Let $x$ and $y$ be two  non adjacent vertices of $G$. Then, $\{x,y\}$ is a $\gp$-pair in $G$ if and only if 
$[N_G(x)\backslash N_G(y), N_G(y)\backslash N_G(x)]$ is a complete-pairing in $G$.
\end{lemma}

\begin{definition}
Let $z$ be any vertex of $G$ and let $[A_z,B_z]$ be a complete-pairing in $G-N_G[z]$. Then the \emph{$\gp$-expansion} $\gp(G;z,A_z,B_z)$ of $G$ with respect to the vertex $z$ and the pairing $[A_z,B_z]$ is the graph constructed by $V(\gp(G;z,A_z,B_z)):=(V\backslash \{z\})\cup \{x_z,y_z\}$ and $E(\gp(G;z,A_z,B_z)):=E(G-z)\cup \{ux_z, uy_z\colon u\in N_G(z)\}\cup \{ax_z\colon a\in A_z\}\cup \{by_z\colon b\in B_z\}$.
\end{definition}
\begin{figure}[ht]
\begin{center}
\begin{tikzpicture}[scale=1.25, >=triangle 45]

\node [noddee] at (0,0) (v1)  {};
\node [noddee] at (1,0) (v2)  {}
	edge [] (v1);
\node [noddee] at (0,1) (v3)  {};
\node [noddee] at (1,1) (v4) {}
	edge [] (v3);


\node at (2,0.5) (v5) {$\longrightarrow$};	
\draw (-0.25,0.75) rectangle (0.25,1.25);
\node at (-0.4,1) (v5) {$a$};

\node [noddee] at (3,0) (t1)  {};
\node [noddee] at (4,0) (t2)  {}
	edge [] (t1);
\node [noddee] at (3,0.75) (t3)[label=below right:$y_a$]  {};
\node [noddee] at (3,1.25) (t4)[label=above right:$x_a$] {};
\node [noddee] at (4,1) (t5) {}
	edge [] (t3)
	edge [] (t4);

\node at (5,0.5) (t5) {$\longrightarrow$};

\node [noddee] at (6,0) (k1)  {};
\node [noddee] at (7,0) (k2)  {}
	edge [] (k1);
\node [noddee] at (6,0.75) (k3) {}
	edge [thick, dashed] (k2);
\node [noddee] at (6,1.25) (k4) {};
\node [noddee] at (7,1) (k5) {}
	edge [] (k3)
	edge [] (k4);

\draw[thick, densely dotted] (6,1.25) arc  (-60:60:-0.75);

\end{tikzpicture}

\end{center}
\caption{A $\gp$-expansion}
\label{expand1}
\end{figure}

When there is no confusion, we abbreviate $\gp(G;z,A_z,B_z)$ to $\gp(G;z)$, and note that the pair $\{x_z,y_z\}$ is always a $\gp$-pair in $\gp(G;z,A_z,B_z)$. Furthermore, the $\gp$-contraction and $\gp$-expansion operations are inverse to each other. In other words, if $w$ is the vertex for which the $\gp$-pair $\{x,y\}$ is contracted, then $\gp(\gp(G;xy); w, A_w,B_w)$ is isomorphic to $G$, where $A_w=N_G(x)\backslash N_G(y)$ and $B_w=N_G(y)\backslash N_G(x)$. Similarly, 
the graph $\gp(\gp(G;z,A_z,B_z);x_zy_z)$ is isomorphic to $G$.

\begin{corollary}\label{cor:ind-contract-expand}
Let $G$ be a graph.
\begin{itemize}
\item[$(i)$] $\Ind(G)\simeq \Ind(\gp(G;xy))$ for any $\gp$-pair $\{x,y\}$ in $G$.\\
\item[$(ii)$] $\Ind(G)\simeq \Ind(\gp(G;z))$ for any expansion $\gp(G;z)$ of $G$.
\end{itemize}
\end{corollary}
\begin{proof}
The part $(i)$ is the consequence of Lemma~\ref{lem:ind-contract-expand} and Theorem~\ref{thm:hmtpy-edge}, and the second follows from the part $(i)$.
\end{proof}

\begin{corollary}\label{cor:ind-contract-contract}
If $\{x,y\}$ is a $\gp$-pair in $G$ such that $N_G(x)\cap N_G(y)=\emptyset$, then $\Ind(G)$ is contractible.
\end{corollary}
\begin{proof}
In such a case, the vertex $w_{xy}$ to which the pair $\{x,y\}$ is contracted is an isolated vertex of $\gp(G;xy)$; hence, the complex $\Ind(\gp(G;xy))$ is contractible, so is $\Ind(G)$ by Corollary~\ref{cor:ind-contract-expand}.
\end{proof}


Our next result describes the effect of a $\gp$-contraction (respectively, a $\gp$-expansion) on a graph to its regularity.

\begin{proposition}\label{prop:reg-contract}
Let $\{x,y\}$ be a $\gp$-pair in $G$, and let $z$ be an non-isolated vertex of $G$ such that $[A_z,B_z]$ is a complete-pairing in $G-N_G[z]$. Then the followings hold:
\begin{itemize}
\item[$(i)$] $\reg(G)\geq \reg(\gp(G;xy))\geq \reg(G)-1$.\\
\item[$(ii)$] $\reg(\gp(G;z))\geq \reg(G)\geq \reg(\gp(G;z))-1$.
\end{itemize}
\end{proposition}
\begin{proof}
We only prove the case $(i)$, since the other case follows from that.
Now, assume first that $\reg(\gp(G;xy))=n$, and let $S\subseteq V(\gp(G;xy))$ be a subset such that $\widetilde{H}_{n-1}(\gp(G;xy)[S])\neq 0$. If $w_{xy}\notin S$, then $S\subseteq V(G)$ so that $\widetilde{H}_{n-1}(G[S])\neq 0$; hence, $\reg(G)\geq n$. So, suppose that $w_{xy}\in S$. If we let $S^*:=(S\setminus \{w_{xy}\})\cup \{x,y\}$, then $\{x,y\}$ is a $\gp$-pair in $G[S^*]$. Therefore, we have $G[S^*]\simeq (\gp(G[S^*]);xy)$ by Corollary~\ref{cor:ind-contract-expand}. However, the latter graph is clearly isomorphic to $\gp(G;xy)[S]$, which implies that $\widetilde{H}_{n-1}(G[S^*])\neq 0$, that is, $\reg(G)\geq n$.

For the second inequality, suppose that $\reg(G)=m$, and let $R\subseteq V(G)$ be a minimal subset satisfying $\widetilde{H}_{m-1}(G[R])\neq 0$. Note first that if $\{x,y\}\subset R$, then for the set $R^*:=(R\setminus \{x,y\})\cup \{w_{xy}\}$, we have $\widetilde{H}_{m-1}(\gp(G;xy)[R^*])\neq 0$ by the fact that $G[R]\simeq \gp(G;xy)[R^*]$. On the other hand, if $R\cap \{x,y\}=\emptyset$, then we necessarily have $R\subseteq V(\gp(G;xy))$ so that $\widetilde{H}_{m-1}(\gp(G;xy)[R])\neq 0$. It follows that $\reg(\gp(G;xy))\geq m$ in either case. 

Therefore, we are left to verify the case where exactly one of $x$ and $y$ belongs to the set $R$. Assume without loss of generality that $x\in R$, $y\notin R$. If we consider the associated Mayer-Vietoris exact sequence of the pair $(G[R],x)$;
\begin{equation*}
\cdots \to \widetilde{H}_{m-1}(G[R]-x) \to \widetilde{H}_{m-1}(G[R])\to \widetilde{H}_{m-2}(G[R]-N_{G[R]}[x])\to \cdots, 
\end{equation*}
it follows that $\widetilde{H}_{m-1}(G[R]-x)=0$ by the minimality of $R$ so that $\widetilde{H}_{m-2}(G[R]-N_{G[R]}[x])\neq 0$. If we set $L:=V(G[R]-N_{G[R]}[x])$, the graph $\gp(G;xy)[L]$ is isomorphic to $G[R]-N_{G[R]}[x]$. Thus, we have
$\widetilde{H}_{m-2}(\gp(G;xy)[L])\neq 0$ so that $\reg(\gp(G;xy))\geq m-1$ as claimed.
\end{proof}

\begin{remark}\label{rem:contract-genuine-reg}
We note that the first inequality in $(i)$ of Proposition~\ref{prop:reg-contract} could be strict, that is, there exist graphs with $\gp$-pairs such that their $\gp$-contractions may decrease the regularity by one. For instance, if we consider the graph $$H=(\{a,b,c,d,x,y\},\{ab,ac,bc,cd,cx,dy\}),$$ the pair $\{x,y\}$ is genuine, while
$\reg(H)=2$ and $\reg(\gp(H;xy))=1$. 
\end{remark}
We therefore look for a condition on a $\gp$-pair for which the regularity remains stable under its contraction.  

\begin{definition}
We call a $\gp$-pair $\{x,y\}$ in $G$ as a \emph{true-pair} (or simply a \emph{$\tp$-pair}) if there exists a vertex $u\in N_G(x)\cap N_G(y)$ with $N_G[u]\subseteq N_G[x]\cup N_G[y]$, and a $\gp$-contraction on a graph with respect to a $\tp$-pair is called a \emph{$\tp$-contraction} of $G$ and denoted by $\tp(G;xy)$. When $\{x,y\}$ is a $\tp$-pair, such a vertex $u$ is called a \emph{true-neighbour} (or simply a \emph{$\tp$-neighbour}) of the pair $\{x,y\}$. Similarly, if $z$ is a non-isolated vertex in $G$, we call a complete pairing $[A_z,B_z]$ in $G-N_G[z]$ as a \emph{$\tp$-pairing} of $z$, if there exists a vertex $v\in N_G(z)$ with $N_G[v]\subseteq A_z\cup B_z\cup N_G[z]$. A $\gp$-expansion of a graph $G$ with respect to a vertex with a $\tp$-pairing of it is called a \emph{$\tp$-expansion} of $G$ and denoted by $\tp(G;z,A_z,B_z)$. 
\end{definition}
We note that if $[A_z,B_z]$ is a $\tp$-pairing of $z$ in $G-N_G[z]$, then $\{x_z,y_z\}$ is a $\tp$-pair in $\tp(G;z,A_z,B_z)$ having the vertex $v$ as a $\tp$-neighbour.

\begin{theorem}\label{thm:reg-contract}
If $\{x,y\}$ is a $\tp$-pair in $G$, then $\reg(G)=\reg(\tp(G;xy))$. 
\end{theorem}
\begin{proof}
Assume that $\reg(G)=m$, and let $R\subseteq V(G)$ be a minimal subset satisfying $\widetilde{H}_{m-1}(G[R])\neq 0$.
In view of Proposition~\ref{prop:reg-contract}, we may restrict the proof to the case where such a set $R$ satisfies
neither $\{x,y\}\subset R$ nor $R\cap \{x,y\}=\emptyset$. Suppose that $x\in R$ and $y\notin R$. We first claim that
$N_G(y)\cap R=\emptyset$. In order to verify that we define $T:=R\cup \{y\}$ and consider the Mayer-Vietoris exact sequence of the pair $(G[T],y)$;
\begin{equation*}
\cdots \to \widetilde{H}_{m-1}(G[T]-N_{G[T]}[y]) \to \widetilde{H}_{m-1}(G[T]-y)\to \widetilde{H}_{m-1}(G[T])\to \cdots, 
\end{equation*}
where $\widetilde{H}_{m-1}(G[T])=0$, since no such set contains both $x$ and $y$. It then follows that 
$\widetilde{H}_{m-1}(G[T]-N_{G[T]}[y])\neq 0$ by the exactness. However, since $R$ is a minimal set satisfying $\widetilde{H}_{m-1}(G[R])\neq 0$, this implies that $N_{G[T]}[y]=\{y\}$, that is, $N_G(y)\cap R=\emptyset$. 

Now, since $\{x,y\}$ is a $\tp$-pair, they have a $\tp$-neighbour, say $u\in N_G(x)\cap N_G(y)$, 
satisfying $N_G(u)\subseteq N_G[x]\cup N_G[y]$. By the above claim, we have $u\notin R$. On the other hand, the minimality of $R$ also implies that $\widetilde{H}_{m-2}(G[R]-N_{G[R]}[x])\neq 0$. Therefore, if we set $K:=V(G[R]-N_{G[R]}[x])\cup \{u,w_{xy}\}$, it follows that the graph $\tp(G;xy)[K]$ is isomorphic to the suspension of $G[R]-N_{G[R]}[x]$ so that $\widetilde{H}_{m-1}(\tp(G;xy)[K])\neq 0$, that is, $\reg(\tp(G;xy))\geq m$. This completes the proof.
\end{proof}

\begin{corollary}\label{cor:reg-expand}
Let $z$ be a non-isolated vertex of a graph $G$ such that 
$[A_z,B_z]$ is a $\tp$-pairing in $G-N_G[z]$, then $\reg(G)=\reg(\tp(G;z,A_z,B_z))$.
\end{corollary}
\begin{proof}
The claim is the consequence of Theorem~\ref{thm:reg-contract},
since $\{x_z,y_z\}$ is then a $\tp$-pair in $\tp(G;z,A_z,B_z)$.
\end{proof}

Following the example provided in Remark~\ref{rem:contract-genuine-reg}, the induced matching number of a graph may not need to increase under the $\gp$-contraction of an arbitrary $\gp$-pair; however, if the $\gp$-pair is a $\tp$-pair, we have the following.

\begin{proposition}\label{prop:im-contract-expand}
Let $G=(V,E)$ be a graph.
\begin{itemize}
\item[$(i)$] If $\{x,y\}$ is a $\tp$-pair in $G$, then $\im(\tp(G;xy))\geq \im(G)$.
\item[$(ii)$] If $z$ is a non isolated vertex with a $\tp$-pairing $[A_z,B_z]$ in $G-N_G[z]$, 
then $\im(G)\geq \im(\tp(G;z,A_z,B_z))$.
\end{itemize}
\end{proposition}
\begin{proof}
Suppose first that $\{x,y\}$ is a $\tp$-pair in $G=(V,E)$, and let $M\subseteq E$ be a maximum induced matching in $G$.
Note that $|V(M)\cap \{x,y\}|\leq 1$, since $\{x,y\}$ is a $\tp$-pair, where $V(M)$ is the set of vertices of $G$
that are incident to an edge in $M$. If $V(M)\cap \{x,y\}=\emptyset$, then $M$ is an induced matching in $\tp(G;xy)$
so that $\im(G)\leq \im(\tp(G;xy))$. We may therefore assume that $|V(M)\cap \{x,y\}|=1$. Assume that $V(M)\cap \{x,y\}=\{x\}$, and let $u\in V$ be a $\tp$-neighbour of the pair $\{x,y\}$. If $xu\in M$, then $M':=(M\setminus \{xu\})\cup \{uw_{xy}\}$ is an induced matching in $\tp(G;xy)$. On the other hand, if $xz\in M$ for some $z\in N_G(x)\setminus N_G(y)$, then $(N_G(u)\setminus \{x,z\})\cap V(M)=\emptyset$; hence, the set $M^*:=(M\setminus \{xz\})\cup \{uw_{xy}\}$ is an induced matching in $\tp(G;xy)$. Thus, we have $\im(G)\leq \im(\tp(G;xy))$.

The second claim follows from the part $(i)$, since $\tp(\tp(G;z,A_z,B_z);x_zy_z)\cong G$, where $\{x_z,y_z\}$ is the $\tp$-pair of the $\tp$-expansion.
\end{proof}

\begin{remark}
We note that a graph $G$ contains no $\gp$-pair if and only if every edge of $G$ is contained in an induced $2K_2$ of $G$, that is, its complement $\overline{G}$ is a $4$-cycled graph (see~\cite{BC1} for details). Furthermore, a $\gp$-pair in a connected claw-free graph is always a $\tp$-pair. On the other hand, any graph $G$ containing a non-isolated vertex admits a $\tp$-expansion (possibly with respect to a half empty $\tp$-pairing).
\end{remark}

\begin{definition}
A graph $H$ is called a \emph{mate} of $G$ provided that there exists  a sequence of graphs $G=G_0, G_1,\ldots, G_k=H$ 
such that $G_{i+1}$ is obtained from $G_i$ by either a $\tp$-expansion or a $\tp$-contraction for each $0\leq i<k$. 
\end{definition} 
We remark that being a mate for two graphs defines an equivalence relation on the set of graphs, and we denote by $\ME_G$, 
the equivalence class of a given graph $G$. The following is the consequence of Theorem~\ref{thm:reg-contract} 
and Corollary~\ref{cor:reg-expand}.
\begin{corollary}\label{cor:reg-mate}
If $H\in \ME_G$, then $\reg(G)=\reg(H)$.
\end{corollary}

\begin{example}
We know that $\reg(C_8)=3$ and $\im(C_8)=2$. The graph $C_8$ contains no $\tp$-pair. However, if we first expand it truely and apply successive $\tp$-contractions, we reach $3K_2$ as depicted in Figure~\ref{C8expand-contract}. Therefore, we have $C_8\in \ME_{3K_2}$.
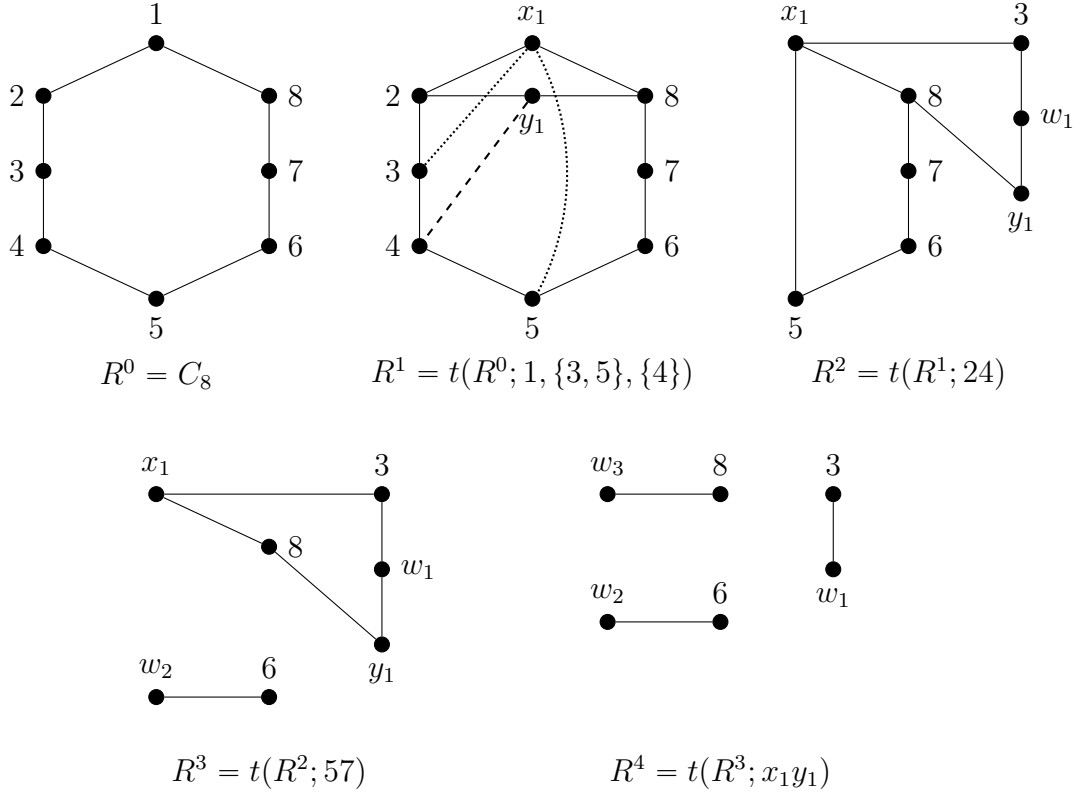
\begin{figure}[ht]
\begin{center}
\begin{tikzpicture}[scale=1]

\node [noddee] at (0,-0.2) (v1)[label=below:$5$]{};
\node [noddee] at (1.5,0.5) (v2)[label=right:$6$]{}
	edge [] (v1);
\node [noddee] at (1.5,1.5) (v3)[label=right:$7$]{}
	edge [] (v2);
\node [noddee] at (1.5,2.5) (v4)[label=right:$8$]{}
	edge [] (v3);
\node [noddee] at (0,3.2) (v5)[label=above:$1$]{}
	edge [] (v4);
\node [noddee] at (-1.5,2.5) (v6)[label=left:$2$]{}
    edge [] (v5);
\node [noddee] at (-1.5,1.5) (v7)[label=left:$3$]{}
     edge [] (v6);
\node [noddee] at (-1.5,0.5) (v8)[label=left:$4$]{}	
	 edge [] (v1)
	 edge [] (v7);
\node  at (0,-1.2) (v9)  {$R^0=C_8$};

\node [noddee] at (5,-0.2) (t1)[label=below:$5$]{};
\node [noddee] at (6.5,0.5) (t2)[label=right:$6$]{}
	edge [] (t1);
\node [noddee] at (6.5,1.5) (t3)[label=right:$7$]{}
	edge [] (t2);
\node [noddee] at (6.5,2.5) (t4)[label=right:$8$]{}
	edge [] (t3);
\node [noddee] at (5,3.2) (t5)[label=above:$x_1$]{}
	edge [] (t4);
\node [noddee] at (3.5,2.5) (t6)[label=left:$2$]{}
    edge [] (t5);
\node [noddee] at (3.5,1.5) (t7)[label=left:$3$]{}
     edge [] (t6)
     edge [thick, densely dotted] (t5);
\node [noddee] at (3.5,0.5) (t8)[label=left:$4$]{}	
	 edge [] (t1)
	 edge [] (t7);
\node [noddee] at (5,2.5) (t9)[label=below:$y_1$]{}
     edge [] (t4)
     edge [] (t6)
     edge [thick, dashed] (t8);	 
     	 
\draw[thick, densely dotted] (5,3.2) arc  (30:-30:3.45);
\node  at (5,-1.2) (t10)  {$R^1=t(R^0;1,\{3,5\},\{4\})$};

\node [noddee] at (8.5,-0.2) (a1)[label=below:$5$]{};
\node [noddee] at (10,0.5) (a2)[label=right:$6$]{}
	edge [] (a1);
\node [noddee] at (10,1.5) (a3)[label=right:$7$]{}
	edge [] (a2);
\node [noddee] at (10,2.5) (a4)[label=right:$8$]{}
	edge [] (a3);
\node [noddee] at (8.5,3.2) (a5)[label=above:$x_1$]{}
	edge [] (a1)
	edge [] (a4);
\node [noddee] at (11.5,3.2) (a6)[label=above:$3$]{}
	edge [] (a5);	 
\node [noddee] at (11.5,2.2) (a7)[label=right:$w_1$]{}
	edge [] (a6);
\node [noddee] at (11.5,1.2) (a8)[label=below:$y_1$]{}
	edge [] (a4)
	edge [] (a7);     	 
\node  at (10,-1.2) (a10)  {$R^2=t(R^1;24)$};

\node [noddee] at (1.5,-5.5) (b2)[label=above:$6$]{};
\node [noddee] at (0,-5.5) (b3)[label=above:$w_2$]{}
	edge [] (b2);
	
\node [noddee] at (1.5,-3.5) (b4)[label=right:$8$]{};
\node [noddee] at (0,-2.8) (b5)[label=above:$x_1$]{}
	edge [] (b4);
\node [noddee] at (3,-2.8) (b6)[label=above:$3$]{}
	edge [] (b5);	 
\node [noddee] at (3,-3.8) (b7)[label=right:$w_1$]{}
	edge [] (b6);
\node [noddee] at (3,-4.8) (b8)[label=below:$y_1$]{}
	edge [] (b4)
	edge [] (b7);     	 
\node  at (1.5,-6.5) (b10)  {$R^3=t(R^2;57)$};

\node [noddee] at (7.5,-4.5) (c2)[label=above:$6$]{};
\node [noddee] at (6,-4.5) (c3)[label=above:$w_2$]{}
	edge [] (c2);
	
\node [noddee] at (7.5,-2.8) (c4)[label=above:$8$]{};
\node [noddee] at (6,-2.8) (c5)[label=above:$w_3$]{}
	edge [] (c4);
\node [noddee] at (9,-2.8) (c6)[label=above:$3$]{};	 
\node [noddee] at (9,-3.8) (c7)[label=below:$w_1$]{}
	edge [] (c6);
     	 
\node  at (7.5,-6.5) (c10)  {$R^4=t(R^3;x_1y_1)$};

\end{tikzpicture}

\end{center}
\caption{A $\tp$-expansion and successive $\tp$-contractions of $C_8$.}
\label{C8expand-contract}
\end{figure}
\end{example}

\begin{definition}
The \emph{virtual induced matching number} $\vim(G)$ of a graph $G=(V,E)$ is defined by
\begin{equation*}
\vim(G):=\max \{\im(H)\colon H \;\textrm{is\;a\;mate\;of}\;G\}.
\end{equation*}
\end{definition} 

\begin{theorem}\label{thm:im-vim-reg}
$\im(G)\leq \vim(G)\leq \reg(G)$ for any graph $G$.
\end{theorem}
\begin{proof}
If $H$ is a mate of $G$, we know that $\reg(G)=\reg(H)$ by Corollary~\ref{cor:reg-mate}. However,
$\reg(H)\geq \im(H)$ by a result of Katzman~\cite{MK}, so maximizing over $H$ gives the desired result.
\end{proof}

\begin{remark}
Obviously, a graph $G$ may have infinitely many (non-isomorphic) mates, while
the set $\{\im(H)\colon H\;\textrm{is a mate of}\;G\}$ is finite, since $\reg(G)=\reg(H)\geq \im(H)$ for any mate $H$ of $G$. Moreover,
the equality $\vim(G)=\vim(H)$ holds whenever $H\in \ME_G$.
\end{remark}

We note that for the Morey-Villarreal graph (see Figure~\ref{fig:MV}), we have $\vim(G_{MV})=2<3=\reg_{\Z_2}(G_{MV})$, where the equality $\vim(G_{MV})=2$ is due to the fact that $\im(G_{MV})=2=\reg_{\Z_3}(G_{MV})$.  Furthermore, the graph $R_n$ depicted in Figure~\ref{fig-reg-im-1} shows that
the gap between the virtual induced matching and the induced matching numbers could be arbitrarily large, since $\vim(R_n)=\im(R_n)+n$ for any $n\geq 1$.

We do not know whether the virtual induced matching number is additive in general, however, we can state the following easy fact:

\begin{corollary}\label{cor:vim-add}
If $G$ and $H$ are two connected graphs on disjoint set of vertices, then
$\vim(G\cup H)\geq \vim(G)+\vim(H)$.
\end{corollary}

\begin{proposition}\label{prop:reg-cycles}
For any $n\geq 3$, we have $\reg(C_n)=\vim(C_n)=\lfloor \frac{n+1}{3}\rfloor$.
\end{proposition}
\begin{proof}
We write $V(C_n)=\{1,2,\ldots,n\}$, and assume that the edges of $C_n$ are in the cyclic fashion. Since the cases $3\leq n \leq 5$ are trivial, we let $n\geq 6$ and apply induction. The vertex $1$ admits a $\tp$-pairing in $L^0=C_n$, namely
$[\{3,5\}, \{4\}]$, and we define $L^1:=\tp(L^0;1,\{3,5\},\{4\})$ in which $\{2,4\}$ is a $\tp$-pair with the vertex $3$ 
is the $\tp$-neighbour. If we set $L^2:=\tp(L^1;24)$, where we denote by $w_1$ the vertex for which the pair $\{2,4\}$ is contracted, the pair $\{3,y_1\}$ becomes a $\tp$-pair in $L^2$ with $w_1$ is a $\tp$-neighbour. Now, the graph $L^3:=\tp(L^2;3y_1)$ is clearly isomorphic to $K_2\cup C_{n-3}$, where $V(K_2)=\{w_1,w_2\}$ and $w_2$ is the vertex for which $\{3,y_1\}$ is contracted. Therefore, 
\begin{equation*}
\reg(C_n)=1+\reg(C_{n-3})=1+\vim(C_{n-3})=\vim(C_n)=\lfloor \frac{n+1}{3}\rfloor,
\end{equation*}
where the first equality is due to Corollary~\ref{cor:reg-mate}, the second and last equalities follow from the induction, and the third equality is
the consequence of Theorem~\ref{thm:im-vim-reg} and Corollary~\ref{cor:vim-add}.
\end{proof}


\begin{proposition}\label{prop:deg2mate}
Let $\{a,b,c,d\}$ be a set of vertices of a $4$-path (not necessarily induced) in $G$ with edges
$ab,bc,cd$ such that $d_G(b)=d_G(c)=2$. If $ad\in E(G)$, then $G-\{a,b,c,d\}\cup K_2$ is a mate of $G$. Furthermore,
if $ad\notin E(G)$, then the graph $\fp(G-\{b,c\};da)\cup K_2$ is a mate of $G$. 
\end{proposition}
\begin{proof}
Suppose first that $ad\in E(G)$. We then apply to a $\tp$-expansion on the vertex $d$ with respect to the $\tp$-pairing $[\{b\},\emptyset]$. Observe that in the resulting graph $\tp(G;d)$, the pair $\{a,c\}$ is a $\tp$-pair with $b$ as a $\tp$-neighbour. The $\tp$-contraction of $\{a,c\}$ provides the graph $\tp(\tp(G;d);ac)$ in which $\{b,y_d\}$ is a $\tp$-pair. When we $\tp$-contract $\{b,y_d\}$, the set $\{x_d,w_{ac}, w_{by_d}\}$ in the graph we obtained induces a $3$-path such that the vertex $w_{by_d}$ is of degree one, that is, $\{x_d,w_{by_d}\}$ is a $\tp$-pair. Therefore the $\tp$-contraction of $\{x_d,w_{by_d}\}$ in $\tp(\tp(\tp(G;d);ac);by_d)$ results a graph isomorphic to $G-\{a,b,c,d\}\cup K_2$ (compare to Figure~\ref{fig:4-sq-path}).

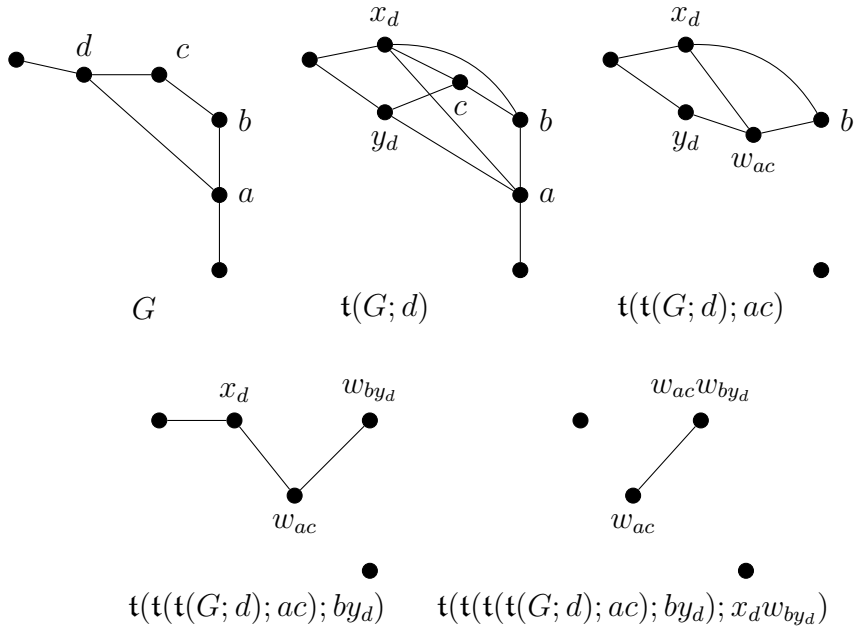
\begin{figure}[ht]
\begin{center}
\begin{tikzpicture}[scale=1]

\node [noddee] at (1,0) (v1)  {};
\node [noddee] at (1,1) (v2) [label=right:$a$] {}
	edge [] (v1);
\node [noddee] at (1,2) (v3) [label=right:$b$] {}
	edge [] (v2);
\node [noddee] at (0.2,2.6) (v4)[label=above right:$c$] {}
	edge [] (v3);
\node [noddee] at (-0.8,2.6) (v5) [label=above:$d$] {}
	edge [] (v2)
	edge [] (v4);
\node [noddee] at (-1.7,2.8) (v6)  {}
	edge [] (v5);
\node  at (0,-0.5) (v6)  {$G$};

\node [noddee] at (5,0) (t1)  {};
\node [noddee] at (5,1) (t2) [label=right:$a$] {}
	edge [] (t1);
\node [noddee] at (5,2) (t3) [label=right:$b$] {}
	edge [] (t2);
\node [noddee] at (4.2,2.5) (t4)[label=below:$c$] {}
	edge [] (t3);
\node [noddee] at (3.2,3) (t5) [label=above:$x_d$] {}
	edge [] (t2)
	edge [] (t4)
	edge [bend left]  (t3);
\node [noddee] at (2.2,2.8) (t6)  {}
	edge [] (t5);
\node [noddee] at (3.2,2.1) (t7)[label=below:$y_d$] {}
	edge [] (t2)
	edge [] (t4)	
	edge [] (t6);
\node  at (3.2,-0.5) (t8)  {$\tp(G;d)$};

\node [noddee] at (9,0) (r1)  {};
\node [noddee] at (9,2) (r3) [label=right:$b$] {};
\node [noddee] at (7.2,3) (r5) [label=above:$x_d$] {}
	edge [bend left]  (r3);
\node [noddee] at (6.2,2.8) (r6)  {}
	edge [] (r5);
\node [noddee] at (7.2,2.1) (r7) [label=below:$y_d$] {}	
	edge [] (r6);
\node [noddee] at (8.1,1.8) (r8) [label=below:$w_{ac}$] {}	
	edge [] (r3)
	edge [] (r5)
	edge [] (r7);
\node  at (7.4,-0.5) (r9)  {$\tp(\tp(G;d);ac)$};

\node [noddee] at (3,-4) (b1)  {};
\node [noddee] at (1.2,-2) (b5) [label=above:$x_d$] {};
\node [noddee] at (0.2,-2) (b6)  {}
	edge [] (b5);
\node [noddee] at (2,-3) (b8) [label=below:$w_{ac}$] {}	
	edge [] (b5);
\node [noddee] at (3,-2) (b3) [label=above:$w_{by_d}$] {}
	edge []  (b8);	
		
\node  at (1.5,-4.5) (b9)  {$\tp(\tp(\tp(G;d);ac);by_d)$};

\node [noddee] at (8,-4) (k1)  {};
\node [noddee] at (7.4,-2) (k5) [label=above:$w_{ac}w_{by_d}$] {};
\node [noddee] at (5.8,-2) (k6)  {};
\node [noddee] at (6.5,-3) (k8) [label=below:$w_{ac}$] {}	
	edge [] (k5);

\node  at (6.5,-4.5) (k9)  {$\tp(\tp(\tp(\tp(G;d);ac);by_d);x_dw_{by_d})$};		
	
\end{tikzpicture}

\end{center}
\caption{The first phase of expansions and contractions in Proposition~\ref{prop:deg2mate}.}
\label{fig:4-sq-path}
\end{figure}

Assume next that $ad\notin E(G)$. We apply to a $\tp$-expansion on the vertex $d$ with respect to the $\tp$-pairing
$[N_G(a)\setminus N_G(d), \{a\}]$ having the vertex $c$ as a $\tp$-neighbour. If we denote by $\{x_d,y_d\}$,
the resulting $\tp$-pair in $\tp(G;d)$, then $\{a,c\}$ is a $\tp$-pair in $\tp(G;d)$ with $\tp$-neighbour $b$. 
Once we contract this $\tp$-pair in $\tp(G;d)$ and denote the newly created vertex by $w_{ac}$, then $\{b, y_d\}$ becomes a $\tp$-pair in $\tp(\tp(G;d);ac)$ with a $\tp$-neighbour the vertex $w_{ac}$. Finally, the $\tp$-contraction of $\{b,y_d\}$ in $\tp(\tp(G;d);ac)$ yields a graph isomorphic to $\fp((G-\{b,c\});da)\cup K_2$, where the isolated edge is induced by $w_{ac}$ and $w_{by_d}$ (see to Figure~\ref{fig:4-path}).
\end{proof}

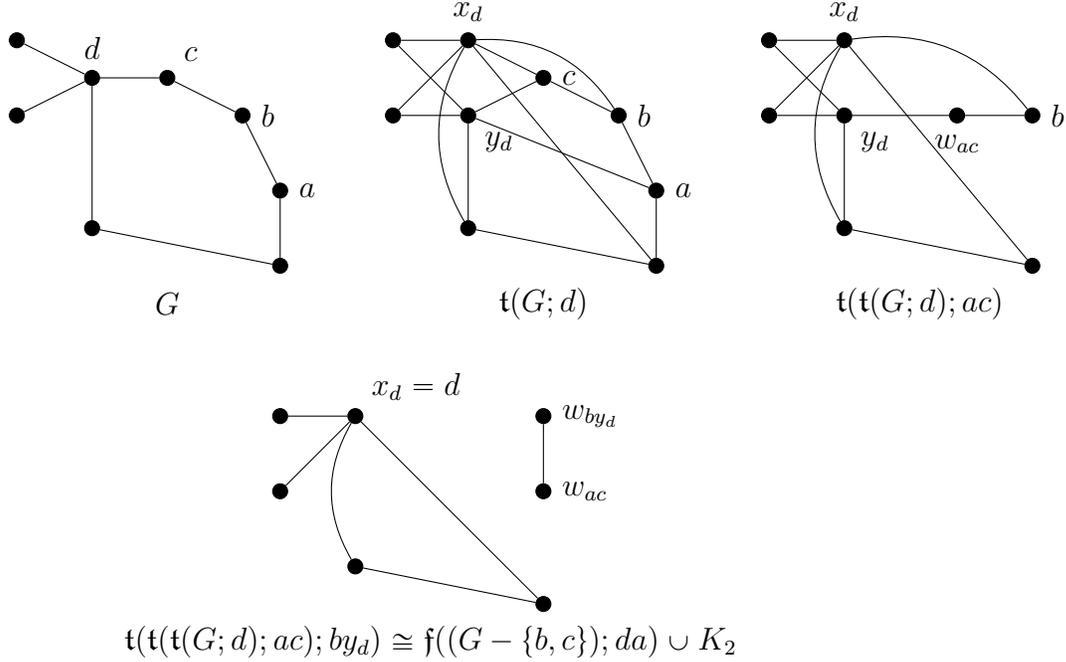
\begin{figure}[ht]
\begin{center}
\begin{tikzpicture}[scale=1]

\node [noddee] at (0,0) (v1)  {};
\node [noddee] at (0,1) (v2) [label=right:$a$] {}
	edge [] (v1);
\node [noddee] at (-0.5,2) (v3) [label=right:$b$] {}
	edge [] (v2);
\node [noddee] at (-1.5,2.5) (v4)[label=above right:$c$] {}
	edge [] (v3);
\node [noddee] at (-2.5,2.5) (v5) [label=above:$d$] {}
	edge [] (v4);
\node [noddee] at (-3.5,3) (v6)  {}
	edge [] (v5);
\node [noddee] at (-3.5,2) (v7)  {}
	edge [] (v5);
\node [noddee] at (-2.5,0.5) (v8)  {}
	edge [] (v1)
	edge [] (v5);
\node  at (-1.5,-0.5) (v9)  {$G$};

\node [noddee] at (5,0) (t1)  {};
\node [noddee] at (5,1) (t2) [label=right:$a$] {}
	edge [] (t1);
\node [noddee] at (4.5,2) (t3) [label=right:$b$] {}
	edge [] (t2);
\node [noddee] at (3.5,2.5) (t4)[label=right:$c$] {}
	edge [] (t3);
\node [noddee] at (2.5,3) (t5) [label=above:$x_d$] {}
	edge [] (t1)
	edge [bend left] (t3)
	edge [] (t4);
\node [noddee] at (2.5,2) (t6) [label=below right:$y_d$] {}
	edge [] (t2)
	edge [] (t4);		
\node [noddee] at (1.5,3) (t7)  {}
	edge [] (t5)
	edge [] (t6);
\node [noddee] at (1.5,2) (t8)  {}
	edge [] (t5)
	edge [] (t6);
\node [noddee] at (2.5,0.5) (t9)  {}
	edge [] (t1)
	edge [bend left] (t5)
	edge [] (t6);

\node  at (3.5,-0.5) (t10)  {$\tp(G;d)$};

\node [noddee] at (10,0) (r1)  {};
\node [noddee] at (10,2) (r3) [label=right:$b$] {};
\node [noddee] at (7.5,3) (r5) [label=above:$x_d$] {}
	edge [] (r1)
	edge [bend left] (r3);
\node [noddee] at (7.5,2) (r6) [label=below right:$y_d$] {};		
\node [noddee] at (6.5,3) (r7)  {}
	edge [] (r5)
	edge [] (r6);
\node [noddee] at (6.5,2) (r8)  {}
	edge [] (r5)
	edge [] (r6);
\node [noddee] at (7.5,0.5) (r9)  {}
	edge [] (r1)
	edge [bend left] (r5)
	edge [] (r6);
	
\node [noddee] at (9,2) (r4)[label=below:$w_{ac}$] {}
	edge [] (r3)
	edge [] (r6);

\node  at (8.5,-0.5) (r10)  {$\tp(\tp(G;d);ac)$};

\node [noddee] at (3.5,-4.5) (k1)  {};

\node [noddee] at (1,-2) (k5) [label=above right:$x_d\equal d$] {}	
	edge [] (k1);		
\node [noddee] at (0,-2) (k7)  {}
	edge [] (k5);
\node [noddee] at (0,-3) (k8)  {}
	edge [] (k5);
\node [noddee] at (1,-4) (k9)  {}
	edge [] (k1)
	edge [bend left] (k5);

\node [noddee] at (3.5,-2) (k3) [label=right:$w_{by_d}$] {};	
\node [noddee] at (3.5,-3) (k4)[label=right:$w_{ac}$] {}
	edge [] (k3);

\node  at (2,-5) (k10)  {$\tp(\tp(\tp(G;d);ac);by_d)\cong \mathfrak{f}((G-\{b,c\});da)\cup K_2$};

			
\end{tikzpicture}
\end{center}
\caption{The second phase of expansions and contractions in Proposition~\ref{prop:deg2mate}.}
\label{fig:4-path}
\end{figure}

Our next aim is to prove the equality $\reg(G)=\vim(G)$ when $G$ is a well-covered block-cactus graph~\cite{RLV} that in turn includes any Cohen-Macaulay graph with girth at least five (see also \cite{BC,HMT}). We recall that a graph $G$ is called a \emph{block-cactus graph}, if each of its blocks is a clique or a cycle. For that purpose, we first introduce a graph class containing all such graphs. 

\begin{definition}
Let $G$ be a graph, and let $C$ be an induced cycle of length $4\leq n\leq 7$ in $G$. Then $C$ is called a \emph{basic cycle} of $G$ if one of the followings holds:
\begin{itemize}
\item[$(i)$] $n=4$ and contains two adjacent vertices of degree two in $G$,
\item[$(ii)$] $n=5$ and contains no two adjacent vertices of degree three or more in $G$,
\item[$(iii)$] $n=6$ or $7$ and if $x,y\in V(C)$ are two vertices such that
$\deg_G(x), \deg_G(y)\geq 3$, then $d_G(x,y)\geq 3$.
\end{itemize} 
We say that $G$ is in the class $\BW$, if its vertex set can be partitioned $V(G)=B\cup W$ such that $B$ consists of vertices of basic cycles of $G$ and basic cycles form a partition of $B$, and $W$ induces a weakly chordal graph in $G$.    
\end{definition}

\begin{theorem}\label{thm:reg=vim}
If $G\in \BW$, then $\reg(G)=\vim(G)$.
\end{theorem}
\begin{proof}
We proceed by an induction on the order of $G$. We first note that if $B=\emptyset$, then $\reg(G)=\vim(G)=\im(G)$ so that we may assume $B\neq \emptyset$. Let $L$ be the set of vertices of a basic cycle, say of length $n\geq 4$.
If every vertex in $L$ has degree two in $G$, we consider any $4$-path (not necessarily induced) on $\{a,b,c,d\}\subseteq L$.
Otherwise, $L$ contains at least one vertex, say $d$ of degree at least three. In such case, we consider a $4$-path
on $\{a,b,c,d\}$ in $L$ such that the adjacent vertices $b$ and $c$ are of degree two in $G$.  If $ad\in E(G)$,
that is, $n=4$, then $G-\{a,b,c,d\}\cup K_2$ is a mate of $G$ by Proposition~\ref{prop:deg2mate}.
Therefore, if we define $B':=B\setminus L$ and $W':=W\cup (L-\{a,b,c,d\})\cup V(K_2)$, then the
graph $G-\{a,b,c,d\}\cup K_2$ belongs to the class $\BW$ with the partition $V(G-\{a,b,c,d\}\cup K_2)=B'\cup W'$
so that $\reg(G)=\reg(G-\{a,b,c,d\})+1=\vim(G-\{a,b,c,d\})+1=\vim(G)$.

We may therefore assume that $ad\notin E(G)$, that is, $n\geq 5$. Note that the graph $\fp((G-\{b,c\});da)\cup K_2$ is a mate of $G$. We examine two cases separately:

{\it Case 1.} $n=5$ or $6$. If we define $B':=B\setminus L$ and $W':=W\cup (L-\{a,b,c\})\cup V(K_2)$,
then the decomposition $V(\fp((G-\{b,c\});da)\cup K_2)=B'\cup W'$ satisfies the required conditions for which
$\fp((G-\{b,c\});da)\cup K_2\in \BW$.

{\it Case 2.} $n=7$. In this case, we note that $L$ can contain at most one other vertex of degree three or more in $G$.
If there exists such a vertex, we may assume without loss of generality that this vertex is $a$.
Therefore, if we set $B':=B-\{a,b,c\}$ and $W':=W\cup V(K_2)$, the resulting decomposition $V(\fp((G-\{b,c\});da)\cup K_2)=B'\cup W'$ satisfies the required conditions for which $\fp((G-\{b,c\});da)\cup K_2\in \BW$.

In any case we conclude that $\reg(G)=\reg(\fp((G-\{b,c\});da)\cup K_2)$ so that the claim follows from the induction.
\end{proof}

The followings are direct consequence of Theorem~\ref{thm:reg=vim}.

\begin{corollary}
If $G$ is a well-covered block-cactus graph, then $\reg(G)=\vim(G)$.
\end{corollary}

\begin{corollary}
If $G$ is a Cohen-Macaulay graph of girth at least five, then $\reg(G)=\vim(G)$.
\end{corollary}
\section{Edge subdivisions and contractions on graphs}\label{sect:edge-subdiv}
A special case of Lozin transformation corresponds to an edge subdivision on a graph with respect to a chosen edge. We next provide a detail analyses on the effect of edge subdivisions of particular length to the regularity.
Moreover, we prove that the regularity of the graph obtained by a double edge subdivision equals to one more than the regularity of the graph resulting from the contraction of that edge.

Let $G=(V,E)$ be a graph, and let $e=xy\in E$ be an edge. An $n$-\emph{edge subdivision} on the edge $e$ for some $n\geq 1$ corresponds to replacing the edge $e$ by a path of length $n+1$, that is, we transform the graph $G$ to the graph $(G;e)_n$,
where $V((G;e)_n):=V\cup \{u_0,u_1,\ldots,u_{n-1}\}$ and $E((G;e)_n)=(E\setminus \{xy\})\cup \{xu_0,u_0u_1,\ldots, u_{n-1}y\}$. We only consider the cases where $n=2,3$. When $n=2,3$, we call the corresponding operations as the \emph{double} and \emph{triple} edge subdivisions, and denote the resulting graphs by $\DE(G;e)$ and $\LE(G;e)$ respectively. In particular, we denote by $\DE(G)$, the graph obtained from $G$ by applying double edge-subdivision to every edge of $G$. 

We remark that any triple edge subdivision on an edge $e=xy$ can be considered as a Lozin transformation by taking one of the sets in the partition of $N_G(x)$ as a singleton. In other words, we may consider the operation $\LE_x(G;\{y\},N_G(x)\backslash \{y\})$ for some $y\in N_G(x)$. Observe that the resulting graph can be obtained from $G$ by replacing the edge $xy$ in $G$ by a path $x-x_0-x_1-x_2-y$. Following Theorem~\ref{thm:lozin+reg}, we may readily describe the effect of a triple subdivision on regularity.

\begin{corollary}\label{cor:lozin+edgesubdiv}
Let $G$ be a graph and let $e=xy$ be an edge. Then $\reg(\LE(G;e))=\reg(G)+1$. 
\end{corollary}

As opposed to a triple edge subdivision, a double edge subdivision on a graph may not need to increase the regularity. For instance, if we consider $G=C_5$, then $\DE(C_5;e)\cong C_7$ so that $\reg(C_5)=\reg(\DE(C_5;e))=2$ for any edge $e\in E(C_5)$ (see also Remark~\ref{rem:double-reg}). On the other hand, it is not difficult to prove in general that an $n$-edge subdivision is a monotone increasing operation on the regularity of graphs for any $n\geq 2$. Similarly, such an operation is also monotone increasing with respect to the induced matching and cochordal cover numbers. It is less obvious that a triple edge subdivision is strictly increasing on the cochordal cover number.

\begin{lemma}\label{lem:triple-cd}
Let $G$ be a graph and let $e=xy$ be an edge. Then $\cd(G)+2\geq \cd(\LE(G;e))>\cd(G)$. 
\end{lemma}
\begin{proof}
Since the first inequality is obvious, we prove the second. Suppose that $\HE=\{H_1,\ldots,H_k\}$ is a cochordal cover of $\LE(G;e)$. Then there exist $i,j\in [k]$ with $i\neq j$ such that $xx_0\in E(H_i)$ and $x_2y\in E(H_j)$. 

{\it Case $1$.} There exists $t\in [k]\setminus \{i,j\}$ such that $x_0x_1\in E(H_t)$ (or similarly $x_1x_2\in E(H_t)$).
In such a case, we may assume that $x_1x_2\in E(H_t)$. If we define $E(H'_i):=(E(H_i)\setminus \{xx_0\})\cup \{xy\}$ and
$E(H'_j):=E(H_j)\setminus \{x_2y\}$, then $\{H'_i,H'_j\}\cup \{H_r\colon r\in [k]\setminus \{i,j,t\}\}$ is a cochordal cover of $G$.

{\it Case $2$.} $x_0x_1, x_1x_2\in E(H_i)$. If we define $E(H'_j):=(E(H_j)\setminus \{x_2y\})\cup \{xy\}$, then
$\{H'_j,H_r\colon r\in [k]\setminus \{i,j\}\}$ is a cochordal cover of $G$.

{\it Case $3$.} $x_0x_1\in E(H_i)$. By Cases $1$ and $2$, this forces $x_1x_2\in E(H_j)$. Now, if we define
$E(H_0):=(E(H_i)\setminus \{xx_0, x_0x_1\}\cup E(H_j)\setminus \{x_1x_2,x_2y\})\cup \{xy\}$, then
$\{H_0,H_r\colon r\in [k]\setminus \{i,j\}\}$ is a cochordal cover of $G$.
\end{proof} 

We note that the first inequality of Lemma~\ref{lem:triple-cd} could be tight. As an example, we state that
$\cd(\LE(C_4;e))=3>1=\cd(C_4)$ for any edge $e\in E(C_4)$.

\begin{proposition}\label{prop:double-all}
If $G=(V,E)$ is a graph with $E\neq \emptyset$, then $\im(\DE(G))=\reg(\DE(G))=\cd(\DE(G))=|E|$.
\end{proposition}
\begin{proof}
If we denote by $M$, the set of middle edges of each $4$-path corresponding to the double edge subdivision of each edge,
then $M$ is clearly an induced matching of cardinality $|E|$. On the other hand, for each such edge $f$, if we denote the subgraph consisting of the edge $f$ and any edge incident to $f$ by $H_f$, then the collection $\{H_f\colon f\in M\}$
is a cochordal covering of $\DE(G)$ from which the claim follows.  
\end{proof}

Our next result interrelates the effect of a double edge subdivision on the regularity to that of an edge contraction in a graph.

\begin{lemma}\label{lem:double-contract}
If $e=uv$ is an edge of a graph $G$, then $\reg(D(G;e))=\reg(G/e)+1$.
\end{lemma}
\begin{proof}
We begin by applying a $\tp$-expansion in the graph $D:=D(G;e)$ to the vertex $u$ with respect to the $\tp$-pairing 
$[\{v\},N_D(v)\setminus N_G(u)]$ having the $\tp$-neighbour the vertex $a$ (see Figure~\ref{cc}).
Observe that $\{v,a\}$ is a $\tp$-pair in the graph $\tp(D;u)$ with a $\tp$-neighbour $b$. Once we contract this
$\tp$-pair, in the resulting graph $\tp(\tp(D;u);va)$, we have that $\{c,y_u\}$ is a $\tp$-pair having the vertex $w_{va}$
as a $\tp$-neighbour, where $w_{va}$ is the vertex for which the pair $\{v,a\}$ is contracted. Finally,
the $\tp$-contraction of $\{b,y_u\}$ yields a graph isomorphic to $(G/e)\cup K_2$, where the isolated edge is induced by
the vertices $w_{va}$ and $w_{by_u}$. Therefore, the graph $D(G;e)$ is a mate of $(G/e)\cup K_2$, that is,
$\reg(D(G;e))=\reg(G/e)+1$.
\end{proof}

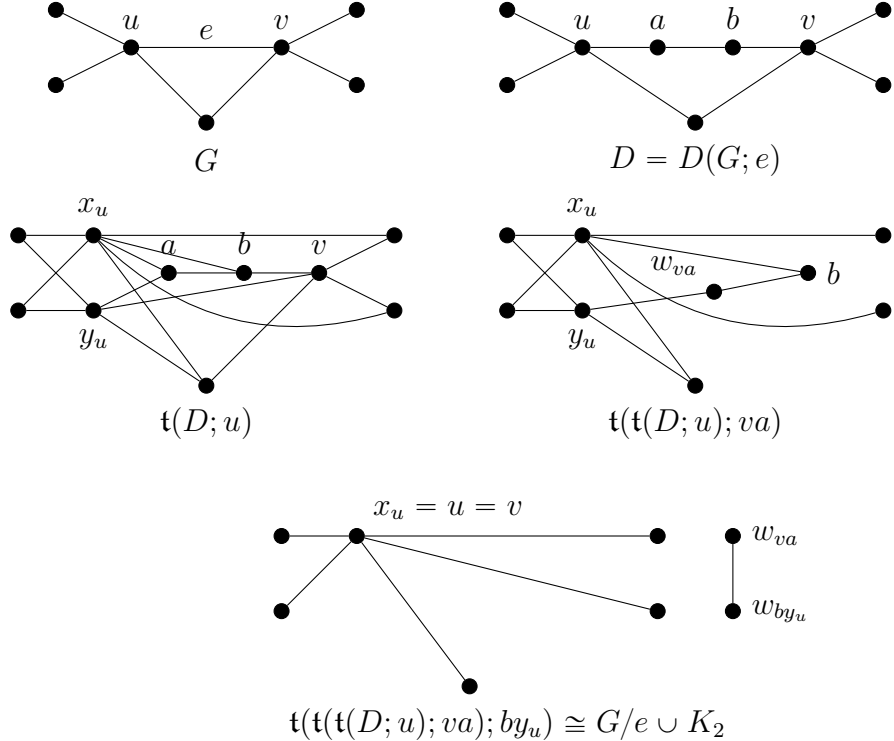
\begin{figure}[ht]
\begin{center}
\begin{tikzpicture}[scale=1]

\node [noddee] at (0,8) (v1)  {};
\node [noddee] at (-1,9) (v2) [label=above:$u$] {}
	edge [] (v1);
\node [noddee] at (1,9) (v3) [label=above:$v$] {}
	edge [] (v1)	
	edge [] (v2);
\node [noddee] at (-2,9.5) (v4) {}
	edge [] (v2);
\node [noddee] at (-2,8.5) (v5)  {}
	edge [] (v2);
\node [noddee] at (2,9.5) (v6)  {}
	edge [] (v3);
\node [noddee] at (2,8.5) (v7)  {}
	edge [] (v3);

\node  at (0,9.2) (v8)  {$e$};
\node  at (0,7.5) (v9)  {$G$};

\node [noddee] at (6.5,8) (t1)  {};
\node [noddee] at (5,9) (t2) [label=above:$u$] {}
	edge [] (t1);
\node [noddee] at (8,9) (t3) [label=above:$v$] {}
	edge [] (t1);
\node [noddee] at (4,9.5) (t4) {}
	edge [] (t2);
\node [noddee] at (4,8.5) (t5)  {}
	edge [] (t2);
\node [noddee] at (9,9.5) (t6)  {}
	edge [] (t3);
\node [noddee] at (9,8.5) (t7)  {}
	edge [] (t3);
\node [noddee] at (6,9) (t8) [label=above:$a$] {}
	edge [] (t2);
\node [noddee] at (7,9) (t9) [label=above:$b$] {}
	edge [] (t3)
	edge [] (t8);

\node  at (6.5,7.5) (t10)  {$D\equal D(G;e)$};

\node [noddee] at (0,4.5) (r1)  {};
\node [noddee] at (1.5,6) (r3) [label=above:$v$] {}
	edge [] (r1);
\node [noddee] at (-2.5,6.5) (r4) {};
\node [noddee] at (-2.5,5.5) (r5)  {};
\node [noddee] at (2.5,6.5) (r6)  {}
	edge [] (r3);
\node [noddee] at (2.5,5.5) (r7)  {}
	edge [] (r3);

\node [noddee] at (-0.5,6) (r8) [label=above:$a$] {};
\node [noddee] at (0.5,6) (r9) [label=above:$b$] {}
	edge [] (r3)
	edge [] (r8);

\node [noddee] at (-1.5,6.5) (r2) [label=above:$x_u$] {}
	edge [] (r1)
	edge [] (r4)
	edge [] (r5)
	edge [] (r6)
	edge [bend right] (r7)
	edge [] (r8)
	edge [] (r9);

\node [noddee] at (-1.5,5.5) (r2) [label=below:$y_u$] {}
	edge [] (r1)
	edge [] (r3)
	edge [] (r4)
	edge [] (r5)
	edge [] (r8);
\node  at (0,4) (r10)  {$\tp(D;u)$};

\node [noddee] at (6.5,4.5) (k1)  {};
\node [noddee] at (4,6.5) (k4) {};
\node [noddee] at (4,5.5) (k5)  {};
\node [noddee] at (9,6.5) (k6)  {};
\node [noddee] at (9,5.5) (k7)  {};
\node [noddee] at (8,6) (k9) [label=right:$b$] {};
\node [noddee] at (5,6.5) (k2) [label=above:$x_u$] {}
	edge [] (k1)
	edge [] (k4)
	edge [] (k5)
	edge [] (k6)
	edge [bend right] (k7)
	edge [] (k9);
\node [noddee] at (5,5.5) (k2) [label=below:$y_u$] {}
	edge [] (k1)
	edge [] (k4)
	edge [] (k5);
	
\node [noddee] at (6.75,5.75) (k3) [label=above left:$w_{va}$] {}
	edge [] (k2)
	edge [] (k9);
\node  at (6.5,4) (k10)  {$\tp(\tp(D;u);va)$};

\node [noddee] at (3.5,0.5) (p1)  {};
\node [noddee] at (1,2.5) (p4) {};
\node [noddee] at (1,1.5) (p5)  {};
\node [noddee] at (6,2.5) (p6)  {};
\node [noddee] at (6,1.5) (p7)  {};

\node [noddee] at (2,2.5) (p2) [label=above right:$x_u \equal u \equal v$] {}
	edge [] (p1)
	edge [] (p4)
	edge [] (p5)
	edge [] (p6)
	edge [] (p7);

\node [noddee] at (7,2.5) (p3) [label=right:$w_{va}$] {};

\node [noddee] at (7,1.5) (p2) [label=right:$w_{by_u}$] {}
	edge [] (p3);

\node  at (4,0) (p10)  {$\tp(\tp(\tp(D;u);va);by_u)\cong G \slash e \cup K_2$};

\end{tikzpicture}
\end{center}
\caption{The $\tp$-expansion and the following two $\tp$-contractions.}
\label{cc}
\end{figure}

\begin{lemma}\label{lem:double-reg}
If $e=uv$ is an edge of a graph $G$, then we have $\reg(G)\leq \reg(D(G;e))$.
\end{lemma}
\begin{proof}
We write $D:=D(G;e)$. Suppose first that $\reg(G)=n$ and let $R\subseteq V(G)$ be a subset such that 
$\widetilde{H}_{n-1}(G[R])\neq 0$. If $x\notin R$ (or $y\notin R$), then $D[R]\cong G[R]$ so that
$\reg(D)\geq n$. So, we may assume that $x,y\in R$. We now consider the Mayer-Vietoris exact sequence of the pair
$(G[R],x)$:
\begin{equation*}
\cdots \to \widetilde{H}_{n-1}(G[R]-x)\to \widetilde{H}_{n-1}(G[R])\to \widetilde{H}_{n-2}(G[R]-N_{G[R]}[x])\to \cdots. 
\end{equation*}
If $\widetilde{H}_{n-1}(G[R]-x)\neq 0$, then we have $\widetilde{H}_{n-1}(D[R]-x)\neq 0$, since $D[R]-x\cong G[R]-x$.
On the other hand, if $\widetilde{H}_{n-2}(G[R]-N_{G[R]}[x])\neq 0$, we define $S:=V(G[R]-N_{G[R]}[x])$ and
$S^*:=S\cup \{a,b\}$. It then follows that $D[S^*]\cong (G[R]-N_{G[R]}[x])\cup K_2$; hence,
$\widetilde{H}_{n-1}(D[S^*])\neq 0$, that is, $\reg(D)\geq n$.
\end{proof}

The following is the main result of this section.

\begin{theorem}\label{thm:contract-reg}
Let $e$ be an edge of a graph $G$. Then $\reg(G/e)\leq \reg(G)\leq \reg(G/e)+1$.
\end{theorem}

\begin{proof}
We note that the inequality $\reg(G)\leq \reg(G/e)+1$ follows from Lemmas~\ref{lem:double-contract} and~\ref{lem:double-reg}. For the first inequality, assume that $\reg(G/e)=m$ and let $T\subseteq V(G/e)$ be a subset such that
$\widetilde{H}_{m-1}((G/e)[T])\neq 0$. We denote by $x$ the vertex for which the edge $e$ is contracted, and
consider the the Mayer-Vietoris exact sequence of the pair $((G/e)[T],x)$:
\begin{equation*}
\cdots \to \widetilde{H}_{m-1}((G/e)[T]-x)\to \widetilde{H}_{m-1}((G/e)[T])\to \widetilde{H}_{m-2}((G/e)[T]-N_{(G/e)[T]}[x])\to \cdots. 
\end{equation*}

Now, if $\widetilde{H}_{m-1}((G/e)[T]-x)\neq 0$, then $\reg(G-\{u,v\})\geq m$ so that $\reg(G)\geq m$, since
$V((G/e)[T]-x)\subseteq V(G-\{u,v\})$ and $G-\{u,v\}$ is an induced subgraph of $G$. We may further assume that
$\widetilde{H}_{m-2}((G/e)[T]-N_{(G/e)[T]}[x])\neq 0$. If we define $L:=V((G/e)[T]-N_{(G/e)[T]}[x])$, the set $L\subset V(G)$ contains no neighbours of $u$ and $v$. Therefore, if we set $L^*:=L\cup \{u,v\}$, 
we conclude that $\widetilde{H}_{m-1}(G[L^*])\neq 0$, that is, $\reg(G)\geq m$.
\end{proof}

We recall that a graph $H$ is said to be a \emph{contraction-minor} of a  graph $G$ if it is obtained from $G$ by a series of edge contractions on $G$.
\begin{corollary}
If $H$ is a contraction-minor of $G$, then $\reg(H)\leq \reg(G)$.
\end{corollary}

\begin{remark}\label{rem:double-reg}
We note that the inequality $\reg(G)\leq \reg(D(G;e))\leq \reg(G)+1$ holds for any graph $G$ and its any edge $e$
as a result of Lemma~\ref{lem:double-contract} and Theorem~\ref{thm:contract-reg}.
\end{remark}

\section{Regularity of $2K_2$-free graphs}\label{sect:$2K_2$-free}
In this section we prove Theorem~\ref{thm:reg-im-pair}, and apart from that we provide a local analyse on the structure of $2K_2$-free graphs in general.

As we have already mention in Section~\ref{sect:intro}, the existence of
$2K_2$-free graphs with arbitrary large regularity can be guaranteed from the following result of Januszkiewicz and Swiatkowski:

\begin{theorem}\cite{JS}\label{thm:JS}
For any $n$ there exists a flag simplicial complex containing no empty square which is an oriented pseudomanifold of dimension $(n-1)$.
\end{theorem}

Some explanations should be in order. In fact, the existence of such complexes ensures the existence of a right angled Coxeter group $(W,S)$ of virtual cohomological dimension $\operatorname{vcd}(W)=n$ which does not contain $\Z\oplus \Z$ and thus is Gromov hyperbolic. Furthermore, the proof of Theorem~\ref{thm:JS} is inductive, and uses the notion of a \emph{simplex of groups}. We note that the flagness corresponds to the fact that the resulting simplicial complex is the independence complex of a graph, say $G_n$, and the no empty square property (sometimes also known as the \emph{flag-no-square property}) guarantees that the complement of $G_n$ is $C_4$-free, that is, $G_n$ is $2K_2$-free. Finally, we have $\reg(G_n)=n$, since the complex is an $(n-1)$-dimensional oriented pseudomanifold.

\begin{proof}[{\bf Proof of Theorem~\ref{thm:reg-im-pair}}]
Let $n$ and $k$ be two positive integers with $n\geq k\geq 1$. If $n=k$, the graph $nK_2$ is an example of a graph satisfying the required conditions. So, we may assume that $n>k$. On the other hand, we know from Theorem~\ref{thm:JS} that there exists a graph with $\reg(G(n,1))=n$ for any $n$ so that we may further assume that $k>1$. We let $m:=n-k+1$, and consider a $2K_2$-free graph $G(m,1)$ such that $\reg(G(m,1))=m$. Now, choose an edge $e$ of $G(m,1)$, and apply $k-1$ triple subdivisions on $e$ consecutively. If we denote the resulting graph by $G(n,k)$, then we have $\im(G(n,k))=k$ by Lemma~\ref{lem:lozin-inmatch} and $\reg(G(n,k))=n$ by Corollary~\ref{cor:lozin+edgesubdiv}.
\end{proof}

We should note that the graph $G(n,k)$ satisfying the required conditions in Theorem~\ref{thm:reg-im-pair} is not unique. When $k=1$, Osajda~\cite{DO} has recently introduced a different construction that produces new examples of such graphs, and in the general case, any distinct choice of a sequence of Lozin operations will produce a different example (see Section~\ref{sect:$2K_2$-free} for details.) We further remark that a detail look at the proof of Theorem~\ref{thm:JS} reveals that the graph $G_n$ is optimal in the sense that $\reg(G_n-x)=\reg(G_n-N_{G_n}[x])=n-1$ for any vertex $x\in V(G_n)$, that is, the graph $G_n$ is a perfect prime graph.

We note that leaving the fact that Theorem~\ref{thm:JS} is only an existence result, that is, a concrete description of such graphs seems quite hopeless, there exists another graph theoretical phenomenon hidden in its proof. We recall that a graph $G$ is said to be a \emph{locally homogeneous graph} if $G[N_G(x)]$ is isomorphic to a fixed graph $H$ for each $x\in V(G)$ (in such case, $G$ is sometimes called a \emph{locally} $H$-graph (see~\citep{GMW} for details)). For instance, the underlying graph of the Coxeter $600$-cell , that is, the graph $\overline{G}_{600}$ is locally icosahedral graph, and the underlying graph of the icosahedron is locally $C_5$. It is proved in~\cite{JS} that the links of vertices are isomorphic that in turn implies that the graph $\overline{G}_n$ is locally homogeneous. In terms of the regularity, this means that $\reg(G_n)=1+\reg(G_n-N_{G_n}[x])$ for any vertex $x\in V(G_n)$, and note that the graph $G_n-N_{G_n}[x]$ is again a $2K_2$-free graph and its independence complex is an oriented pseudomanifold. 

We also remark that Dao, Huneke and Schweig~\citep{DHS} have two logarithmic upper bounds on the regularity of $2K_2$-free graphs, the one is in terms of the maximum degree and the other involves the number of vertices. We refer to Theorems $4.1$ and $4.9$ in that paper, and note that a graph is $2K_2$-free if and only if its edge ideal is $1$-step linear (compare to Corollary $2.9$ in~\citep{DHS}).

We next state a graph theoretical characterization of $2K_2$-free graphs due to Chung et al.~\cite{CGTT}.

\begin{theorem}\cite{CGTT}\label{thm:chung}
If $G=(V,E)$ is a $2K_2$-free graph, then one of the followings hold.
\begin{itemize}
\item[(a)] If $G$ is bipartite, then both color classes of $G$ contain vertices adjacent to all vertices of the other color class.\\
\item[(b)] If $\omega(G)=2$ and $G$ is not bipartite, then $G$ can be obtained from the five-cycle by vertex multiplication.\\
\item[(c)] If $\omega(G)\geq 3$, then $G$ has a maximum dominating clique.
\end{itemize}
\end{theorem}
Note that for a vertex $x$ of a graph $G$ and a positive integer $n$, a graph $H$ is said to be obtained from $G$ by multiplying $x$ by $n$ whenever $H$ is formed by replacing the vertex $x$ by an independent set of $n$ vertices each having the same neighbours as $x$.
\begin{corollary}\label{cor:clique2}
If $G=(V,E)$ is a $2K_2$-free graph with $\omega(G)\leq 2$, then $\cd(G)\leq 2$.
\end{corollary}
\begin{proof}
If $G$ is bipartite with $V(G)=A\cup B$, then $A$ and $B$ form cliques in $\ol{G}$. There can be no induced $C_4$ in $\ol{G}$, otherwise $G$ can not be $2K_2$ -free. By the pigeonhole principle, for a possible induced $C_k$ for $k\geq 5$, at least three vertices of $C_k$ must be in the same part which forms a triangle; hence $G$ is cochordal, that is, $\cd(G)=1$.

By Theorem~\ref{thm:chung}, if $G$ is not bipartite and $\omega(G)=2$, then $G$ can be obtained by a vertex multiplication of a five-cycle. It means that there is an induced five-cycle $C$, say on vertices $\{v_1,\ldots, v_5\}$ (in cyclic fashion) such that $G$ is a vertex multiplication on $C$. We let $A_i$ be the independent sets such that $N_G(A_i)=N_G(v_i)$ for $1\leq i\leq 5$, and set $A^i:=A_i\cup \{v_i\}$. We define $H_1:=G[A^1\cup A^2\cup A^3\cup A^4]$ and
$H_2:=G[A^2\cup A^3\cup A^4\cup A^5]$, and claim that both $H_1$ and $H_2$ are cochordal subgraphs.
Firstly, $H_i$ is $2K_2$-free as $G$ is, and since $\omega(G)=2$, each $H_i$ is $\ol{C}_k$-free for any
$k\geq 6$. Furthermore, since each $A^j$ is also an independent set in $H_i$, these graphs are necessarily $C_5$-free. 

Finally, the equality $E(G)=E(H_1)\cup E(H_2)$ is straightforward so that $\cd(G)=2$.
\end{proof}

We remark that if $G$ is a $2K_2$-free graph and $e=xy\in E(G)$ is an edge of $G$, then $U=V(G)\backslash N_G[e]$ is an independent set, since otherwise
if two vertices $u, v\in U$ form an edge, then $G[\{x,y,u,v\}]\cong 2K_2$, a contradiction.

\begin{proposition}\label{prop:reg-privacy}
Let $G$ be a $2K_2$-free graph and let $e=xy\in E(G)$ be an edge of $G$. If $|N_G[x]\backslash N_G[y]|\leq k$, then $\reg(G-N_G[y])\leq \frac{k}{2}+1$.
\end{proposition}
\begin{proof}
By the above observation, the vertex set of the graph $G-N_G[y]$ can be splitted into two parts $N_G[x]\backslash N_G[y]$ and $U=V(G)\backslash N_G[e]$, where $U$ is an independent set. Let $M=\{x_1y_1,\ldots, x_ry_r\}$ be a maximum matching of $G[N_G[x]\backslash N_G[y]]$. We then consider split subgraphs $H_i=(K_i,I_i)$, where $K_i$ is the complete graph on $\{x_i,y_i\}$, i.e., a single edge $x_iy_i$, and $I_i$ is the independent set on the set of vertices
$N_G(x_i)\cup N_G(y_i)$ for each $1\leq i\leq r$. Furthermore, if we set $A:=N_G[x]\backslash N_G[y]-V(M)$, then $A$ is an independent set so that $G[U\cup A]$ is a $2K_2$-free bipartite graph; hence, it is a cochordal graph by Corollary~\ref{cor:clique2}. Therefore, we have $$\reg(G-N_G[y])\leq \cd(G-N_G[y])\leq \frac{k}{2}+1,$$
since $E(G-N_G[y])=E(H_1)\cup \ldots \cup E(H_r)\cup E(G[U\cup A])$ and $r\leq \frac{k}{2}$.
\end{proof}

\begin{definition}
For a given graph $G=(V,E)$, we define the \emph{maximum privacy degree} of $G$ by 
$\Gamma(G):=\max \{|N_G[x]\backslash N_G[y]|\colon x,y\in V\;\textrm{and}\;xy\in E\}$.
\end{definition}

\begin{theorem}\label{thm:main}
If $G$ is a $2K_2$-free graph, then $\reg(G)\leq \frac{1}{2}\Gamma(G)+2$.
\end{theorem}
\begin{proof}
We proceed by an induction on the order of $G$. Since the base case is trivial, we assume that the claim holds for any graph whose order is less than $|G|$. Suppose that $\Gamma(G)=|N_G[x]\backslash N_G[y]|$ for some edge $e=xy$ in $G$. We first note that $\Gamma(G-y)\leq \Gamma(G)$, since $|N_{G-y}[u]\backslash N_{G-y}[v]|\leq |N_G[u]\backslash N_G[v]|$ for any $uv\in E(G-y)$. Therefore, it follows from the induction that $\reg(G-y)\leq \frac{1}{2}\Gamma(G-y)+2\leq \frac{1}{2}\Gamma(G)+2$. On the other hand, we have $\reg(G-N_G[y])\leq \frac{1}{2}\Gamma(G)+1$ by Proposition~\ref{prop:reg-privacy}; hence,
$$\reg(G)\leq \max \{\reg(G-y),\reg(G-N_G[y])+1\}\leq \frac{1}{2}\Gamma(G)+2$$
by Corollary~\ref{cor:induction-sc}.
\end{proof}

\begin{proposition}\label{prop:mindeg}
Let $G$ be a $2K_2$-free graph. If $x$ is a vertex such that $d_G(x)<2\reg(G)-3$, then $\reg(G)=\reg(G-N_G[x])$.
\end{proposition}
\begin{proof}
We let $d_G(x)=d$. If $y\in N_G(x)$, then $|N_G[x]\backslash N_G[y]|\leq d-1$. Thus, the inequality $\reg(G-N_G[y])\leq\frac{d+1}{2}<\reg(G)-1$ holds by Proposition~\ref{prop:reg-privacy}. However, this implies that $\reg(G)=\reg(G-y)$ by Corollary~\ref{cor:induction-sc}. We may therefore remove any neighbourhood of $x$ without altering the regularity in which case $x$ becomes an isolated vertex, that is, $\reg(G)=\reg(G-N_G[x])$ as claimed.
\end{proof}

In fact Proposition~\ref{prop:reg-privacy} provides a  strict (local) restriction on the minimum degree of $2K_2$-free graphs with a given regularity.
\begin{corollary}\label{cor:mindeg}
If $G$ is a $2K_2$-free graph such that $\reg(G)=k$, then there exists an induced subgraph $H$ of $G$ such that $\reg(G)=\reg(H)$ and $\delta(H)\geq 2k-3$.
\end{corollary}
\begin{proof}
Suppose that $\delta(G)<2k-3$, and let $x$ be a vertex of minimum degree in $G$. We have $\reg(G)=\reg(G-N_G[x])$ by Proposition~\ref{prop:mindeg}, since $d_G(x)<2k-3$. If $\delta(G-N_G[x])\geq 2k-3$, we are done. Otherwise, we continue in a similar way, and since $G$ is finite, this process will eventually terminate.
\end{proof}
In terms of prime graphs, we may restate Corollary~\ref{cor:mindeg} as follows.
\begin{corollary}\label{cor:prime-2k2}
If $G$ is a prime $2K_2$-free graph, then $\reg(G)\leq \frac{\delta(G)+3}{2}$.
\end{corollary}

\begin{definition}
Let $G$ be a graph. A degree-one vertex of $G$ is called a \emph{pendant vertex}, and if $x$ is a pendant vertex and $N_G(x)=\{y\}$, then the vertex $y$ is said to be the \emph{support vertex} of $x$ in $G$. We denote by $P_G$ and $SP_G$, the set of pendant and their support vertices in $G$ respectively.
\end{definition}

\begin{proposition}\label{prop:remove}
Let $G$ be a $2K_2$-free graph such that $G$ is not a cochordal graph. Then, $\reg(G)=\reg(G-(P_G\cup SP_G))$.
\end{proposition}
\begin{proof}
Since $G$ is not a cochordal graph, we have $\reg(G)\geq 2$. Assume that $x\in P_G$ and $N_G(x)=\{y\}$. Since $2K_2$-free, the graph $G-N_G[y]$ is an edgeless graph so that $\reg(G-N_G[y])=0$. However, this implies that $\reg(G)=\reg(G-y)$ 
by Corollary~\ref{cor:induction-sc}. On the other hand, any pendant vertex whose support is $y$ in $G$ is an isolated vertex of $G-y$, i.e., $\reg(G-y)=\reg((G-y)\cup SP_G(y))=\reg(G-(\{y\}\cup SP_G(y)))$, where $SP_G(y)=\{x\in P_G\colon N_G(x)=\{y\}\}$. We then apply to an induction on $|SP_G|$ to conclude the claim.  
\end{proof}

\begin{corollary}\label{cor:pendant-support}
If $G$ is a $2K_2$-free graph such that $G$ is not cochordal, then $\reg(G)\leq \frac{1}{2}\Gamma(G-(P_G\cup SP_G))+2$.
\end{corollary}

\section{Conclusion}\label{sect:concl}
In this section, we offer a short discussion on the results of this paper and possible new directions to lead from here.

Regarding the results of Section~\ref{sect:primes}, the most prominent question would be the characterization of connected bipartite (perfect) prime graphs of maximum degree at most three. 
The investigation of the structural properties that they need to carry would be valuable. Furthermore, it would be interesting to determine whether there exists a small constant $c>0$ such that $\reg(B)\leq c\im(B)$ holds for any such graph, while our primarily search (compare to Theorem~\ref{thm:sharp-bip}) hints that the right value of such a constant $c$ would be  $\frac{3}{2}$.
On the other hand, in the view of Theorem~\ref{thm:complex-reg-dec} (or Corollary~\ref{cor:reg-gim}), a possible characterization of graphs whose prime subgraphs can be completely determined would be highly interesting.
\begin{problem}
Is it possible to characterize graphs $G$ for which only induced prime subgraph that G can contain is isomorphic to either $K_2$ or a cycle $C_{3k+2}$ for some $k\geq 1$?
\end{problem}
Observe that the regularity of any such graph must satisfy $\reg(G)\leq 2\im(G)$. In particular, we do not know whether claw-free graphs can contain primes other than a $K_2$, a cycle or the complement of a cycle.

\begin{problem}
Is it true that $\reg(G)=\im(G;\{K_2,C_5\};{\bf(1,2)})$ if $G$ is a $(C_3,P_6)$-free graph?
\end{problem}

For a complete description of the role of the virtual induced matching number on the computation of graph's regularity, a further study is definitely needed. 

\begin{problem}
Does there exist a graph $G$ whose regularity is independent of the characteristic of the coefficient field such that $\vim(G)<\reg(G)$?
\end{problem}

There could be two directions to lead on the study of virtual induced matching number. In one way, there is the characterization of graphs (or graph classes) for which the equality $\reg(G)=\vim(G)$ holds.

\begin{problem}
Is it true that $\reg(G)=\vim(G)$ for any vertex-decomposable graph $G$?
\end{problem}

We remark that if $G$ is a graph for which it has a mate $H$ satisfying $\reg(H)=\im(H)$, then the equality $\reg(G)=\vim(G)$ must hold for $G$. For instance, the regularity of any graph having a chordal graph as a mate equals to its virtual induced matching number.

\begin{problem}
Characterize graphs that have the graph $nK_2$ as a mate for some $n\geq 1$? 
\end{problem}

On the other hand, since a $\tp$-contraction or $\tp$-expansion does not need to preserve the bipartiteness of a graph, it would be valuable to find such closed operations on the class of bipartite graphs under which the regularity is stable.

\section*{Acknowledgments}
We would like to thank Michal Adamaszek for bringing the work of Przytycki and Swiatkowski~\cite{PS} to our attention, Damian Osajda for explaining to us the details of his Basic Construction~\cite{DO}, and Zakir Deniz for his help with the Sage codes.



\begin{thebibliography}{00}

\bibitem{MA} M.~Adamaszek,
\newblock{\em Splittings of independence complexes and the powers of cycles,}
\newblock{J. Combin. Theory, Series A, 119, (2012), 1031-1047.}


\bibitem{ALS} D.~Attali, A.~Lieutier and D.~Salinas,
\newblock{\em Efficient data structure for representing and simplifying simplicial complexes in high dimensions,}
\newblock{Int. J. Compt. Geometry, 22, (2012), 279-303.}




\bibitem{BC1} T.~B\i y\i ko$\breve{g}$lu and Y.~Civan,
\newblock{\em Four-cycled graphs with topological applications,}
\newblock{Annals Combin., 16, (2012), 37-56.} 

\bibitem{BCU} T.~B\i y\i ko$\breve{g}$lu and Y.~Civan,
\newblock{\em Bounding Castelnuovo-Mumford regularity of graphs via Lozin's operations,}
\newblock{unpublished manuscript, available at arXiv:1302.3064, (2013), 12pp.}


\bibitem{BC} T.~B\i y\i ko$\breve{g}$lu and Y.~Civan,
\newblock{\em Vertex-decomposable graphs, codismantlability, Cohen-Macaulayness, and Castelnuovo-Mumford regularity,}
\newblock{Electronic J. Combin., 21:1, (2014), \#P1, 1-17.}

\bibitem{BM} J.A.~Bondy and U.S.R.~Murty,
\newblock{\em Graph Theory,}
\newblock{GTM $244$, Springer, (2008).}

\bibitem{BLS} A.~Brandst{\"a}dt, V.B.~Le and J.P.~Sprinrad.
\newblock{\em Graph Classes: A Survey}. 
\newblock{SIAM Monographs on Disc. Math. and Appl., Philadelphia, (1999).}



\bibitem{CGTT} F.R.K.~Chung, A.~Gyarfas, Z.~Tuza and W.T.~Trotter,
\newblock{\em The maximum number of edges in $2K_2$-free graphs of bounded degree,}
\newblock{Discrete Math., 81, (1990), 129-135.}

\bibitem{DHS} H.~Dao, C.~Huneke and J.~Schweig,
\newblock{\em Bounds on the regularity and projective dimension of ideals associated to graphs,}
\newblock{J. Algebraic Combin., 38:1, (2013), 37-55.}


\bibitem{EH} R.~Ehrenborg and G.~Hetyei,
\newblock{\em The topology of the independence complex,}
\newblock{European J. of Combinatorics, 27, (2006), 906-923.}


\bibitem{AE} A.~Engstr\"om,
\newblock{\em Complexes of directed trees and independence complexes,}
\newblock{Discrete Math., 309, (2009), 3299-3309.}

\bibitem{HVT} H.T.~H\'a and A.~Van Tuyl,
\newblock{\em Monomial ideals, edge ideals of hypergraphs, and their graded Betti numbers,}
\newblock{J. Alg. Combin., 27, (2008), 215-245.}



\bibitem{Ha} H.T.~H\'a,
\newblock{\em Regularity of squarefree monomial ideals,}
\newblock{Connections Between Algebra, Combinatorics, and Geometry,
Springer Proc. in Math. $\&$ Stat., 76, (2014), 251-276.}

\bibitem{HW} H.T.~H\'a and R.~Woodroofe,
\newblock{\em Results on the regularity of squarefree monomial ideals,}
\newblock{Adv. Appl. Math, 58, (2014), 21-36.}

\bibitem{HHKT} T.~Hibi, H.~Higashitani, K.~Kimura, and A.~Tsuchiya,
\newblock{\em Dominating induced matchings of finite graphs and regularity of edge ideals,}
\newblock{preprint, available at arXiv:1412.3881v1, (2014), 23pp.}


\bibitem{HMT} H.T.~Hoang, N.C.~Minh and T.N.~Trung.
\newblock{On the Cohen-Macaulay graphs and girth,}
\newblock{preprint, available at arXiv:1204.5561v2, (2013), 16pp.}

\bibitem{HH} H.~Hoppe,
\newblock{\em Progressive meshes,}
\newblock{Proc. 23rd Annual Conf. Computer Graphics and Interactive Techniques, (1996), 99-108.}

\bibitem{JS} T.~Januszkiewicz and J.~Swiatkowski,
\newblock{\em Hyperbolic Coxeter groups of large dimension,}
\newblock{Comment. Math. Helv., 78, (2003), 555-583.}

\bibitem{MK} M.~Katzman,
\newblock{\em Characteristic-independence of Betti numbers of graph ideals,}
\newblock{J. Combin. Theory, Ser. A, 113, (2006), 435-454.}

\bibitem{KR} D.~Kobler and U.~Rotics.
\newblock{\em Finding maximum induced matchings in subclasses
of claw-free and $P_5$-free graphs, and in graphs with matching and induced matching of equal maximum size.}
\newblock {Algorithmica, 37, (2003), 327-346.}

\bibitem{DK} D. Kozlov,
\newblock{\em Combinatorial algebraic topology,}
\newblock{ACM $21$, Springer, (2008).}


\bibitem{MMCRTY} M.~Mahmoudi, A.~Mousivand, M.~Crupi, G.~Rinaldo, N.~Terai and S.~Yassemi,
\newblock{\em Vertex decomposability and regularity of very well-covered graphs,}
\newblock{J. Pure and Appl. Alg., 215, (2011), 2473-2480.}

\bibitem{MT} M.~Marietti and D.~Testa,
\newblock{\em A uniform approach to complexes arising from forests,}
\newblock{Elect. J. Comb., {\bf 15}, (2008), \#R101.}

\bibitem{MP} D.~Marusic and T.~Pisanski,
\newblock{\em The remarkable generalized Petersen graph $G(8,3)$,}
\newblock{Math. Slovaca, 50, (2000), 117-121.}

\bibitem{MV} S.~Morey and R.H.~Villarreal,
\newblock{\em Edge ideals: algebraic and combinatorial properties,}
\newblock{In Progress in Commutative Algebra, Combinatorics and Homology, Vol.1, De Gruyter, Berlin, (2012), 85-126.}

\bibitem{JM} J.R.~Munkres,
\newblock{\em Elements of algebraic topology,}
\newblock{Addison-Wesley Pub., (1993).}


\bibitem{EN} E.~Nevo,
\newblock{\em Regularity of edge ideals of $C_4$-free graphs via the topology of the lcm-lattice,}
\newblock{J. Comb. Theory, Ser. A 118, (2011), 491-501.}


\bibitem{EP} E.~Nevo and I.~Peeva,
\newblock{\em $C_4$-free edge ideals,}
\newblock{J. Algebraic Combin., 37, (2013), 243-248.}

\bibitem{VVL} V.V.~Lozin,
\newblock{\em On maximum induced matchings in bipartite graphs,}
\newblock{Information Processing Letters, 81, (2002), 7-11.}


\bibitem{DO} D.~Osajda,
\newblock{\em A construction of hyperbolic Coxeter groups,}
\newblock{Comment. Math. Helv., 88, (2013), 353-367.}

\bibitem{IP} I.~Peeva,
\newblock{\em Graded Syzygies,}
\newblock{Algebra and Applications, Vol. 14, Springer-Verlag, (2011).}


\bibitem{PS} P.~Przytycki and J.~Swiatkowski,
\newblock{\em Flag-no-square triangulations and Gromov boundaries in dimension $3$,}
\newblock{Groups Geom. Dyn., 3, (2009), 453-468.}

\bibitem{RLV} B.~Randerath and L.~Volkmann,
\newblock{\em A characterization of well-covered block cactus graphs,}
\newblock{Australasian J. Combin., 9, (1994), 307-314.}


\bibitem{SAGE} W. A.~Stein et al.,
\newblock{\em Sage Mathematics Software,}
\newblock{The Sage Development Team, (2014), http://www.sagemath.org.}


\bibitem{VT} A.~Van Tuyl,
\newblock{\em Sequentially Cohen-Macaulay bipartite graphs: vertex decomposability and regularity,}
\newblock{Arch. Math., 93, (2009), 451-459.}

\bibitem{HP} P.~van't Hof and D.~Paulusma,
\newblock{\em A new characterization of $P_6$-free graphs,}
\newblock{Discr. Appl. Math., 158, (2010), 731-740.}


\bibitem{GW}G.~Whieldon,
\newblock{\em Jump sequences of ideals,}
\newblock{preprint, available at arXiv:1012.0108, (2010), 27pp.}

\bibitem{GMW} G.M.~Weetman,
\newblock{\em A construction of locally homogeneous graphs,}
\newblock{J. London Math. Soc. 50:2, (1994), 66-86.}

\bibitem{RW3} R.~Woodroofe,
\newblock{\em Chordal and sequentially Cohen-Macaulay clutters,}
\newblock{Electronic J. Combin., 18, (2011), \#P208, 20 pp.}

\bibitem{RW2} R.~Woodroofe,
\newblock{\em Matchings, coverings, and Castelnuovo-Mumford regularity,}
\newblock{J. Commut. Algebra, 6, (2014), 287-304.}














\end{thebibliography}
\end{document}